\documentclass[11pt]{amsart}
\usepackage{amssymb}
\DeclareMathAlphabet{\mathantt}{OT1}{antt}{li}{it}
\DeclareMathAlphabet{\mathpzc}{OT1}{pzc}{m}{it}

\usepackage{tikz}
\usepackage{youngtab}
\usepackage[colorlinks=true,linkcolor=blue,urlcolor=blue,citecolor=blue, pagebackref]{hyperref}
\usepackage[alphabetic]{amsrefs}
\usepackage[mathscr]{euscript}
\usepackage{soul}
\usepackage[margin=1.0in]{geometry}
\usepackage{xypic, verbatim,amscd,color}

\newcommand\mk{\mathcal{K}}

\newcommand\sas{s_a}
\newcommand\sbs{s_b}
\newcommand\ml{\mathcal{L}}
\newcommand\mg{\mathcal{G}}

\newcommand\mf{\mathcal{F}}

\newcommand{\mc}[1]{\mathcal{#1}}
\newcommand{\leto}[1]{\stackrel{#1}{\rightarrow}}
\newcommand{\ov}[1]{\overline{#1}}
\newcommand{\op}[1]{\operatorname{#1}}
\newcommand{\ovop}[1]{\ov{\op{#1}}}

\newcommand\xrminusone{x_{1}}
\newcommand\xzero{s}
\newcommand\xone{a}
\newcommand\xtwo{b}
\newcommand\Cee{\vartheta}

\newcommand{\ovmc}[1]{\ov{\mc{#1}}}

\newcommand\tensor{\otimes}

\newcommand{\sL}{\mathfrak{sl}}

\newcommand\me{\mathcal{E}}
\newtheorem{theorem}{Theorem}[section]
\newtheorem*{theorem*}{Theorem}
\newtheorem{corollary*}{Corollary}
\newtheorem{corollary}[theorem]{Corollary}
\newtheorem{remark}[theorem]{Remark}
\newtheorem{question}[theorem]{Question}
\newtheorem{bigquestion}{Question}

\newtheorem{project*}{Project}

\newtheorem{construction}{Construction}

\newtheorem{mainconjecture*}[theorem]{Main Conjecture}

\newtheorem*{lemma*}{Lemma}
\newtheorem{lemma}[theorem]{Lemma}
\newtheorem{claim}[theorem]{Claim}

\newtheorem*{claim*}{Claim}
\newtheorem*{approach*}{Approach}

\newtheorem*{example*}{Example}

\newtheorem{Def/Prop}{Definition/Proposition}
\newtheorem{remark/definition}{Remark/Definition}

\newtheorem*{proof*}{proof}
\newtheorem*{proofsketch*}{Sketch of proof:}

\newtheorem*{prop*}{Proposition}
\newtheorem{proposition}[theorem]{Proposition}

\newtheorem{defn/thm}[theorem]{Definition/Theorem}
\newtheorem{definition/lemma}[theorem]{Definition/Lemma}
\newtheorem{definition}[theorem]{Definition}
\newtheorem{example/lemma}[theorem]{Example/Lemma}
\newcommand{\git}{\ensuremath{\operatorname{/\!\!/}}}

\newcommand\rk{\operatorname{rk}}

\begin{document}

\pagenumbering{arabic}

\title[P. Belkale, and A. Gibney]{On finite generation of the section ring of the determinant of cohomology line bundle}
\author{P. Belkale, and A. Gibney}
\date{\today}
\maketitle

\begin{abstract}
For $C$ a stable curve of arithmetic genus $g\ge 2$,   and $\mc{D}$ the determinant of cohomology
 line bundle on $\op{Bun}_{\op{SL}(r)}(C)$, we show the section ring for the pair $(\op{Bun}_{\op{SL}(r)}(C), \mc{D})$ is finitely generated.  Three applications are given.
  \end{abstract}

\section{Introduction}
For $\op{G}$ a simple, simply connected complex linear algebraic group, and $C$ a stable curve of arithmetic genus $g\ge 2$, let ${\op{Bun}}_{{\op{G}}}(C)$ be the stack parameterizing principal ${\op{G}}$-bundles on $C$.  To any representation $\op{G}\to \op{Gl}(V)$, there corresponds a distinguished line bundle $\mc{D}(V)$ on $\op{Bun}_{\op{G}}(C)$,  the determinant of cohomology
 line bundle, described in Def \ref{DCLB}.

 The main result of this work is the following
 \begin{theorem}\label{qmain} For $\op{G}=\op{SL}(r)$, for the standard representation $\op{SL}(r)\to \op{Gl}(V)$, setting $\mc{D}=\mc{D}(V)$,
 $$\mathscr{A}^{C}_{\bullet}=\bigoplus_{m\in \mathbb{Z}_{\ge 0}} \op{H}^0(\op{Bun}_{\op{SL}(r)}(C), \mc{D}^{\tensor m})$$
is finitely generated.
\end{theorem}
Theorem \ref{qmain} may seem surprising because the stack $\op{Bun}_{\op{SL}(r)}(C)$ is not proper, and hence there is no expectation that $\op{H}^0(\op{Bun}_{\op{SL}(r)}(C), \mc{L}^{\tensor m})$ is finite dimensional for any  line bundle $\mc{L}$.  Even if the summands were finite dimensional, as in various  examples of line bundles on projective varieties, finite generation of the section ring would not necessarily follow.

In fact, Theorem \ref{qmain} is well known for smooth curves $C$, and in this case:
$$\op{H}^0(\op{Bun}_{\op{SL}(r)}(C), \mc{D}^{\tensor m}) \cong  \op{H}^0(\op{SU}_{C}(r), \theta^{\tensor m}),$$
where $\op{SU}_{C}(r)$ is a moduli space parameterizing semistable vector bundles of rank $r$ with trivial determinant on $C$, and $\theta$ is an ample line bundle on it \cite{BeauvilleLaszlo,Faltings}.
These form a flat family over  $\mc{M}_g$, and it is  natural to ask if one can extend it to  a family over Deligne and Mumford's compactification $\ovmc{M}_{g}$ (This problem is discussed further in Section \ref{historie}).

As a first application of Theorem \ref{qmain}, we show

\begin{theorem}\label{globaleJJs} There is a flat family $p: \mathcal{X} \to \ovmc{M}_g$, with $\mathcal{X}$ relatively projective over $\ovmc{M}_g$,
such that
\begin{enumerate}
\item $\mathcal{X}_{C}\cong \op{Proj}(\mathscr{A}^{C}_{\bullet})$ is  integral, normal, and irreducible, for   $[C] \in \ovmc{M}_g$; and
\item $\mathcal{X}_{C} \cong \op{SU}_{C}(r)$, for $[C] \in \mc{M}_g$.
\end{enumerate}
\end{theorem}

The main idea in our proof of Theorem \ref{qmain} is to factor each vector space $\op{H}^0(\op{Bun}_{\op{SL}(r)}(C), \mc{D}^{\tensor m})$, writing it as a direct sum.  We then show that each summand of the factorization corresponds to the global sections of a line bundle
 on a suitably defined projective variety constructed using torsion free sheaves (see Proposition \ref{SomethingIntro}).  This is described below in Section \ref{qmainDescribe}.

The theory of conformal blocks plays a crucial role in the proof of Theorem \ref{globaleJJs}: While individual fibers may be composed without them, vector bundles of conformal blocks are used in our construction of the flat family.  The proof of Theorem \ref{globaleJJs}, and further applications of  Theorem \ref{qmain}, all having to do with conformal blocks, are described below in Section \ref{Further}.

\subsection{Applications of Theorem \ref{qmain}}\label{Further}
Vector bundles of conformal blocks are intrinsic objects which relate questions about the moduli stack $\ovmc{M}_{g,n}$ on which they are defined, to problems in representation theory; information on one side tells us something about the other.

We write $\mathbb{V}(\mathfrak{g},\vec{\lambda},\ell)$ to denote the bundle given by the Lie algebra  $\mathfrak{g}$ for simply connected group $\op{G}$, the
integer $\ell$, and the n-tuple $\vec{\lambda}=(\lambda_1,\ldots, \lambda_n)$ of dominant integral weights for $\mathfrak{g}$ at level $\ell$ (See Section \ref{VS} for a definition).  We show fibers may be interpreted geometrically in terms of sections of a line bundle on the stack  parameterizing generalized parabolic $\op{G}$-bundles on $C$:

\begin{theorem}\label{one}$\mathbb{V}(\mathfrak{g},\vec{\lambda},\ell)|_{(C;\vec{p})}^* \cong \op{H}^0(\op{Parbun}_{\op{G}}(C,\vec{p}),\mc{L}_{\op{G}}(C,\vec{p}, \vec{\lambda})).$
\end{theorem}

The line bundle $\mc{L}_{\op{G}}(C,\vec{p},\vec{\lambda})$ on $\op{Parbun}_{\op{G}}(C,\vec{p})$ is  analogous to $\mc{D}$ on $\op{Bun}_{\op{G}}(C,\vec{p})$. Theorem \ref{one}  was proved for smooth curves $C$  \cites{BeauvilleLaszlo,  LaszloSorger}, and
 for families of singular stable curves in \cite{BF}, using a different method. Since our original proof of Theorem \ref{one}    for a single  stable curve, using factorization, is needed in one of our proofs of Theorem \ref{qmain}, we include it here (see Section \ref{TheoremOneProof}).

Vector bundles of conformal blocks give rise to graded sheaves of $\mc{O}_{\ovmc{M}_{g,n}}$ algebras \cite[p.~368]{Faltings}, \cite{Manon}.  For example, $\mathscr{A}_{\bullet}=\bigoplus_{m\in \mathbb{Z}_{\ge 0}} \mathbb{V}(\sL_{r+1}, m)^*$, is the algebra of conformal blocks in type $\op{A}$.  To prove Theorem \ref{globaleJJs}, we show $\mathscr{A}_{\bullet}$ is finitely generated, and set $\mc{X}=\op{Proj}(\mathscr{A}_{\bullet})$.  Taking fibers commutes with  $\op{Proj}$, so $\mc{X}_{[C]}=\op{Proj}(\mathscr{A}^{[C]}_{\bullet})$. Flatness and assertion (1) follow from properties  of conformal blocks. For $C$ is smooth, there is a geometric interpretation for $\mathbb{V}(\sL_r,\ell)|_{C}^*$ on $\ovmc{M}_g$, as global sections of $\theta$ on $\op{SU}_{C}(r)$, giving (2).

In our second application of Theorem  \ref{qmain}, we show there are geometric interpretations for  fibers $\mathbb{V}(\sL_r,\ell)|_{C}^*$ at singular stable curves $C \in \ovmc{M}_g$, as global sections of a line bundle on a projective variety, under certain assumptions:

 \begin{theorem}\label{A2}Given  $[C]\in \ovmc{M}_g$, and a positive integer $r$, there exists a projective polarized pair $(\mc{X}_{C}(r,\ell), \mc{L}_{C}(r,\ell))$, and a positive integer $\ell$ such that
 \begin{equation}\label{BGKEq}
 \bigoplus_{m\in \mathbb{Z}_{\ge 0}}\mathbb{V}(\sL_r,m \ell)|_{[C]}^* \cong \bigoplus_{m\in \mathbb{Z}_{\ge 0}}\op{H}^0(\mc{X}_{C}(r,\ell), \mc{L}_{C}(r,\ell)^{\tensor m}).
 \end{equation}
We can be more precise about $\ell$ in some cases:
\begin{enumerate}
\item  For general $r$ if $C$ has only nonseparating nodes,  $\ell\ge 1$;
\item  For $r=2$,  $\ell$ divisible by $2$;
\item For general $r$, and $C$ with separating nodes, we know  such an $\ell$ exists.
\end{enumerate}
 \end{theorem}
We note that by  \cite[Theorem 1.1]{BGK}, for $\mathbb{V}(\sL_2,1)$ on $\ovmc{M}_{2}$ there are points $[C]\in \Delta_1$ for which there is no projective polarized pair $(\mc{X}_C, \mc{L}_C)$ such that Eq \ref{BGKEq} holds with $\ell=1$.  Theorem \ref{A2}  is proved assuming Theorems \ref{qmain} and \ref{one} and with basic properties of $\op{Proj}$.

As  a third application of Theorem  \ref{qmain}, in Proposition \ref{C} we show that
the Chern character of $\mathbb{V}(\sL_r,\ell)$ on $\ovmc{M}_g$ is quasi-polynomial in $\ell$ for sufficiently large $\ell$.

\subsection{Outline of the proof of Theorem \ref{qmain}}\label{qmainDescribe} We use varieties  $\mc{X}(\vec{a})$, described in Def \ref{XA}, which compactify isomorphism classes of $\vec{a}$-semistable vector bundles of rank $r$ on $C$ with trivializable determinants (Def \ref{aStability}).
In Proposition \ref{OneIntro}, we show
for certain vector bundles $\mc{G}$ on $C$ of rank $m$, there are line bundles $\mc{L}_{\mc{G}}$ on $\mc{X}(\vec{a})$ (Def \ref{LBG}), such that
$\op{H}^0(\mc{X}(\vec{a}), \mc{L}_{\mc{G}})\hookrightarrow \op{H}^0(\op{Bun}_{\op{SL}(r)}(C), \mc{D}^{\tensor m})$.  These inclusions give way to a map
$$F: \bigoplus_{\{\vec{a},\mc{G}\}}\op{H}^0(\mc{X}(\vec{a}), \mc{L}_{\mc{G}})\to \bigoplus_{m\in \mathbb{Z}_{\ge 0}} \op{H}^0(\op{Bun}_{\op{SL}(r)}(C), \mc{D}^{\tensor m}).$$  We show  $F$ is surjective, and  the part  of the sum needed to prove surjectivity, is finitely generated.

To prove surjectivity, we first show  $\op{H}^0(\op{Bun}_{\op{SL}(r)}(C), \mc{D}^{\tensor m})$ can be expressed as a direct sum  of factors from the moduli stack  of principal bundles on the normalization of $C$.  Two proofs of this are given: one using conformal blocks, and a second one which is independent of conformal blocks (see Lemma \ref{Lemma3}(3)). Second, we  argue using pole calculations, that  components of the sum are contained in the image of particular such embeddings.  Here we use the crucial fact that that the direct decomposition  $\op{H}^0(\op{Bun}_{\op{SL}(r)}(C), \mc{D}^{\tensor m})$ only features summands of level $m$.

To prove finite generation of the part of the sum on the left hand side necessary for the map $F$ to be surjective, we show that for the relevant $\vec{a}$, the varieties $\mc{X}(\vec{a})$ are GIT quotients of the same Thaddeus master space $\mathcal{M}_{X_0}$ by different ample linearizations (Def \ref{biglist}).
Finite generation then follows from basic arguments  (see Section \ref{FG}).

For related finite generation results for $G=\op{SL}(2)$, and $\op{SL}(3)$ see \cite{Man2, Man1}.

\subsection{Modular interpretations and other questions} It is natural to study $\mc{X}=\op{Proj}(\mathscr{A}_{\bullet})$, and the fibers $\op{Proj}(\mathscr{A}^{C}_{\bullet})$ at singular stable curves $C$, since
$\mathscr{A}_{\bullet}$ and $\mathscr{A}^{C}_{\bullet}$ are finitely generated.

  Projective varieties analogous to  $\op{SU}_{C}(r)$, for stable curves $C$,  constructed  using torsion free sheaves on  $C$,  will differ from $\op{SU}_{C}(r)$ in fundamental ways.  First, there is no known definition of a determinant of a torsion-free sheaf on a stable curve (see e.g., \cite{Faltings2} for a discussion). Second, for semistability, one needs to choose  polarizations on the stable curve, and the resulting spaces depend upon the choice of polarization, unlike for the smooth case. Moreover, sections of the analogous  line bundles over such moduli spaces are not equal to sections of the determinant of cohomology line bundle over the stack ${\op{Bun}}_{\op{SL}(r)}(C)$. In Section \ref{partiale} we outline an approach toward a potential modular interpretation for an open subset of $\op{Proj}(\mathscr{A}^{C}_{\bullet})$ for $C$ a stable curves with singularities.

We may consider
$\mc{X}=\op{Proj}(\mathscr{A}_{\bullet})$ in analogy with the Satake compactification of the variety $\op{A}_g$, which parametrizes principally polarized Abelian varieties over $\mathbb{C}$ of dimension $g\ge 2$.  Just as with $\mc{D}$ on  $\op{Bun}_{\op{SL}(r)}(C)$,  there is a distinguished ample line bundle $\mc{L}$ on the stack $\mc{A}_g$. Global sections of $\mc{L}^{\tensor k}$ are Siegel modular forms of weight $k$, and the section ring of $(\mc{A}_g,\mc{L})$ is finitely generated.

While the original proof  used the existence of the Satake compactification  as an analytic space   \cite{Cartan},  finite generation of the section ring of $(\mc{A}_g,\mc{L})$  can be proved (eg. \cite{FaltingsChai}), using toroidal compactifications $\ovop{A}_g$  in a manner reminiscent of our proof of finite generation of $(\op{Bun}_{\op{SL}(r)}(C), \mc{D})$: We use extensions of global sections of $\mc{L}_{\mc{G}}=\pi^*\mc{D}^{\otimes m}$  to compactifications $\mc{X}(\vec{a})$, while  extensions of global sections of $\omega^{m}=\ov{\pi}^*(\mc{L})$ over $\op{A}_g$  to compactifications $\ovop{A}_g$ can be used in the Satake case.    There is a difference: All global sections extend to any compactification in the Satake case, while in our case, global sections extend to different compactifications.

 By taking $\op{Proj}\big(\bigoplus_{k\in \mathbb{Z}_{\ge 0}} \op{H}^0(\mc{A}_g, \mc{L}^{\tensor k})\big)$,
one obtains the Satake (or Baily-Borel) compactification $\op{A}^{\star}_g$  (eg. \cites{SatakeCompactification, FaltingsChai}). While no clear modular interpretation is known, much has been learned about $\op{A}_g^{\star}$.  For instance, $\op{A}^{\star}_g$ is considered the {\em{smallest}} known compactification of $\op{A}_g$: There are canonical
morphisms $\ov{\pi}:\ovop{A}_g \to \op{A}_g^{\star}$ from all (smooth) toroidal compactifications $\ovop{A}_g$ of  $\op{A}_g$  (eg. \cite[Theorem 2.3]{FaltingsChai}).     It would be interesting to know how $\mc{X}$ and $\op{A}_g^{\star}$ compare more generally.

It is natural to ask if Theorems \ref{qmain} and  \ref{globaleJJs} generalize to arbitrary simple groups
$G$ (and parabolic analogues with marked points), and whether Theorem \ref{qmain} holds for all singular curves.  We also wonder if finite generation holds for higher dimensional varieties: For example, for a smooth projective surface $Z$, the moduli-stack of $G=\op{SL}(r)$ bundles, denoted $\op{Bun}_G(Z)$, again carries a determinant of cohomology line bundle $\mathcal{D}$. It makes sense to ask (at least for special types of surfaces) whether the section ring of $(\op{Bun}_G(Z),\mathcal{D})$ is finitely generated.

\section{Two basic line bundles}\label{LBS}
Here we define two line bundles: the determinant of cohomology line bundle (Def \ref{DCLB}), and  $T(\mc{A})$, (Def \ref{TA}), basic to all of our constructions.   The latter gives the projective embedding of the locus which is used to construct moduli spaces $\mc{X}(\vec{a})$ described in the introduction (Def \ref{XA}).  This locus lies in a particular
Quot scheme, also introduced here (Def \ref{Quot}).
\subsection{The determinant of cohomology line bundle}
Following  \cite{Faltings4}, we describe the determinant of cohomology of a vector bundle on a curve.
\begin{definition}\label{DetCo}
For any vector bundle $\me$ on a curve $C$, the determinant of cohomology of $\me$ on $C$ is the one dimensional vector space given by
\begin{equation}\label{deCo}
\mathcal{D}(C,\mc{E})=\Big(\Lambda^{\mbox{\tiny{max}}}\op{H}^0\big(C, \mc{E}\big)\Big)^* \tensor \Big(\Lambda^{\mbox{\tiny{max}}}\op{H}^1\big(C, \mc{E}\big)\Big).
\end{equation}
\end{definition}
${\op{Bun}}_{{\op{G}}}(C)$ is the  smooth {algebraic} stack whose fiber over a scheme $T$ is the groupoid of principal ${\op{G}}$-bundles on $C\times T$ \cite[Thms 1.0.1 and 6.0.18]{Wang}. Following  \cite{LaszloSorger}, we define the determinant of cohomology line bundle on ${\op{Bun}}_{{\op{G}}}(C)$.
\begin{definition}\label{DCLB}Let $\rho: \op{G}\to \op{Gl}(V)$ be a representation of $\op{G}$.
If  $E$ is a family of $\op{G}$-bundles on $C$ parameterized by a scheme $T$,  then given a point $t \in T$, one has that $E_t$ is a $\op{G}$-bundle on $C$, and one can form  a vector bundle $\mc{E}_t(V)$ on $C$ by taking the contracted product $\mc{E}_t(V) = E_t\times_{G} V$.
The determinant of cohomology line bundle $\mc{D}_{E}(V)$ is the line bundle on $T$  whose fiber over a  point $t \in T$ is the line $\mathcal{D}(C,\mc{E}_t(V))$, described in Def \ref{DetCo}. \end{definition}
The determinant of cohomology bundle on ${\op{Bun}}_{{\op{SL}(r)}}(C)$ associated to any representation is a tensor power of that associated to the standard representation.  This reduces easily, by passing to the normalization using the method of Eq \eqref{long} to the case for smooth curves, and the smooth case can  proved using \cite{DNar,LaszloSorger}. Similar statements for other groups appear in \cite{LaszloSorger}.
\subsection{The Quot Scheme and $T(\mc{A})$}

Given an ample line bundle $\mc{L}$ on $X_0$, we write  $\op{P}_{\mc{L}}(\mc{E},m)$ to denote the Hilbert polynomial of a  coherent sheaf  $\mc{E}$ on $X_0$ with respect to $\mc{L}$.  When clear from the context, we drop the subscript $\mc{L}$, writing $\op{P}(\mc{E},m)$, or simply $\op{P}$.

\begin{definition}\label{PDO} A coherent sheaf $\mathcal{E}$ on a curve $X_0$ is torsion free if it satisfies either of the following equivalent conditions:
\begin{enumerate}
\item $\mc{E}$ has depth one at any closed point;
\item  $\mc{E}$ is pure of dimension one: the dimension of support of any non-zero subsheaf of $\mc{E}$ is one.
\end{enumerate}
\end{definition}

\begin{definition}\label{RelTF}
A quasi-coherent sheaf $\me$ on a family of curves $X\to T$ is relatively torsion free if is flat of finite presentation, and torsion free when restricted to any fiber.
\end{definition}

\begin{definition}\label{Quot}Given an ample line bundle $\mc{L}$ on $X_0$, let $V$ be a fixed vector space of dimension $h^0(X_0, \mc{O}_{X_0}^{\oplus r}\tensor \mc{L})$.  By
$\op{Quot}_{X_0}(V\tensor \mc{L}^{-1},\op{P})$ we mean the component of the Quot scheme consisting of quotients $[V\tensor \mc{L}^{-1} \twoheadrightarrow \mc{E}]$ on $X_0$
such that  $\op{P}_{\mc{L}}(\mc{E},m)=\op{P}_{\mc{L}}(\mc{O}_{X_0}^{\oplus r},m)=\op{P}$. Let $\op{Quot}_{X_0}(V\tensor \mc{L}^{-1},\op{P},1)$ be the closure in $\op{Quot}_{X_0}(V\tensor \mc{L}^{-1},\op{P})$ of the set of points   $[V\tensor \mc{L}^{-1} \twoheadrightarrow \mc{E}]$ in $\op{Quot}_{X_0}(V\tensor \mc{L}^{-1},\op{P})$ such that  $\mc{E}$ is torsion free
(cf. the discussion before Lemma 1.17 in \cite{simpson}).
\end{definition}

Recall Grothendieck's embedding:  Let $\mc{L}$ and $\mc{A}$ be ample line bundles on $X_0$.  There is an integer $N$ such that for all $m\ge N$,
there is a natural $\op{GL}(V)$-equivariant embedding
\begin{equation}\label{springbreak}
\psi_m: \op{Quot}_{X_0}(V\tensor \mc{L}^{-1},\op{P},1) \rightarrow \op{Grass}^{\op{quot}}(V\tensor W,\rho),\ W=\op{H}^0(X_0, \mc{A}^{m}),
\end{equation}
where
$$\Bigl(V\tensor \mc{L}^{-1} \twoheadrightarrow \mc{E} \Bigr) \mapsto \Bigl(V\tensor \op{H}^0(X_0, \mc{A}^{m}) \twoheadrightarrow \op{H}^0(X_0,\mc{E} \tensor \mc{L} \tensor \mc{A}^{m})\Bigr).$$

\begin{definition}\label{TA}
Let $T(\mathcal{A})$ be the line bundle on  $\op{Quot}_{X_0}(V\tensor \mc{L}^{-1},\op{P},1)$ obtained as the pull back of the ample line bundle $\mathcal{O}_{\op{Grass}^{\op{quot}}(V\tensor W,\rho)}(1)$ under the map $\psi_m$ given in Eq \eqref{springbreak}. We do not wish to complicate the notation further by expressing the dependence of $T(\mathcal{A})$ on
$m$: Instead we will assume that $\mathcal{A}$ has been replaced by $\mathcal{A}^m$, so that $m=1$. We will use the $\op{GL}(V)$ linearization on $T(\mathcal{A})$ coming from the $\op{PGL}(V)$ linearization on $\mathcal{O}_{\op{Grass}^{\op{quot}}(V\tensor W,\rho)}(1)$ (strictly speaking we need to take a tensor power for the $\op{PGL}(V)$-linearization).

\end{definition}

\section{Notions of semi-stability}
In Def \ref{XA} we describe varieties   $\mc{X}(\vec{a})$,  constructed as GIT quotients of $\vec{a}$-semistable torsion free sheaves on $X_0$ using a linearization given by the line bundle $T(\mc{A})$.   In this section we generalize the standard notion of {$\vec{a}$-semistable torsion free sheaves on $X_0$ to allow for some negative weights, as long as $\sum_{i}a_i >0$ (see Section \ref{GenSes}). We also describe
the correspondence between  slope-semistability weights $\vec{\alpha}$ for the line bundle  $T(\mc{A})$, defined in Section \ref{LBS}, and the semistability weights  $\vec{a}$.

\subsection{$\vec{a}$-semistable torsion free sheaves on $X_0$}

\begin{definition}\label{aStability}Suppose that $X_0$ is a stable curve with irreducible components $\{X_{0,i}\}_{i\in I}$.  Let $\mc{E}$ be a torsion free sheaf on $X_0$.  We will use the following notation.
\begin{enumerate}
\item If $r_i=\op{rk}_i(\mc{E})=\op{rk}(\mc{E}|_{X_{0,i}})$, then $(r_i)_{i\in I}$ is the multi-rank of $\mc{E}$, and $\mc{E}$ will be said to have uniform multi-rank if $r=r_i$ for all $i \in I$.
\item If $d_i=\op{deg}_i(\mc{E})=\op{deg}(\mc{E}|_{X_{0,i}})$, then $(d_i)_{i\in I}$ is the multi-degree of $\mc{E}$.
\item Let  $\vec{a}=(a_i)_{i\in I}$ be positive rational numbers.  Suppose $\mc{E}$ has multi-rank $(r_i)_{i\in I}$.  The $\vec{a}$-slope of $\mc{E}$ is
$$\mu_{\vec{a}}(\mc{E})=\frac{\chi(X_0,\mc{E})}{\sum_{i\in I}r_i a_i}.$$
\item  We say that a sheaf $\mc{E}$ on $X_0$ is $\vec{a}$- semistable (respectively $\vec{a}$-stable) if for every nonzero proper subsheaf $\mc{F} \subset \mc{E}$, one has
\begin{equation}\label{bart}
\mu_{\vec{a}}(\mc{F})\le \mu_{\vec{a}}(\mc{E}) \ \ \ (\mbox{respectively } \mu_{\vec{a}}(\mc{F}) < \mu_{\vec{a}}(\mc{E})).
\end{equation}
\end{enumerate}
\end{definition}
We have  found it useful to generalize Def (4) above, which is due to Seshadri \cite{Seshadri},  to allow for the possibility that some weights  $a_i$ can be negative or zero.  To do so, we write the semistability inequality in Eq  \eqref{bart} as follows:
\begin{equation}\label{newform}
\chi(X_0,\mc{F})\leq \mu_{\vec{a}}(\mc{E})\cdot\sum a_i r_i(\mf).
\end{equation}
Setting $\gamma_i=a_i  \mu_{\vec{a}}(\mc{E})$, we can rewrite the inequality as
$$\chi(X_0,\mc{F})\leq \sum \gamma_i r_i(\mf),$$
and we note that this inequality holds as an equality for $\mf=\me$. The following generalization, given in Def \ref{Generalization} is therefore natural.

\begin{definition}\label{Generalization}
 Suppose we are given real numbers $\vec{\gamma}=\{\gamma_i\}_{i\in I}$, where the $\gamma_i$ may possibly be negative or zero. Assume further that
 \begin{equation}\label{smiley0}
\chi(X_0,\me)= \sum \gamma_i r_i(\me).
\end{equation}
We say that  $\me$ is $\vec{\gamma}$-linearly semistable if for any subsheaf $\mf\subseteq \me$, the following inequality
\begin{equation}\label{smiley}
\chi(X_0,\mf) \leq \sum \gamma_i r_i(\mf)
\end{equation}
holds.
\end{definition}

\subsection{}\label{GenSes}We next explain why it is natural,  in Seshadri's definition of semistability  given in Def \ref{aStability}(4), if $\me$ has uniform rank,  to allow some weights $a_i$ to be negative, as long as $\sum_{i}a_i \neq 0$.   In particular, we will show that in case $\me$ has uniform rank, then modulo tensoring with line bundles, the new notion of semistability
is equivalent to the original definition of semi-stability.

If $\gamma_i>0$ for all $i$, then we can write Eq \eqref{smiley} as
\begin{equation}\label{smiley1}
\frac{\chi(X_0,\mf)}{\sum \gamma_i r_i(\mf)}\leq 1=\frac{\chi(X_0,\me)}{\sum \gamma_i r_i(\me)}.
\end{equation}

If $\gamma_i<0$ for all $i$, let $\delta_i=-\gamma_i$ and write Eq \eqref{smiley} as
\begin{equation}\label{smiley2}
\frac{\chi(X_0,\mf)}{\sum \delta_i r_i(\mf)}\leq -1=\frac{\chi(X_0,\me)}{\sum \delta_i r_i(\me)}.
\end{equation}

So in these two cases we recover Seshadri's definition of semistability. For the other cases, if $\ml$ a line bundle on $X_0$, consider the formula  \cite[Corollary 8, page 152]{Seshadri}:
\begin{equation}\label{norepeat}
\chi(X_0,\mf\tensor\ml)=\chi(X_0,\mf)+\sum_{i} \deg_i(\ml)r_i(\mf).
\end{equation}
This gives that Eq \eqref{smiley} is equivalent to, with $\mf'=\mf\tensor\ml\subset\me'=\me\tensor\ml$,
\begin{equation}\label{smileyshift}
\chi(X_0,\mf')\leq \sum_i (\gamma_i+\deg_i(\ml)) r_i(\mf)=\sum \tau_i r_i(\mf'),\ \ \tau_i=\gamma_i+\deg_i(\ml),
\end{equation}
with Eq \eqref{smiley} holding as an equality for $\mf'=\me'$. Here we have used $r_i(\mf)=r_i(\mf')$. Therefore, $\me\tensor \ml$ is $\vec{\tau}$-linearly semistable if and only if
$\me$ is $\vec{\gamma}$-linearly semistable.

But $\tau_i>0$ if $\ml$ is sufficiently ample, and therefore up to tensoring with line bundles, the new notion of semistability
is equivalent to the original (Seshadri) one.Therefore,
\begin{remark}\label{SeshadriGeneral}
We will always write Seshadri's notion of semistability using inequality \eqref{newform} for vector bundles $\me$ of uniform rank, and generalize this definition to allow some $a_i$ to be negative or zero, and $\sum a_i\neq 0$.
\end{remark}

\subsection{Correlation between semistability weights $\vec{a}$, and slope-semistability weights $\vec{\alpha}$}\label{twist}
\begin{definition}\label{ADef}Let $\vec{a}=\{a_i\}_{i\in I}$, and $\sum_{i\in I}a_i>0$, and let $\mc{L}$ be an ample line bundle on $X_0$.  Set $\vec{\alpha}=\{\alpha_i\}_{i\in I}$, where
\begin{equation}\label{haupti}
\alpha_i=\alpha \Bigl(\deg_i (\mathcal{L}) - \frac{a_i(g-1)}{\sum_{i\in I} a_i}\Bigr),
\end{equation}
and let $\mathcal{A}$ be the line bundle on $X_0$ with degree $\alpha_i=\deg_i(\mathcal{A})$ on $X_{0,i}$. Here $\alpha$ is a suitable very divisible integer (so that $\alpha_i$ is an integer). The line bundles  $\mathcal{A}$ are well defined up to scale. We will always work in situations where $\deg_i(\mathcal{A})>0$ for all $i$, and hence $\mathcal{A}$ is ample.
\end{definition}

This choice of weights $\vec{\alpha}$  is the following, an immediate consequence of  Eq \eqref{norepeat}:
\begin{lemma}\label{pastel}
Let $\me$ be a torsion free sheaf on $X_0$ with Hilbert polynomial with respect to $\ml$:
$$\chi(X_0,\me\tensor\ml^m)=\chi(X_0,\mathcal{O}^{\oplus r}\tensor\ml^m)=r\chi(X_0, \ml^m).$$
The following are equivalent,
\begin{enumerate}
\item $\me$ is $\vec{a}$- semistable (i.e., inequality \eqref{newform} holds for all $\mf\subseteq\me$, see Remark \ref{SeshadriGeneral});
\item $\me\tensor\ml$ is $\vec{\alpha}$- semistable; for weights given by $\alpha_i=\deg_i(\mathcal{A})$.
\end{enumerate}
\end{lemma}
\begin{remark}\label{folge}
We will work in situations where the absolute values of $\frac{a_i}{\sum a_i}$ are bounded above.
\begin{enumerate}
\item For $\mathcal{A}$ as in Def \ref{ADef},  summing Eq \eqref{haupti} over $i$, we get
$$\deg(\mathcal{A})=\alpha(\deg(\mathcal{L})-(g-1))=\frac{\alpha}{r} \chi(X_0,\me\tensor\mathcal{L}).$$

\item If $\deg_i(\ml)$ are sufficiently large (which we will assume), then $\deg_i(\mathcal{A})>0$.
\item Writing $c_i=\frac{\deg_i(\mathcal{A})}{\deg(\mathcal{A})}$, and noting that $\sum c_i=1$, Eq \eqref{haupti} implies:
$$\deg(\ml)(c_i-\frac{\deg_i(\ml)}{\deg(\ml)})= (g-1)(c_i-\frac{a_i}{\sum a_i}).$$
Hence, $|(c_i-\frac{\deg_i(\ml)}{\deg(\ml)})|< \frac{E}{\deg(\ml)}$ for a constant $E$.
    Therefore the weights $c_i$ are very close to the weights determined by $\ml$, when $\deg(\ml)$ is large.
\end{enumerate}
\end{remark}

\section{The semistability locus  $\op{Q}_{X_0}^0$ for $T(A)$, and line bundle identities on $\op{Q}_{X_0}^0$}
In this section we define a locus  $\op{Q}_{X_0}^0$, and  show that all points semistable  for $T(\mathcal{A})$ are in $\op{Q}_{X_0}^0$ (see Remark \ref{attention}).  We also introduce line bundles $\mc{L}_{\mc{G}}$, and we describe relationships between   $T(\mathcal{A})$ and $\mathcal{D}$ with  $\mc{L}_{\mc{G}}$ on the locus $\op{Q}_{X_0}^0$.

\subsection{Notation}
\begin{definition}\label{Quot0} Let $\op{Q}_{X_0}^0$ be the open set of points $[V\tensor\ml^{-1}\to \me] \in \op{Quot}_{X_0}(V\tensor \mc{L}^{-1},\op{P},1)$ such that
\begin{enumerate}
\item  $\mc{E}$ is torsion free; and
\item the map $V\to \op{H}^0(X_0,\mc{E}\tensor \mc{L})$ is an isomorphism; in particular $h^1(X_0,\mc{E}\tensor \mc{L})=0$.
\end{enumerate}
\end{definition}

As mentioned in the introduction, throughout this work, while discussing singular stable curves, we use the notation $X_0$, and for their normalization, we write $\nu: X\to X_0$.  The set of irreducible components of $X_0$ is denoted $\{X_{0,i}\}_{i\in I}$, while the set of irreducible components of $X$ is expressed as $\{X_{i}\}_{i\in I}$.  In particular $\nu_i = \nu|_{X_i}: X_i \to X_{0,i}$ is the normalization, for each $i \in I$.
Choose smooth points $p_i$ on irreducible components $X_{0,i}$ of $X_0$. Henceforth we will  assume that $\mathcal{A}=\mathcal{O}(\sum \alpha_i p_i)$ and $\mathcal{L}=\mathcal{O}(\sum \beta_i p_i)$ for suitable $\alpha_i$ and $\beta_i$.

\begin{definition}\label{plus}
Let $\mathcal{G}=\mathcal{G}_{\vec{a}}$ be a vector bundle on $X_0$ of rank $r(\mg)=r\deg(\mathcal{A})$ obtained as a direct sum of line bundles of form $\mathcal{O}(\sum \gamma_i p_i)$ such that
$$\frac{\deg_i(\mathcal{G})}{\rk(\mg)}=\frac{a_i(g-1)}{\sum a_i}.$$
\end{definition}
\begin{definition}
Let  $\aleph(X_0)$ denote  the moduli stack of torsion free sheaves on $X_0$.
\end{definition}

\begin{definition}\label{LBG}For $\mc{G}$ a vector bundle on $X_0$, let $\mathcal{L}_{\mg}$ be a line bundle on $\aleph(X_0)$, whose fiber at $\mathcal{E}$ is the determinant of cohomology $\mc{D}(X_0,\me\tensor\mg)$.
\end{definition}

\subsection{The relationships between $\mc{L}_{\mc{G}}$, $T(\mathcal{A})$ and $\mathcal{D}$ on $\op{Q}_{X_0}^0$}
\begin{lemma}There is an identity of $T(\mathcal{A})$ and $\mathcal{D}$, as $\op{GL}(V)$ linearized line bundles, which at fibers over points
  $x=[V\tensor \mc{L}^{-1}\twoheadrightarrow \mc{E}]\in \op{Q}^0_{\mc{L}}(r)$, is given by
\begin{equation}\label{Khokar}
T(\mathcal{A})|_{x}^{\chi(X_0,\me\tensor \mathcal{L})} \cong (\mathcal{D}(X_0,\me\tensor \mathcal{L}\tensor\mathcal{A})^*)^{-\chi(X_0,\me\tensor \mathcal{L})} \tensor \mathcal{D}(X_0,\me\tensor \mathcal{L})^{\chi(X_0,\me\tensor \mathcal{L}\tensor\mathcal{A})}.
\end{equation}
\end{lemma}

\begin{proof}

There is an isomorphism $$T(\mathcal{A})|_{x}\to \det(\op{H}^0(X_0,\me\tensor \mathcal{L}\tensor \mathcal{A}))=(\mathcal{D}(X_0,\me\tensor \mathcal{L}\tensor \mathcal{A}))^*.$$
This map is not $\op{GL}(V)$ equivariant, but is equivariant for $\op{PGL}(V)$: A scalar matrix $t\in \op{GL}(V)$ acts trivially on the left hand side, but acts by $t$ raised to $\chi(X_0,\me\tensor \mathcal{M})$ on the right hand side.
Let $\mc{S}$ be the constant line bundle on $\op{Quot}_{\ml}(r)$, but with $t\in G$ acting like  inverse of how it does on $\det(V)$. Therefore, for $x\in \op{Q}^0_{\mc{L}}(r)$, we have
$$T(\mathcal{A})|_{x} = (\mathcal{D}(X_0,\me\tensor \mathcal{L}\tensor \mathcal{A}))^* \tensor \mathcal{S}^{\chi(X_0,\me\tensor \mathcal{L}\tensor\mathcal{A})}.$$

However on $\op{Q}_{X_0}^0$, we have that  $\op{H}^0(X_0,\me\tensor \mathcal{L})$ is trivial (compatibly with $\op{PGL}(V)$) but $\lambda\in \op{GL}(V)$ acts on the right hand side by $t$ raised to $-\chi(X_0,\me\tensor \mathcal{L}\tensor\mathcal{A})$. Therefore we obtain
canonical isomorphisms on $\op{Q}_{X_0}^0$:
$$T(\mathcal{A})|_{x} = (\mathcal{D}(X_0,\me\tensor \mathcal{L}\tensor \mathcal{A}))^* \tensor \mathcal{D}(X_0,\me\tensor \mathcal{L})^{\frac{\chi(X_0,\me\tensor \mathcal{L}\tensor\mathcal{A})}{\chi(X_0,\me\tensor \mathcal{L})}}.$$

\end{proof}

\begin{lemma}\label{repeat}
The restriction of $\mc{L}_{\mc{G}}$ on $\aleph_{r}(X_0)$ to $\op{Bun}_{\op{SL}(r)}(X_0)$ is $\mc{D}^{\ell}$, where
$\ell= \op{rk}(\mg)$.
\end{lemma}

\begin{proof}
This follows from a repeated use of Lemma \ref{repeatedly}.
\end{proof}

\begin{lemma}\label{repeatedly}
Let $\me$ be a torsion free sheaf on $X_0$, and $p\in X_0$ a smooth point. Then $$\mathcal{D}(X_0,\me\tensor\mathcal{O}(p))=\mc{D}(X_0,\me)\tensor \det(\me_p)^*.$$
\end{lemma}
\begin{proof} There is an exact sequence
$$0\to \me\to\me\tensor\mathcal{O}(p)\to \me_p\to 0,$$
since  $\me$ is a vector bundle in a neighborhood of $p$.
\end{proof}

\begin{lemma}\label{WhiteDog}
For $x=[V\tensor \mc{L}^{-1}\twoheadrightarrow \mc{E}]\in \op{Q}^0_{\mc{L}}(r)$,
$$T(\mathcal{A})|_{x}^{\chi(X_0,\me\tensor \mathcal{L})}=\mathcal{D}(X_0,\me\tensor\mg)$$
as line bundles with $\op{GL}(V)$ linearizations.
\end{lemma}

\begin{proof}
Note that by \eqref{haupti} and Remark \ref{folge}(1),
$$\frac{\deg_i(\mathcal{G})}{\rk(\mg)}=\deg_i(\ml) - \frac{\deg_i(\mathcal{A})}{\deg(\mathcal{A})}\cdot {\chi(X_0,\ml)}.$$

Then, using Eq \ref{Khokar}, and Lemma \ref{repeatedly} repeatedly, we obtain the result claimed.
\end{proof}

\subsection{All semistable points for $T(\mathcal{A})$ are in $\op{Q}_{X_0}^0$ }

In Proposition \ref{simpsonian} we will show that with regard to understanding $T(\mathcal{A})$, the set $\op{Quot}_{X_0}(V\tensor \mc{L}^{-1},\op{P},1)\setminus \op{Q}_{X_0}^0$ is rather wild from a geometric point of view, and we would therefore want to
restrict our study to the sublocus $\op{Q}_{X_0}^0$. The section rings over  $\op{Q}_{X_0}^0$ and $\op{Quot}_{X_0}(V\tensor \mc{L}^{-1},\op{P},1)$ (actually of normalizations of certain closed subsets, see Lemma \ref{EndGoal}) coincide because of the following lemma, attributed to Seshadri, and Proposition \ref{simpsonian}.

 \begin{lemma}\label{extendo}  \cite[Lemma 4.15]{NRam} Let $M$ be a projective scheme on which a reductive group $\op{G}$ acts. Let
$\mc{N}$ an ample line bundle linearizing the $\op{G}$ action, and $M^{ss}$ the open subscheme of semistable points.
Let $W$ be an open $\op{G}$-invariant (irreducible) normal subscheme of $M$ containing
$M^{ss}$.  Then $$\op{H}^{0}(M^{ss},\mc{N})^{\op{G}} = \op{H}^{0}(W,\mc{N})^{\op{G}}.$$
In particular, we may take $W=M$, if $M$ is normal.
\end{lemma}
\begin{remark}\label{attention}
 The following result is valid after replacing  an initial choice $\ml_0$ of $\ml$ by a suitable $\ml^N$. Note that changing $\ml$ also changes the line bundles $\mathcal{A}$.
\end{remark}
\begin{proposition}\label{simpsonian}
If $x=[V\tensor \mc{L}^{-1}\twoheadrightarrow \mc{E}]\in \op{Quot}_{X_0}(V\tensor \mc{L}^{-1},\op{P},1)\setminus \op{Q}_{X_0}^0$, then $x$ is not semistable for
 $T(\mc{A})$.
\end{proposition}
Our goal is Lemma \ref{EndGoal}, for which we need a number of results, relying upon the work of  \cite{simpson}.

 There are three possible reasons for $x$ not to be in $\op{Q}_{X_0}^0$:
\begin{enumerate}
\item[(A)] $\me$ is torsion free but the map $V\to \op{H}^0(X_0,\me\tensor \mathcal{L})$ is not injective.
\item[(B)] $\me$ is torsion free and the vector space $\op{H}^0(X_0,\me\tensor \mathcal{L})$ has dimension greater than the Euler characteristic, and   hence $V\to \op{H}^0(X_0,\me\tensor \mathcal{L})$ is not surjective, but injective.
\item [(C)] The sheaf $\mathcal{E}$ is not torsion free.
\end{enumerate}

\begin{lemma}\label{GrassGIT} \cite[Proposition 1.14]{simpson}
A quotient $V\tensor W\to U\to 0$ is semistable (for the action of $\op{GL}(V)$ if and only if for every $H\subseteq V$, the image $\op{Im}(H\tensor W)$ of $H\tensor W$
in $U$ is not zero, and
\begin{equation}\label{ss}
\frac{\dim H}{\dim \op{Im}(H\tensor W)}\leq \frac{\dim V}{\dim U}.
\end{equation}
\end{lemma}

\begin{definition} Let $X_0$ be a nodal curve with irreducible components $\{X_{0,i}\}_{i\in I}$ and $\mc{A}$ a line bundle on $X_0$.  For a sheaf $\mf$ on $X_0$, we set
$$r_{\mathcal{A}}(\mf)=\sum_i r_i(\mf)\deg_i(\mathcal{A}),$$
where $r_i(\mf)=\op{rk}(\mf |_{X_{0,i}})$, and $\deg_i(\mathcal{A})=\op{deg}(\mathcal{A}|_{X_{0,i}})$.
\end{definition}
\begin{lemma}\label{monty2}
There exists a number $N_0=N_0(\mathcal{A},\ml)$ so that for $m \geq N_0$, if $x=[V\tensor\ml^{-1}\to \me] \in \op{Quot}_{\ml, \op{P},1}(r)$ which is $T(\mathcal{A}^m)$-semi-stable, then the following property holds: For any non-zero subspace $H\subseteq V$, let $\mathcal{F}\subset\mathcal{E}$ the subsheaf generated by $H\tensor \mathcal{L}^{-1}$. Then, $r_{\mathcal{A}}(\mathcal{F})>0$ and
\begin{equation}\label{nom}
\frac{\dim H}{r_{\mathcal{A}}(\mathcal{F})}\leq \frac{\dim V}{r_{\mathcal{A}}(\mathcal{E})}.
\end{equation}
\end{lemma}
\begin{proof}
Assume $r_{\mathcal{A}}(\mathcal{F})=0$ or that Eq \eqref{nom} fails, and hence
\begin{equation}\label{nomnew}
\frac{\dim H}{r_{\mathcal{A}}(\mathcal{F})}> \frac{\dim V}{r_{\mathcal{A}}(\mathcal{E})}.
\end{equation}

Let $\mathcal{K}$ be the kernel of $H\tensor \ml^{-1}\to\mf$. We therefore have an exact sequence
$$0\to\mk\to H\tensor \ml^{-1}\to\mf \to 0.$$
Since the set of $H$ is bounded, tensoring by $\mathcal{A}_m=\ml\tensor\mathcal{A}^m$ for large enough $m$, we can assume
$$H^1(X_0,\mk\tensor\mathcal{A}_m)= H^1(X_0,\mf\tensor\mathcal{A}_m)=0,$$
for all $H$.  It follows that $\op{Im}(H\tensor W)$ in Lemma \ref{GrassGIT} equals $\op{H}^0(X_0, \mf\tensor\mathcal{A}_m)$. Therefore the  inequality given in
Eq \eqref{ss} (since $x$ is semistable for $T(\mathcal{A}^m)$) implies the inequality
\begin{equation}\label{nomnom}
\frac{\dim H}{\chi(X_0, \mf\tensor\mathcal{A}_m)}\leq \frac{\dim V}{\chi(X_0, \me\tensor\mathcal{A}_m)}.
\end{equation}
Here we have used that $\chi(X_0, \me\tensor\mathcal{A}_m)=\dim W$. But
$$\chi(X_0, \mf\tensor\mathcal{A}_m)=\chi(X_0,\mf\tensor\ml) + m \sum_i r_i(\mf) \deg_i(\mathcal{A})= \chi(X_0,\mf\tensor\ml) + m r_{\mathcal{A}}(\mathcal{F}).$$
Similarly
$$\chi(X_0, \me\tensor\mathcal{A}_m)=\chi(X_0,\me \tensor \ml) + m r_{\mathcal{A}}(\mathcal{E})$$
and therefore Eq \eqref{nomnom} becomes
\begin{equation}\label{nomnomnom}
\frac{\dim H }{\chi(X_0, \mf\tensor\ml)+ m r_{\mathcal{A}}(\mathcal{F})}\leq \frac{\dim V}{\chi(X_0,\me \tensor \ml) + m r_{\mathcal{A}}(\mathcal{E})}.
\end{equation}
It is easy to see that  Eq \eqref{nomnew} and Eq \eqref{nomnomnom} contradict each other for $m$ large.
\end{proof}

\begin{corollary}\label{morn}
 Assume $x=[V\tensor\ml^{-1}\to \me] \in \op{Quot}_{\ml, \op{P},1}(r)$ which is semi-stable for $T(\mathcal{A}^m)$, with $m\geq N_0$ and  $\me\to\mf$  a quotient. Let $J$ be the image of $V\to \op{H}^0(X_0,\ml\tensor\mf)$
 then $$\frac{\dim(J)}{r_{\mathcal{A}}(\mathcal{F})}\geq \frac{\dim V}{r_{\mathcal{A}}(\mathcal{E})}.$$
 \end{corollary}
We will replace $\mathcal{A}$ by $\mathcal{A}^m$ with a $m\geq N_0$.

\subsubsection{Treating reason (A)}

\begin{lemma}\label{nA}
 Suppose $x=[V\tensor\ml^{-1}\to \me] \in \op{Quot}_{\ml}(r)$ is such that $V\to \op{H}^0(X_0, \me\tensor\ml)$ is not injective.  If  $H$ is the kernel of $V\to \op{H}^0(X_0,\me\tensor\ml)$, the semistability inequality for $H$ fails.
 \end{lemma}
\begin{proof}

   By hypothesis, $H=Ker(V \to  \op{H}^0(X_0, \me\tensor\ml))$ is nontrivial.  This means, considering the surjection
   $V\tensor \mc{O}_X \to \me \tensor \mathcal{L}$, the restriction $H\tensor \mc{O}_X \to \me \tensor \mathcal{L}$ is the zero map.  In other words,
   $H\tensor \mathcal{L}^{-1} \to \me$ is the zero map.  Tensoring with $\mathcal{A}_m=\ml\tensor\mathcal{A}^m$ and taking global sections, we get maps
   $$H \tensor W \to V \tensor W \to \op{H}^0(X_0, \me \tensor \mathcal{A}_m)=U,$$
   such that the image of $H \tensor W$ in $U$ is zero.  This means, by Lemma \ref{monty2}, that $\tilde{\tilde{x}}$ is not semistable in
    $\op{Gr}^{Quot}(V \tensor W)$, which  implies that $x$ is not semistable in $\op{Quot}_{\ml}(r)$.
\end{proof}

\subsubsection{An important estimate to treat cases (B) and (C)}\label{last}
If $\me$ belongs to an bounded set (specified  a priori), we can assume that $H^1(\me\tensor\ml)=0$ (replacing $\ml$ by $\ml^m$, and rescaling $\mathcal{A}$ after that). This will rule out (B).

We will fix an initial value of $\ml$, say $\ml_0$. The slope of a sheaf $\mf$ (always with respect to the polarization $\ml_0$ below) is
\begin{equation}\label{LA}
\mu(\mf)=\mu(\mf,\ml_0)=\frac{\chi(X_0,\mf)}{r_{\ml_0}(\mf)}.
\end{equation}
We will use the theory of Harder-Narasimhan filtrations (see e.g., \cite{simpson}), always with respect to the fixed $\ml_0$.

Now  assume that $\me$ is a sheaf (with Hilbert polynomial $r\chi(\ml_0)$) which has a Harder -Narasimhan quotient   of  slope $\mu<\mu_0$ (the complement is bounded). We will specify $\mu_0$ at the very end of this argument.

Therefore from a  semi-stable quotient $\me\to \mf$, and $\mu(\mf,\ml_0)<\mu_0$, we will need to  produce a subspace $H\subset V$ which contradicts the $T(\mathcal{A})$ semistability of $x$.

\begin{lemma}\label{eclat}
There is a constant $\Cee$, depending only $r$ and the bounds we have assumed for the absolute values of $\frac{a_i}{\sum a_i}$, with the following property: Let $\mathcal{F}$ be a quotient of $\me$ where $x=[V\tensor\ml^{-1}\to \me] \in \op{Quot}_{\ml}(r)$ , and such that
\begin{equation}\label{inegality}
\op{H}^0(X_0,\mf\tensor \ml)-\sum r_i(\mf)\deg_i(\ml)< \Cee.
\end{equation}
Here $\me$ is allowed to have torsion, a case that is used in (C).
The kernel
$H$ of the (composite) map $V\to \op{H}^0(X_0,\me\tensor\ml)\to \op{H}^0(X_0,\mf\tensor\ml)$
contradicts semistability of the point $x$ for the polarization $T(\mathcal{A})$.
\end{lemma}

\begin{proof}
Assume, by way of contradiction, that $x$ is semistable for the polarization given by $T(\mathcal{A})$.
Now let $J$ be the image of the map $V\to \op{Hom}(\ml^{-1},\mf)$, therefore
$\dim J\leq \op{H}^0(X_0,\mf\tensor\ml)$. However by Corollary \ref{morn} (see also \cite[Remark after Lemma 1.16]{simpson})
\begin{equation}
\dim J\geq \frac{r(1-g)+\deg(L)r}{r\sum b_i}(\sum b_i r_i(\mf))= ((1-g)+\deg(L))(\sum \frac{b_i}{\sum b_i} r_i(\mf))
\end{equation}
where $b_i=\deg_i(\mathcal{A})$. Now let $c_i=\frac{b_i}{\sum b_i}=\frac{\deg_i(\mathcal{A})}{\deg(\mathcal{A})}$. We therefore get
$$  \op{H}^0(X_0,\mf\tensor\ml)\geq ((1-g)+\deg(\ml))(\sum c_i r_i(\mf))$$
and hence
\begin{equation}\label{bund}
 \op{H}^0(X_0,\mf\tensor \ml)-\sum r_i(\mf)\deg_i(\ml)\geq  (1-g)\sum c_ir_i(\mf) +\deg(\ml)\sum (c_i -\frac{deg_i(\ml)}{\deg(\ml)})r_i(\mf).
 \end{equation}

The term $(1-g)\sum c_ir_i(\mf)$ is bounded below. Therefore we need to bound the remaining terms. By Remark \ref{folge},
$|(c_i-\frac{\deg_i(\ml)}{\deg(\ml)})|< \frac{E}{\deg(\ml)}$, and hence the second term on the right in Eq \eqref{bund} is also bounded below.

\end{proof}

\begin{lemma}
There is an integer $\beta$ which depends only on integer $s$  such that if $\mf$ is any torsion-free sheaf on $X_0$, which satisfies
\begin{itemize}
\item $\mf$ is semistable for the polarization given by $\ml_0$.
\item  $r_{\ml_0}(\mf)=s$,
\end{itemize}
 then for any integer $m\geq 0$
$$\op{H}^0(X_0,\mf\tensor \ml_0^{\tensor m})-m\sum r_i(\mf)\deg_i(\ml_0)\leq (\mu(\mf)+\beta)\sum r_i(\mf)\deg_i(\ml_0),$$
where $\mu(\mf)=\mu(\mf,\ml_0)$ is given by Eq \eqref{LA}.
\end{lemma}
 \begin{proof}
See \cite[Corollary 1.7]{simpson}.
\end{proof}

If $\mf$ is a quotient of a sheaf $\me$ of a fixed $\ml_0$ rank $r_0$, then the possible ranks of $\mf$ is a finite set. Therefore we can assume that we can choose a $\beta$-uniformly for such ranks in Lemma.  Find a $\mu_0$ such that $(\mu_0+\beta)r_{\ml_0}(\me)< \Cee$
($\Cee$ as specified in Lemma \ref{eclat}). Then we can  handle case (B) as follows: We only need to consider the case $\me$ has a Harder-Narasimhan quotient $\mf$ of slope $<\mu_0$. The LHS of \eqref{inegality} is $\leq (\mu_0 +\beta)r_{\ml_0}(\mf)<\Cee$, therefore Lemma \ref{eclat} produces a canonical witness to the non-semistability of $x$.

\subsubsection{Treating reason (C)}

Suppose $x=[V\tensor\ml^{-1}\to \me] \in \op{Quot}_{\ml}(r)$ is such that $\me$ is not torsion free. Let $\mathcal{C}\subset \mathcal{E}$ be the torsion subsheaf. By \cite[Lemma 1.17]{simpson} we can find a torsion free $\me'$ of the same Hilbert polynomial as $\me$ with respect to
$\ml$ so that there is an inclusion
$$0\to \me/\mathcal{C}\to \me'.$$

\subsubsection{If all Harder-Narasimhan quotients of $\me'$ (with respect to $\ml$) have slopes $>\mu_0$}
In this case we know that $\op{H}^1(X_0,\me'\tensor\ml)=0$ and
$\me'\tensor\ml$ is generated by global sections.

But global sections  of $(\me/\mathcal{C}\tensor\ml)^m$ sit inside $\op{H}^0(X_0,\me'\tensor\ml^m)$. Therefore define the following quotient of $\me$: $\mf=\me/\mathcal{C}$, and obtain
$\op{H}^0(X_0,\mf\tensor\ml)\le \chi(X_0,\me\tensor\ml)$.

If $x$ is semistable for $T(\mathcal{A})$, by Corollary \ref{morn} (also the remark following Lemma 1.16 in \cite{simpson}), we would have $\op{H}^0(X_0,\mf\tensor\ml)\geq \dim J \geq \chi(X_0,\me\tensor\ml)$
Therefore $\op{H}^0(X_0,\mf\tensor\ml)\to \op{H}^0(X_0,\me'\tensor\ml)$ is an isomorphism. But $\me'\tensor\ml$ is globally generated. This gives $\mathcal{C}=0$.
\subsubsection{If the smallest Harder-Narasimhan quotient of $\me'$ (with respect to $\ml$) has slopes $\leq \mu_0$}

Let $\me'\to \mf'$ be the corresponding quotient. Let $\mf\subseteq\mf'$ be the image of $\me$. Therefore $\mf$ is a quotient of $\me$. Now,
$$\op{H}^0(X_0,\mf'\tensor \ml)-\sum r_i(\mf')\deg_i(\ml)<\Cee$$
(note $\ml=\ml_0^m$)
for the same reasons as before. But
$\op{H}^0(X_0,\mf\tensor \ml)\leq \op{H}^0(X_0,\mf'\tensor \ml)$ and $r_i(\mf')=r_i(\mf)$, therefore
$$\op{H}^0(X_0,\mf\tensor \ml)-\sum r_i(\mf)\deg_i(\ml)<\Cee$$
and using Lemma \ref{eclat}, we reach a contradiction.

\section{Definition of the map $F$ and varieties $\mc{X}(\vec{a})$}
Here, in Proposition \ref{OneIntro}, we establish the inclusion which leads to the map $F$
discussed in the introduction, which we show in Section \ref{SomethingIntroSection} is surjective.  We also define the varieties $\mc{X}(\vec{a})$.

For this section assume that we are given weights $\vec{a}$, a vector bundle ${\mc{G}}=\mc{G}_{\vec{a}}$ on $X_0$ with  rank $m=r\deg(\mathcal{A})$,  and (see Def \ref{plus})
$$\frac{\deg_i(\mathcal{G})}{\rk(\mg)}=\frac{a_i(g-1)}{\sum a_i}.$$

\begin{proposition}\label{OneIntro}Let $\mc{D}$ be the determinant of cohomology line bundle on $\op{Bun}_{\op{SL}(r)}(X_0)$ associated to the standard representation of $\op{SL}(r)$. Then, there is a natural inclusion
\begin{equation}\label{SectionEmb}
\op{H}^0(\mc{X}(\vec{a}), \mc{L}_{\mc{G}})\hookrightarrow \op{H}^0(\op{Bun}_{\op{SL}(r)}(X_0), \mc{D}^{\tensor m}).
\end{equation}
In particular, there is a map
$$F: \bigoplus_{(\vec{a},\mc{G})} \op{H}^0(\mc{X}(\vec{a}), \mc{L}_{\mc{G}}) \to \bigoplus_{m\in \mathbb{Z}_{\ge 0}} \op{H}^0(\op{Bun}_{\op{SL}(r)}(X_0), \mc{D}^{\tensor m}).$$
\end{proposition}

\subsection{Notation and basic results}\label{MasterSpace}
To prove Proposition \ref{OneIntro} we will  refer to a number of stacks, some of which are pictured in the following diagram:
$$
\xymatrix{
  & {Q}_{X_0}^{\op{det}} \ar[r]\ar[d] & \mc{M}_{X_0}^0\\
 \op{Bun}_{\op{SL}(r)}(X_0)\ar[r]^{\phi} & \beta_{X_0},
}
$$
and defined below:
\begin{definition}\label{biglist}
\begin{itemize}
\item $\op{Q}^{\op{LF}}_{X_0}$, the set of points $[V\tensor\ml^{-1}\to \me] \in \op{Q}_{X_0}^0$, and $\mc{E}$  locally free;
\smallskip
\item $\op{Q}^{\op{det}}_{X_0}$, the set of points  $[V\tensor\ml^{-1}\to \me] \in \op{Q}^{\op{LF}}_{X_0}$  such that $\op{det}(\me)$ is trivializable;
\smallskip
\item  $\overline{\op{Q}}^{\op{det}}_{X_0}$, the closure of $\op{Q}^{\op{det}}_{X_0}$ in $\op{Quot}_{X_0}(V\tensor \mc{L}^{-1},\op{P},1)$;
\smallskip
\item $\mc{M}_{X_0}$,  the normalization of $\overline{\op{Q}}^{\op{det}}_{X_0}$;
\smallskip
\item $\mc{M}_{X_0}^0$, the inverse image in $\mc{M}_{X_0}$ of $\op{Q}_{X_0}^0$; and
\smallskip
\item $\beta_{X_0}$ is the moduli stack of vector bundles with trivializable determinant on $X_0$, (For any $T$ we consider vector bundles $\mathcal{E}$ on $X_0\times T$, such that Zariski locally on $T$, the determinant of $\mathcal{E}$ is the pull back of a line bundle from $T$.)
\end{itemize}
\end{definition}

Here we note that the Picard variety of $X_0$ is a separated scheme (disjoint union of open quasi-projective subschemes), see e.g.,
\cite[Corollary 4.18.3]{Kleiman}. This makes the locus of vector bundles with trivializable determinant a closed condition in families.

\begin{remark}
Recall that we refer to the restriction of $T(\mathcal{A})$ to $\op{Q}^{det}_{X_0}$ as $T(\mathcal{A})$. We also refer to it's
 pullback  to  $\mc{M}_{X_0}$ along the normalization map as $T(\mathcal{A})$.
\end{remark}

\begin{definition}\label{XA}We define $\mc{X}(\vec{a})$ to be the GIT quotient
$$\mc{X}(\vec{a})\cong \mc{M}_{X_0}\git_{T(\mc{A})} \op{PGL}(V).$$
\end{definition}

\begin{lemma}\label{LFSmooth}
\begin{enumerate}
\item[(a)] $\op{Q}^{\op{LF}}_{X_0}$ is  a smooth variety.
\item[(b)] $\op{Q}^{\op{det}}_{X_0}$ is  a smooth variety, and  is open in $\mc{M}_{X_0}$.
\end{enumerate}
\end{lemma}

\begin{proof}
Recall that for $x=[V\tensor \mc{L}^{-1} \twoheadrightarrow \mc{E}]\in \op{Quot}_{\mc{L},\op{P},1}(\ml)$, one has that the tangent space $\op{T}_{x}(\op{Quot}_{\mc{L},\op{P},1}(\ml)) =\underline{\mc{H}om}(\mc{K},\mc{E})$, where $\mc{K}=\ker(V\tensor \mc{L}^{-1} \twoheadrightarrow \mc{E})$. We have $$0 \to \mc{K} \to V\tensor \mc{L}^{-1} \twoheadrightarrow \mc{E} \to 0.$$

 Now pick a point $x=[V\tensor \mc{L}^{-1} \twoheadrightarrow \mc{E}]$ in $\op{Q}^{\op{LF}}_{X_0}$.
The kernel $\mathcal{K}$ is locally free since both $V\tensor \mc{L}^{-1}$ and $\mc{E}$ are vector bundles.
We have the induced short exact sequence of sheaves:
\begin{equation}\label{jacoby}
0 \to \underline{\mc{H}om}(\mc{E},\mc{E}) \to \underline{\mc{H}om}(V\tensor \mc{L}^{-1}, \mc{E}) \twoheadrightarrow \underline{\mc{H}om}(\mc{K}, \mc{E}) \to 0.
\end{equation}
Since $x\in \op{Q}^{\op{LF}}_{X_0} \subset \op{Q}_{X_0}^0$, we have $V\overset{\cong}{\to} \op{H}^0(X_0,\mc{E}\tensor \mc{L})$, and hence $\op{H}^1(X_0, \mc{E}\tensor \mc{L})=0$.  This gives that $H^1(X_0,\underline{\mc{H}om}(V\tensor \mc{L}^{-1}, \mc{E}))=0$. By \eqref{jacoby},
$\op{H}^1(X_0,\underline{\mc{H}om}(\mc{K}, \mc{E}) )=0$, and hence $x$ is a smooth point of $\op{Q}^{\op{LF}}_{X_0}$. This proves (a).

Consider the map
$\op{Q}^{\op{LF}}_{X_0} \to \op{Jac}(X_0)$, given by taking $x=[V\tensor \mc{L}^{-1} \twoheadrightarrow \mc{E}]$ to  $\op{det}(\mc{E})$.  We will show that  the induced
 map on tangent spaces $\op{T}_{x}(\op{Q}^{\op{LF}}_{X_0}) \to \op{T}_{\op{det}(\mc{E})} (\op{Jac}(X_0))$ is surjective.

By \eqref{jacoby},
$$\op{T}_{x}(\op{Q}^{\op{LF}}_{X_0})=\op{H}^0(X_0, \underline{\mc{H}om}(\mc{K}, \mc{E})) \to \op{H}^1(X_0, \underline{\mc{H}om}(\mc{E},\mc{E})),$$
is surjective.  Now, composing with the (split) trace map to $\op{H}^{1}(X_0,\underline{\mc{H}om}(\op{det}(\mc{E}),\op{det}(\mc{E})))$ gives (b).

\end{proof}

\begin{lemma}\label{W1}$\op{H}^0(\mc{M}_{X_0}^0,\mathcal{L}_{\mg})^{\op{GL}(V)}\hookrightarrow \op{H}^0(\beta_{X_0},\mathcal{L}_{\mg}).$
\end{lemma}

\begin{proof}  First note that  ${Q}^{\op{det}}_{X_0} \subset  \op{Q}_{X_0}^0$. Moreover, by Lemma \ref{LFSmooth}, ${Q}^{\op{det}}_{X_0}$ is smooth, and so the inverse image of ${Q}^{\op{det}}_{X_0}$ in the normalization
$\mc{M}_{X_0}$ of $\overline{\op{Q}}^{\op{det}}_{X_0}$ is isomorphic to ${Q}^{\op{det}}_{X_0}$.  This gives an inclusion
$${Q}^{\op{det}}_{X_0}\hookrightarrow \mc{M}_{X_0}^0.$$
We may therefore restrict sections of $\mathcal{L}_{\mg}$ on $\mc{M}_{X_0}^0$ to ${Q}^{\op{det}}_{X_0}$, to get a map:
$$\op{H}^0(\mc{M}_{X_0}^0,\mathcal{L}_{\mg})^{\op{GL}(V)} \hookrightarrow \op{H}^0({Q}^{\op{det}}_{X_0},\mathcal{L}_{\mg})^{\op{GL}(V)}.$$
Let $\beta^0(X_0)$ be the moduli stack  parameterizing vector bundles $\mc{E}$ on $X_0$ with trivializable determinant, such that $H^1(X_0,\mc{E}\tensor \ml)=0$ and $\mc{E}\tensor \ml$ is globally generated. In particular, $\beta^0(X_0)\cong {Q}^{\op{det}}_{X_0}/\op{GL}(V)$,  and $\beta^0(X_0)\subset \beta_{X_0}$. We can assume that our line bundle $\ml$ is sufficiently ample to have the codimension of the complement of $\beta^0(X_0)$ in $\beta(X_0)$ to be at least two,
and therefore:
$$\op{H}^0({Q}^{\op{det}}_{X_0},\mathcal{L}_{\mg})^{\op{GL}(V)}=\op{H}^0(\beta^0(X_0),\mathcal{L}_{\mg})=\op{H}^0(\beta_{X_0},\mathcal{L}_{\mg}).$$
\end{proof}

\begin{lemma}\label{W2}Let $\mc{G}$ be a vector bundle on $X_0$ with the property that if $\mc{E}$ is any  vector bundle on $X_0$ with trivializable determinant, then $\chi(X_0,\mc{E}\tensor \mc{G})=0$.  Then. $\phi^*\mathcal{L}_{\mc{G}}=\mathcal{D}^{\tensor\ell}$, $\ell=\rk\mg$, and  the pull back map $$\op{H}^0(\beta_{X_0},\mathcal{L}_{\mg})\to \op{H}^0(\op{Bun}_{\op{SL}(r)}(X_0),\mathcal{D}^{\tensor \ell})$$ is an isomorphism.
\end{lemma}

\begin{proof}
For the assertion about $\phi^*\mathcal{L}_{\mc{G}}$, use Lemma \ref{repeatedly}.

Now,
$\op{Bun}_{\op{SL}(r)}(X_0)$ is the moduli stack of pairs $(\mathcal{E},\theta)$ where $\mathcal{E}$ is a vector bundle of rank $r$ and $\vartheta:\wedge^r\mathcal{E}\leto{\sim}\mathcal{O}_{X_0}$ is a trivialization of the determinant of $\mathcal{E}$.   Therefore there is a natural surjective (on points) map $\op{Bun}_{\op{SL}(r)}(X_0)\to \beta_{X_0}$. This gives the map in the statement of the lemma. For the surjection given a vector bundle $\mathcal{E}$ with trivializable determinant, the trivialization $\vartheta$ is unique up to  the action of $\Bbb{C}^*$. Now $\Bbb{C}^*$ acts trivially on the determinant of cohomology of $\mathcal{E}\tensor\mathcal{G}$ and hence we get the same element of determinant of cohomology of $\mathcal{E}\tensor\mathcal{G}$ independently of the choice of $\vartheta$. This argument works in families as well, and hence the lemma is proved.
\end{proof}

\subsection{Proof of Proposition \ref{OneIntro}}\label{OIProof}
\begin{proof}Given $\vec{a}$, and $\mc{L}$ an ample line bundle on $X_0$, we let $\mc{A}$ be an ample line bundle on $X_0$ as described in Def \ref{ADef}.  Let $T(\mathcal{A})$ on $\mc{M}^{0}_{X_0}$ be the pull back of  $\ml_{\mathcal{G}}$ from the moduli-stack of torsion free sheaves, defined in Section \ref{twist}.   By Proposition \ref{simpsonian}, $\mc{M}^{0}_{X_0}$  contains
$\mc{M}_{X_0}^{ss}$.   We denote the restriction of $T(\mc{A})$ to $\mc{M}_{X_0}^{ss}$ by $T(\mc{A})$, and recall that $\mc{X}(\vec{a})$ is the GIT quotient
$$\mc{X}(\vec{a})\cong \mc{M}_{X_0}\git_{T(\mc{A})} \op{PGL}(V) = \mc{M}_{X_0}^{ss}/\op{PGL}(V).$$
The variety  $\mc{X}(\vec{a})$ maps to the GIT quotient $\op{Quot}_{\mc{L},\op{P},1}(\ml)\git_{T(\mc{A})} \op{PGL}(V)$. On the latter space, a multiple of $\mc{L}_{\mc{G}}$ will descend (this amounts to replacing $\mathcal{A}$ by a multiple), and we will denote by $\mc{L}_{\mc{G}}$ the pull back to $\mc{X}(\vec{a})$.  Then
$$\op{H}^0(\mc{X}(\vec{a}), \mc{L}_{\mc{G}})=\op{H}^0(\mc{M}^{ss}(\mc{L})/\op{PGL}(V),\mc{L}_{\mc{G}})=\op{H}^0(\mc{M}_{X_0}^{ss}, T(\mathcal{A}))^{\op{GL}(\op{V})}.$$
We recall that $\mc{M}_{X_0}^0$ is the inverse image in $\mc{M}_{X_0}$ of $\op{Q}_{\mc{L}}^{0}(r)$, so $\mc{M}_{X_0}^0$ is an open, irreducible, normal subvariety,  So by Lemma \ref{extendo}, one has $\op{H}^0(\mc{M}_{X_0}^{ss}, T(\mathcal{A}))^{\op{GL}(\op{V})}=\op{H}^0(\mc{M}_{X_0}^0, T(\mathcal{A}))^{\op{GL}(\op{V})}$.
Finally, by Lemmas \ref{W1} and \ref{W2}, one has $\op{H}^0(\mc{M}_{X_0}^0, T(\mathcal{A}))^{\op{GL}(\op{V})}\hookrightarrow \op{H}^0(\beta_{X_0},T(\mathcal{A})) \overset{\cong}{\to}\op{H}^0(\op{Bun}_{\op{SL}(r)}(X_0),\phi^*\mathcal{L}_{\mg})$.
\end{proof}
\subsection{}
The GIT quotient  $\mc{X}(\vec{a})$ maps to the GIT quotient of
$\op{Quot}_{X_0}(V\tensor \mc{L}^{-1},\op{P})$  by $\op{GL}(V)$ for the linearization $T(\mathcal{A})$. It can be shown, using \cite{simpson}, that
the latter  is the same as the moduli space of semistable torsion free sheaves for the weights $\vec{a}$. We omit the proof since it does not play a role in our assertions, although it does give a heuristic as to how one may  consider $\mc{X}(\vec{a})$ as the compactified moduli space of vector bundles with trivial determinant using the polarization $\vec{a}$.

 We indicate why a point $x=[V\tensor \mc{L}^{-1} \twoheadrightarrow \mc{E}]\in \op{Quot}_{X_0}(V\tensor \mc{L}^{-1},\op{P})$ semistable for the linearization $T(\mathcal{A})$ is semistable for the weights $\vec{a}$. Assume the contrary, We can assume that the set of $\me$ is a priori bounded (see Section \ref{last}, here we look at the worst quotient of $\me$ for the polarization $\ml_0$). Now we may pick contradictors of semistability which remain bounded. We can assume that such $\mf\subset \me$ are generated by global sections after tensoring with $\ml$ (and $h^1(X_0,\mf\tensor\ml)=0)$. Let $H=H^0(X_0,\mf\tensor\ml)\subset V=H^0(X_0,\me\tensor\ml)$.  Apply inequality \eqref{nom} to get an inequality resulting from the semistability of $x$. By Lemma \ref{pastel}, we obtain a contradicting inequality since $\mf\tensor\ml$ should  contradict the $\vec{\alpha}$-semistability ($\alpha_i=\deg \mathcal{A}_i$) of $\me\tensor\ml$ (since $\mf$ contradicts the $\vec{a}$ semistability of $\me$).

\section{A factorization assertion}\label{FA}
In this section we prove that ${\op{H}}^0\big({\op{Parbun}}_{{\op{G}}}(X_0, \vec{p}), \mc{L}_{{\op{G}}}(X_0,\vec{p})^{\tensor m}\big)$ is a direct sum
\begin{equation}\label{decompEq}
{\op{H}}^0\Big({\op{Parbun}}_{{\op{G}}}(X_0, \vec{p}), \mc{L}_{{\op{G}}}(X_0,\vec{p})^{\tensor m}\Big)\cong \bigoplus_{\lambda}{\op{H}}^0\Big({\op{Parbun}}_{{\op{G}}}(X, \vec{p}), \mc{D}_{m,\lambda}\Big),
\end{equation}
where $X$ is the normalization of $X_0$.  The line bundles $\mc{D}_{m,\lambda}$ and weights $\lambda$ are precisely described in Lemma \ref{Lemma3}, after sufficient notation is given.  The arguments used here are inspired by  \cite[Proof of Theorem 2.4]{Faltings} and \cite[Page 41]{Teleman}.
\subsection{Notation}    Let $X_0$ be a reduced curve with at worst ordinary double points defined over an algebraically closed field, and let $\nu: X \longrightarrow X_0$
be its normalization.   The curve $X$, which may have more than one component, is  smooth,   and given any set of smooth marked points $\vec{p}=\{p_1,\ldots,p_n\}$ on $X_0$
we put $\vec{p}=\{\nu^{-1}(p_1),\ldots,\nu^{-1}(p_n)\}$.  We denote the set of nodes of $X_0$ by $S$, and   $\nu^{-1}(s)=\{\sas,\sbs\}$ for $s\in S$. Let  $\vec{\lambda}=(\lambda_1,\ldots, \lambda_n)$ is an $n$-tuple of dominant integral weights for $\mathfrak{g}$ at level $\ell$. Let $\op{Parbun}_{\op{G}}(X_0,\vec{p})$ denote the moduli stack parameterizing tuples $(\me,\gamma_1,\dots, \gamma_n)$ where $\me$ is a principal $G$-bundle over $X_0$, and $\gamma_i\in\mathcal{E}_{p_i}/B$. Here $B$ is a fixed Borel subgroup of $G$.
  Let
$$p: \op{Parbun}_{\op{G}}(X_0,\vec{p})\rightarrow \op{Parbun}_{\op{G}}(X,\vec{p}),$$
be the representable morphism of stacks given by pullback along $\nu$. For each irreducible component  $X_{i}$ of $X$ one has
 $$g_{i}: \op{Parbun}_{\op{G}}(X,\vec{p}) \longrightarrow \op{Parbun}_{\op{G}}(X_{i},\vec{p}(X_i)),$$
 given by restriction to $X_{i}$, and where  $\vec{p}(X_i)$ denotes the set of marked points $\{p_j \in X_i\}$.    If $X$ is irreducible,  $g_{i}=g_1=\op{id}$.

\begin{definition}\label{LX_0}
\begin{enumerate}
 \item  $\mc{L}_{\op{G}}(X,\vec{p}): =\bigotimes_{i\in I} g_{i}^{\star}(\mc{L}_{\op{G}}(X_{i},\vec{p}(X_i)))$,  where the $\mc{L}_{\op{G}}(X_{i},\vec{p}(X_{i}))$ are the line bundles from \cite[Theorem, p 499]{LaszloSorger} given on $\op{Parbun}_{\op{G}}(X_{i},\vec{p}(X_i))$, associated to the level $\ell$ and those weights $\lambda_i$ at the points on the component $X_{i}$.
 \item  $\mc{L}_{\op{G}}(X_0,\vec{p}):=p^*\mc{L}_{\op{G}}(X,\vec{p})$.
 \item For $G=\op{SL}(r)$ and no marked points, $$\mc{L}_{\op{G}}(X,\vec{p})=\mathcal{D}^{\otimes \ell}, \ \mbox{ and } \ p^*(\mathcal{D}^{\otimes \ell})=\mathcal{D}^{\otimes\ell},$$
  see Eq \eqref{long}.
\end{enumerate}
\end{definition}

\begin{definition} Let $\mathcal{E}$ be the universal bundle over ${\op{Parbun}}_{{\op{G}}}(X, \vec{p})$, and for $x \in X$, let $\mathcal{E}_x$  be the induced principal $G$-bundle over
${\op{Parbun}}_{{\op{G}}}(X, \vec{p})$.  Given a weight $\lambda$, and its associated irreducible representations $V_{\lambda}$, we let
\begin{equation}\label{Fiber}
\mathcal{E}_{\lambda}^x = \mathcal{E}_x \times_{{\op{G}}} V_{\lambda},
\end{equation}
denote the corresponding vector bundle  over ${\op{Parbun}}_{{\op{G}}}(X, \vec{p})$.
\end{definition}

\begin{remark}\label{PeterWeyl}
In the proofs of Lemmas \ref{Lemma3}, and \ref{LiftLemma}, and in Construction \ref{StepOnePropSI},  we use the algebraic Peter-Weyl Theorem \cite[Theorem 27.3.9]{TY},  which says
$$\bigoplus_{\lambda } V_{\lambda}^* \otimes_{k} V_{\lambda} \cong \mathcal{O}_{{\op{G}}}({\op{G}}), \ \ v_{\lambda}^*\tensor  v_{\lambda}\mapsto f_{v_{\lambda}^*\tensor  v_{\lambda}},$$
where $f_{v_{\lambda}^*\tensor  v_{\lambda}}(g)=v_{\lambda}^*(g\cdot v_{\lambda})$.
\end{remark}

\subsection{The precise statement and proof}
\begin{lemma}\label{Lemma3} For $p: {\op{Parbun}}_{{\op{G}}}(X_0, \vec{p}) \longrightarrow {\op{Parbun}}_{{\op{G}}}(X, \vec{p})$,
given by pulling back along the normalization $\nu: X \rightarrow X_0$,
\begin{enumerate}
\item[(1)] $$p_*(\mathcal{O}) = \bigoplus_{\lambda}\bigl(\bigotimes_{s\in S} \mathscr{E}_{\lambda_{\sas}}^{\sas} \otimes \mathscr{E}_{\lambda_{\sbs}}^{\sbs}\bigr).$$ Here $\lambda$ ranges over all assignments $x\mapsto \lambda_x$, where  $x\in \nu^{-1}(S)$, $\lambda_x$ is a dominant integral weight, such that $\lambda_{\sas}$ is dual to $\lambda_{\sbs}$ for all $s\in S$.
\item[(2)]${\op{H}}^0\Big({\op{Parbun}}_{{\op{G}}}(X_0, \vec{p}), \mc{L}_{{\op{G}}}(X_0,\vec{p})^{\tensor m}\Big)$ is a direct sum $$\bigoplus_{\lambda}{\op{H}}^0\Big({\op{Parbun}}_{{\op{G}}}(X, \vec{p}), \mc{L}_{{\op{G}}}(X,\vec{p})^{\tensor m} \otimes \bigl(\bigotimes_{s\in S} \mathscr{E}_{\lambda_{\sas}}^{\sas} \otimes \mathscr{E}_{\lambda_{\sbs}}^{\sbs}\bigr)\Big).$$ Here $\lambda$ ranges over the same set as in (1).
\item[(3)] In (2) above, the summand corresponding to $\lambda$ is zero if the level of $\lambda$ is not $m$ (i.e., if $(\lambda_x,\theta)>m$ for some $x\in \nu^{-1}(S)$).
\item[(4)] In the case of $G=\op{SL}(r)$ and no marked points, the sum $\sum_{x\in \nu^{-1}(S)\cap X_i} \lambda_x$ is in the root lattice for non zero summands.
\end{enumerate}
\end{lemma}

\begin{proof}
For simplicity, assume $X_0$ has a single node $\xzero$ and $\nu^{-1}(\xzero)=\{\xone,\xtwo\}$.

For $(1)$, we first describe the map between the two vector spaces we'll show the same in $(2)$.  This will give a map between the vector bundles in $(1)$, which we  show it is an isomorphism by choosing appropriate trivializations.

To simplify the argument, we will also assume there are no marked points.
 Given an element $\sigma \in \op{H}^0\big(\op{Bun}_{\op{G}}(X),\mc{L}_{\op{G}}(X)^{\tensor m}\otimes  \mathscr{E}_{\lambda^*}^{\xtwo} \otimes \mathscr{E}_{\lambda}^{\xone}\big)$, one can construct a section of $\mc{L}_{\op{G}}(X_0)^{\tensor m}$ on $\op{Bun}_{\op{G}}(X_0)$ as follows.  For this, let $E_0$ be a principal $\op{G}$ bundle on $X_0$, so that $\nu^*E_0=E$ is a $\op{G}$-bundle on $X$.  In particular, $E|_{\xone}\cong \op{G}$, and by trivializing at $\xtwo$ we can obtain an element $g \in \op{G}$.  Now $\mc{L}_{\op{G}}(X_0)\cong p^*\mc{L}_{\op{G}}(X)$, and we have the isomorphism on fibers:
$$\Big(\mc{L}_{\op{G}}(X)^{\tensor m}\otimes  \mathscr{E}_{\lambda^*}^{\xtwo} \otimes \mathscr{E}_{\lambda}^{\xone}\Big)|_{\mc{E}}=\Big(\op{det}\big(\op{H}^*(X,\mc{E})\big)\Big)^{\tensor m}\otimes V_{\lambda^*} \otimes V_{\lambda}\cong \Big(\op{det}\big(\op{H}^*(X_0,\mc{E}_0)\big)\Big)^{\tensor m}\otimes V_{\lambda^*} \otimes V_{\lambda},$$
where $\mc{E}$ and $\mc{E}_0$ are the vector bundles obtained from $E$ and $E_0$ by contracting with the standard representation $G\to \op{GL}(V)$.  Moreover, under this isomorphism, $\sigma|_{E}$ corresponds to an element $s\tensor v^* \tensor v$, where
$$s\in \Big(\op{det}\big(\op{H}^*(X_0,\mc{E}_0)\big)\Big)^{\tensor m}, \ v^*\in V_{\lambda^*},  \mbox{ and } \ v\in V_{\lambda}.$$
  To obtain a section of $\mc{L}_{\op{G}}(X_0)$ at $E_0$, we  take  $s \tensor v^*( g\cdot v)$.  Now to see this map is an isomorphism we do the following: Since $\mathscr{E}_{\lambda}^{\xone}  \cong \mathscr{E}_{\lambda}^{\xtwo}$ and $(\mathscr{E}_{\lambda}^{x})^* \cong \mathscr{E}_{\lambda^*}^{x}$,
$$\bigoplus_{\lambda } \mathscr{E}_{\lambda^*}^{\xtwo} \otimes \mathscr{E}_{\lambda}^{\xone} \cong \bigoplus_{\lambda } \mathscr{E}_{\lambda^*}^{\xone} \otimes \mathscr{E}_{\lambda}^{\xone}
=\bigoplus_{\lambda } (\mathscr{E}_{\lambda}^{\xone})^* \otimes \mathscr{E}_{\lambda}^{\xone}.$$

Recall that given a weight $\lambda$, and its associated irreducible representations $V_{\lambda}$, we let
\begin{equation}\label{FiberUp}
\mathcal{E}_{\lambda}^x = \mathcal{E}_x \times_{{\op{G}}} V_{\lambda},
\end{equation}
be the corresponding bundle of \ ${\op{G}}$-representations over ${\op{Bun}}_{{\op{G}}}(X_0, \vec{p})$.
Therefore, at a point of ${\op{Bun}}_{{\op{G}}}(X, \vec{p})$ corresponding to a particular principal ${\op{G}}$-bundle $\mc{E}$
$$\Big(\bigoplus_{\lambda} (\mathscr{E}_{\lambda}^{\xone})^* \otimes \mathscr{E}_{\lambda}^{\xone}\Big)|_{\mc{E}} = \bigoplus_{\lambda} (\mc{E}_{\lambda}^{\xone})^* \otimes \mc{E}_{\lambda}^{\xtwo}.$$
Now choosing trivializations of $\mc{E}$ at $\xone$ and $\xtwo$, we have the three identifications:
$$\mc{E}_{\lambda}^{\xone} \cong V_{\lambda},  \  \mc{E}_{\lambda^*}^{\xtwo} \cong V_{\lambda^*}, \ \mbox{ and } \ p_*(\mathcal{O})|_{E}=\mathcal{O}(G),$$
giving the map $T$ is an isomorphism on this trivialization.
By the algebraic Peter-Weyl Theorem (See Remark \ref{PeterWeyl}),  one has that  the sum
$$\bigoplus_{\lambda } V_{\lambda}^* \otimes_{k} V_{\lambda} \cong \mathcal{O}_{{\op{G}}}({\op{G}}).$$
The resulting map $(\mc{E}_{\lambda}^{\xone})^* \otimes \mc{E}_{\lambda}^{\xtwo} \rightarrow p_*(\mathcal{O})|_{E}$ is independent of the choice of trivialization.  The argument with parabolic structures is the same.

\smallskip
To prove Part $(2)$, by Def \ref{LX_0}, and the assertion of Part $(1)$
\begin{multline}
{\op{H}}^0\Big({\op{Parbun}}_{{\op{G}}}(X_0, \vec{p}), \mc{L}_{{\op{G}}}(X_0,\vec{p})^{\tensor m}\Big)
={\op{H}}^0\Big({\op{Parbun}}_{{\op{G}}}(X, \vec{p}), \mc{L}_{{\op{G}}}(X,\vec{p})^{\tensor m} \otimes p_*(\mathcal{O})\Big)\\
={\op{H}}^0\Big({\op{Parbun}}_{{\op{G}}}(X, \vec{p}), \mc{L}_{{\op{G}}}(X,\vec{p})^{\tensor m} \otimes \big(  \bigoplus_{\lambda} \mathscr{E}_{\lambda^*}^{\xtwo} \otimes \mathscr{E}_{\lambda}^{\xone} \big)\Big).
\end{multline}

\smallskip

Part (4) follows from the action of the center of $G$ on spaces of sections. To prove Part $(3)$, we recall a standard argument  that since $X$ is smooth,
\begin{equation}\label{quantity}
{\op{H}}^0\Big({\op{Parbun}}_{{\op{G}}}(X, \vec{p}), \mc{L}_{{\op{G}}}(X,\vec{p})^{\tensor m} \otimes \mathscr{E}_{\lambda^*}^{\xtwo} \otimes \mathscr{E}_{\lambda}^{\xone}\Big)=0
\end{equation}
if $(\lambda,\theta)>m$.  See e.g., \cite[Lemma 6.5]{BK}, where the level is $2g^*m$.
 We add on $\xone$ and $\xtwo$ to the collection of points $\vec{p}$, and get a new set of points $\vec{q}$. We can write  \eqref{quantity} as sections of a line bundle over a new parabolic moduli space, here $\mathcal{L}'$ includes contributions at $\xone$ and $\xtwo$ via the Borel-Weil theorem:
 \begin{equation}\label{quantity2}
{\op{H}}^0\Big({\op{Parbun}}_{{\op{G}}}(X, \vec{q}), \mc{L}'\Big)
\end{equation}
The first proof given in \cite{BK} uses the Borel-Weil theorem on affine flag varieties, due to S. Kumar and O. Mathieu (and does not use the theory of conformal blocks). The Iwahori affine Grassmannian is a quotient $G(\Bbb{C}((t)))/I$ where $I\subset G(\Bbb{C}[[t]])$ is the subset of elements which lie in $B$ when  $t$ is set equal to $0$. This is the moduli space of principal bundles on $X$ trivialized outside of a point $x$, equipped with a full flag structure at $x$. A product of these dominates (indeed, surjects onto) ${\op{Parbun}}_{{\op{G}}}(X, \vec{q})$. By a basic calculation (see \cite{BK}), the pull back of $\mc{L}'$ to this product of Iwahori Grassmannians can be identified. The line bundle breaks up into an external  product of line bundles, and the line bundle corresponding to $\xone$ is not dominant and hence we get the desired vanishing. Note that this vanishing  is an affine generalization of the classical vanishing $H^0(G/B,L_{\lambda})=0$ if $\lambda:B\to \Bbb{C}^*$ is an integral weight which is not dominant.
\end{proof}
\begin{remark}
The vanishing assertion in Lemma \ref{Lemma3} is proved again using the theory of conformal blocks in the next section (see Remark \ref{YesBlocks}). Assertion (4) follows also from the theory of conformal blocks.
\end{remark}
\subsection{Normalizations and determinant of cohomology}
For a vector bundle $\mc{E}_0$ on $X_0$ with trivialized determinant, we have an exact sequence
\begin{equation}\label{sess1}
0\to \mc{E}_0\to \nu_*\nu^*\mc{E}_0\to \oplus s (\mc{E}_0)_s\to 0
\end{equation}
and the given triviality of the determinant of $\me_0$, we can write
\begin{equation}\label{long}
\mathcal{D}(X,\nu^*(\mc{E}_0))= \mathcal{D}(X_0,\nu_*\nu^*(\mc{E}_0)) =\mathcal{D}(X_0,\mc{E}_0).
\end{equation}

\section{Surjectivity of $F$: The proof of Proposition \ref{SomethingIntro}}\label{SomethingIntroSection}
In Proposition \ref{OneIntro} we proved that there are natural inclusion maps
\begin{equation}\label{NIMs}
\op{H}^0(\mc{X}(\vec{a}), \mc{L}_{\mc{G}})\overset{f_{(\vec{a}, \mc{G})}}{\hookrightarrow} \op{H}^0(\op{Bun}_{\op{SL}(r)}(X_0), \mc{D}^{\tensor m}).
\end{equation}
that give rise to the morphism
 $$F: \bigoplus_{(\vec{a},\mc{G})} \op{H}^0(\mc{X}(\vec{a}), \mc{L}_{\mc{G}}) \to \bigoplus_{m\in \mathbb{Z}_{\ge 0}} \op{H}^0(\op{Bun}_{\op{SL}(r)}(X_0), \mc{D}^{\tensor m}).$$
Then, in Section \ref{FA} we proved that for each $m \in \mathbb{Z}$, the summand  $\op{H}^0(\op{Bun}_{\op{SL}(r)}(X_0), \mc{D}^{\tensor m})$ factors into so-called
$\lambda$ components
$\op{H}^0(\op{Bun}_{\op{SL}(r)}(X), \mc{D}_{m,\lambda})$
where $X$ is the normalization of $X_0$, and line bundles $\mc{D}_{m,\lambda}$ and weights $\lambda$ are as described in Lemma \ref{Lemma3}.  Here, in Proposition \ref{SomethingIntro}, we show that given
pair $(\lambda,\ell)$, where $\ell$ is an even integer,
one can find a pair $(\vec{a}, \mc{G})$ so that the given $\lambda$ component  of $\op{H}^0(\op{Bun}_{\op{SL}(r)}(X_0), \mc{D}^{\ell})$ is contained in the image of the map $f_{(\vec{a}, \mc{G})}$.  To say this precisely, we need a small amount of notation.

\subsubsection{Notation}
Let $\nu: X=\cup_{i\in I}X_i \to X_0$ be the normalization of $X_0$, and  $S$ the set of nodes of $X_0$.  For each node $s \in S$, we let $\nu^{-1}(s)=\{\sas,\sbs\}$.   Let $n_i=|\nu^{-1}(S)\cap X_i|$,  be the number of (inverse images of) nodes on the component $X_i$, and $c: \nu^{-1}(S)\to I$ be such that for $x\in \nu^{-1}(s)$, we have $x\in X_{c(x)}$.
We use the notation $\xrminusone=\frac{1}{r}(r-1,-1,\dots,-1)\in \mathfrak{h}$, where $\mathfrak{h}$ is the Lie algebra of the (standard) Cartan subgroup of $\op{SL}(r)$. Note that $\alpha_i(\xrminusone)=\delta_{1,i}$.  Moreover, if  $\mathcal{G}$ is a vector bundle on $X_0$, then $\deg(\mg_i)$ denotes the degree of the restriction of the pullback of $\mathcal{G}$ to the irreducible component of the normalization containing $X_i$.

\begin{definition}\label{CP}Given $\lambda$, a representation for $\op{SL}(r)$ of level $\ell$, we say  $(\vec{a},\mathcal{G})$ is a covering pair for $(\lambda, \ell)$, if $\mathcal{G}$ is a vector bundle on $X_0$ of  rank $\ell$ and $\vec{a}=(a_i)_{i\in I}$, such that for each $i \in I$:
\begin{equation}\label{perturbation}
(g-1) a_i = \frac{\deg(\mg_i)}{\rk(\mg)}=\frac{2g_i-2+n_i}{2}+\sum_{x\in \nu^{-1}(S) \cap X_i} \epsilon_x,
\end{equation}
where $\epsilon_x$ is either element of the following two element set $\{\frac{\lambda_x^*(\xrminusone)}{\ell} -\frac{1}{2},\frac{1}{2}-\frac{\lambda_{x}(\xrminusone)}{\ell}\},$
so that if $\nu^{-1}(s)=\{\sas,\sbs\}$, $\epsilon_{\sas}+\epsilon_{\sbs}=0$ for all nodes $s\in S$.
\end{definition}

\begin{remark}\label{something}
\begin{enumerate}
\item Note that ${\deg(\mg_i)}$ is an integer in Eq \eqref{perturbation} if we assume that $\ell$ is even: This is because  the difference of the possible values for $\epsilon_x$ is an integer multiple of $\frac{1}{\ell}$ (see Lemma \ref{integra} below). Further since we can assume the $\lambda$ summand to be non-zero, the
sum $\sum_{x\in \nu^{-1}{S}}\lambda_{x}$ is in the root lattice and hence $\sum_{x\in \nu^{-1}{S}}\lambda_{x}(\xrminusone)\in \Bbb{Z}$. Replacing a $\lambda_{x}(\xrminusone)$ by $-\lambda^*_{x}(\xrminusone)$ does not change the sum modulo $\Bbb{Z}$.
\item The term $\frac{2g_i-2+n_i}{2}$ in Eq \eqref{perturbation} is the (one half of) degree of the canonical polarization of $X_0$ on the component which corresponds to $X_i$. Therefore we are assigning to $\vec{a}$, a value that is a ``perturbation'' of the canonical polarization on $X_0$.
\item In Section \ref{prius}, we present a variation of the argument which  shows that the finite generation statements can be proved even if the assigned values of $\deg \mathcal{G}_i$ are non-integral, as long as the denominators are bounded.
\end{enumerate}
\end{remark}

\subsubsection{Main result}

\begin{proposition} \label{SomethingIntro} Let $\ell$ be an even integer, and $\lambda$ a representation for $\op{SL}(r)$ of level $\ell$.  If $(\vec{a},\mathcal{G})$ is a covering pair for $(\lambda, \ell)$,
then the image of  the natural map
$$\op{H}^0(\mc{X}(\vec{a}), \mc{L}_{\mc{G}})=\op{H}^0(\mc{M}^0(\ml),\mathcal{L}_{\mathcal{G}})^{\op{GL}(V)}\overset{f_{(\vec{a},\mc{G})}}{\hookrightarrow} \op{H}^0(\op{Bun}_{\op{SL}(r)}(X_0),\phi^*\mathcal{L}_{\mg})=\op{H}^0(\op{Bun}_{\op{SL}(r)}(X_0),\mc{D}^{\ell})$$
contains  the $\lambda$ component of $\op{H}^0(\op{Bun}_{\op{SL}(r)}(X_0),\mc{D}^{\ell})$.
\end{proposition}

\begin{remark}  By  Lemma \ref{integra}, it can be seen that in the definition of $a_i$ given in Eq \eqref{perturbation}, since
$$\frac{\lambda_x^*(\xrminusone)}{\ell} -\frac{1}{2}< \frac{1}{2}-\frac{\lambda_{x}(\xrminusone)}{\ell},$$ any choice of $\epsilon_x$ which lies in the closed interval formed by these points is valid for the conclusion of  Proposition \ref{SomethingIntro} as long as the degrees of $\mg_i$ are integers. For example, for $G=\op{SL}(2)$, one can take $\epsilon_x$ to be the mid-point, which is zero, so that the $a_i$ are the canonical polarization.
\end{remark}
\subsubsection{Outline of the proof of Proposition \ref{SomethingIntro}}
Fix an even integer $\ell$ and a representation $\lambda$ of $\op{SL}(r)$ at level $\ell$.  We take the following steps:
\begin{enumerate}
\item In Construction \ref{StepOnePropSI}, which gives Lemma \ref{Lemma3} in a simpler setting, we show that given an element $\sigma$ in $\op{H}^0(\op{Bun}_{\op{SL}(r)}(X),\mathcal{D}_{m,\lambda})$, how to engineer a section of
$\mc{D}^{\ell}$ over $\op{Bun}_{\op{SL}(r)}(X_0)$.
\item For $R=k[[t]]$, we consider any map $\pi: \op{Spec}(R)\setminus \{0\} \to \op{Bun}_{\op{SL}(r)}(X_0)$.
 In Lemma \ref{ETS}, we show that given $\lambda$, and any $\sigma$ of $\op{H}^0(\op{Bun}_{\op{SL}(r)}(X_0),\mathcal{D}^{\ell})$, suppose that  $a_{\sigma}$ extends to an element $\widetilde{a}_{\sigma} \in \op{H}^0(\op{Spec}(R), \widetilde{\pi}^*(\mc{L}_{\mc{G}}))$ for all possible $\pi$ which extend to a map $\widetilde{\pi}$. Then the image of the natural map
$$\op{H}^0(\mc{M}^0(\ml),\mathcal{L}_{\mathcal{G}})^{\op{GL}(V)}\to \op{H}^0(\op{Bun}_{\op{SL}(r)}(X_0),\phi^*\mathcal{L}_{\mg})=\op{H}^0(\op{Bun}_{\op{SL}(r)}(X_0),\mathcal{D}^{\ell})$$
contains the $\lambda$ component of $\op{H}^0(\op{Bun}_{\op{SL}(r)}(X_0),\mathcal{D}^{\ell})$.
\item We show that every $a_{\sigma} \in \op{H}^0(\op{Bun}_{\op{SL}(r)}(X),\mathcal{D}_{\ell, \lambda})$ extends to an element
 $\widetilde{a}_{\sigma} \in \op{H}^0(\op{Spec}(R), \tilde{\pi}^*( \mc{L}_{\mc{G}}))$, for all possible $\pi$ which extend to a map $\widetilde{\pi}$. This  is explained in Section \ref{StepThree}.
 \end{enumerate}

\subsection{Step (1)}
We recall the following construction which  is a special case of Lemma \ref{Lemma3}.
\begin{construction}\label{StepOnePropSI}
Every section $\sigma \in \op{H}^0(\op{Bun}_{\op{SL}(r)}(X),\mathcal{D}_{m,\lambda})$ gives rise to an associated section in $\op{H}^0(\op{Bun}_{\op{SL}(r)}(X_0),\mc{D}^{\ell}).$
\end{construction}

Given  $\sigma \in \op{H}^0(\op{Bun}_{\op{SL}(r)}(X),\mathcal{D}_{m,\lambda})$, for $E_0 \in \op{Bun}_{\op{SL}(r)}(X_0)$, set $E=\nu^*E_0$.  We have that $\sigma|_{E}$ is an element of
$$\mathcal{D}(X, \nu^*\mc{E}_0)^{\tensor \ell} \tensor  \bigotimes_{s \in S}{\mc{E}}^{\sbs}_{\lambda^*}\tensor \mc{E}^{\sas}_{\lambda},$$
where $\mc{E}_0$ is the vector bundle obtained from $E_0$ by  contracting with the standard representation $\op{SL}(r) \to \op{GL}(V)$.
For each node $s\in S$, choose a trivialization of $(\mc{E}_0)_s$. This trivializes $\mc{E}^{\sas}$ and $\mc{E}^{\sbs}$, and we also get a patching function $g \in \op{SL}(r)$ (which is identity given the above choices). Therefore $\sigma|_{E}$ is a sum of elements  of pure tensors of the form the form $\tau \tensor_{s\in S}v_s^* \tensor v_s$,
where $\tau \in \mathcal{D}(X, \mc{E})^{\tensor \ell}$,  $v_s^* \in {V}_{\lambda^*}$, and $v_s \in V_{\lambda}$ (and $\mc{E}$ is the vector bundle obtained from $E$ by contraction with $V$, and $\nu^*\mc{E}_0=\mc{E}$). The corresponding
section $a_{\sigma}$ at $E_0$ is the corresponding sum of $(\prod_{s\in S} v_s^*(g\cdot v_s))\cdot \tau$. Note that by Eq \eqref{long}, $\mathcal{D}(X, \mc{E})$  can be identified with  $\mathcal{D}(X_0,\mc{E}_0)$.

\subsection{Step (2)}

By contracting with the standard representation, one can associate to every principal $\op{SL}(r)$-bundle $E$ on $X_0$,   a vector bundle,  $\mc{E}$ on $X_0$, such that $\op{det}(\mc{E})$ is  trivializable.  This gives a natural transformation from $\op{Bun}_{\op{G}}(X_0)$ to  $\aleph_{r}(X_0)$.  Composing with $\pi$, we obtain a map:
$$\op{Spec}(R)\setminus \{0\} \overset{\pi}{\to}\op{Bun}_{\op{G}}(X_0) \to \aleph_{r}(X_0),$$
which we continue to call $\pi$.  We will assume that $\pi$ extends to a map $\widetilde{\pi}: \op{Spec}(R) \to \aleph_{r}(X_0)$.
We therefore have an extension of the family of vector bundles, with trivial determinant, parameterized by $\op{Spec}(R)\setminus \{0\}$, to a family of torsion free sheaves parameterized by $\op{Spec}(R)$.
If $\mc{G}$ is a vector bundle on $X_0$ of rank $\ell$,  by Lemma \ref{repeat}, the restriction of $\mc{L}_{\mc{G}}$ on $\aleph_{r}(X_0)$ to $\op{Bun}_{\op{G}}(X_0)$ is $\mc{D}^{\ell}$.

 \begin{lemma}\label{ETS}
Given $\lambda$, and any $\sigma$ in the $\lambda$ component of $\op{H}^0(\op{Bun}_{\op{SL}(r)}(X_0),\mathcal{D}^{\ell})$, suppose that  $a_{\sigma}$ extends to an element $\widetilde{a}_{\sigma} \in \op{H}^0(\op{Spec}(R), \widetilde{\pi}^*(\mc{L}_{\mc{G}}))$ for all possible $\pi$ which extend to a map $\widetilde{\pi}$. Then the natural map
$$\op{H}^0(\mc{X}(\vec{a}), \mc{L}_{\mc{G}})=\op{H}^0(\mc{M}^0(\ml),\mathcal{L}_{\mathcal{G}})^{\op{GL}(V)}\to \op{H}^0(\op{Bun}_{\op{SL}(r)}(X_0),\phi^*\mathcal{L}_{\mg})=\op{H}^0(\op{Bun}_{\op{SL}(r)}(X_0),\mathcal{D}^{\ell})$$
contains in its image, the $\lambda$ component of $\op{H}^0(\op{Bun}_{\op{SL}(r)}(X_0),\mathcal{D}^{\ell})$.
\end{lemma}
\begin{proof}
Clearly $Z=\mc{M}^0(\ml)$ is a normal variety and $\sigma$ gives rise to a section of $\mc{L}_{\mc{G}}$ on the smooth
open subset $\op{Q}^{\op{det}}_{X_0}$. We need to make sure that this section does not have poles on any codimension one point of $Z$. Since $Z$ is smooth in codimension one, poles can be detected along curves.
\end{proof}

\subsection{Step (3)}\label{StepThree}  To carry out the final step of the proof of Proposition \ref{SomethingIntro}, we proceed in four steps:

\begin{enumerate}
\item[(3.1)] Define stack $\Upsilon_r(\nu)$ of Bhosle bundles (Section \ref{BB});
\item[(3.2)] Show any extension $\widetilde{\pi}: \op{Spec}(R) \to \op{Bun}_{\op{SL}(r)}(X_0)$, lifts  to a map $\widetilde{\widetilde{\pi}}$ to $\Upsilon_r(\nu)$ (Lemma \ref{LiftLemma});
\item[(3.3)] Using Bhosle bundles, classify poles of  $a_{\sigma}\in \op{H}^0(\op{Spec}(R)\setminus \{0\}, \pi^*(\mc{L}_{\mc{G}}))$ (Proposition \ref{PoleBreakDown});
\item[(3.4)]  Do pole analysis to show  $a_{\sigma}$ extends to
$\widetilde{a}_{\sigma}\in \op{H}^0(\op{Spec}(R), \widetilde{\pi}^*(\mc{L}_{\mc{G}}))$  (divided into three cases).
\end{enumerate}

\subsubsection{Bhosle bundles}\label{BB}

\begin{definition}\label{BBDef}   A Bhosle bundle of rank $r$  on $\nu: X\to X_0$ is a triple  $(\mc{E}, \vec{q}, \delta)$ where
\begin{enumerate}
\item  $\mc{E}$ is a vector bundle of rank $r$ on $X$;
\item For each  $s\in S$,  and points $\sas$, and $\sbs \in X$ such that $\nu^{-1}(s)=\{\sas,\sbs\}$, there are maps
$$q_{\sas}: \mc{E}|_{\sas}\to Q^s, \mbox{ and } \  q_{\sbs}: \mc{E}|_{\sbs}\to Q^s,$$
such that $\mc{E}|_{\sas} \oplus \mc{E}|_{\sbs} \to Q^s$,  is  an $r$-dimensional quotient; and
\item  $\delta: \mathcal{O}_{X} \rightarrow \op{det}\mc{E}$ is isomorphism.
\end{enumerate}
\end{definition}
\begin{definition}\label{asgpbStackDef} Let  $\Upsilon_r(\nu)$ denote the stack of Bhosle bundles of rank $r$ on $\nu: X\to X_0$.
\end{definition}

\begin{definition/lemma}\label{f} Let  $f:\Upsilon_r(\nu)\to \aleph_r(X_0)$, be the map which takes $(\mc{E}, \vec{q}, \delta)\in \Upsilon_r(\nu)$ to  $\mathcal{K}={\op{ker}}\big(\nu_*\mc{E} \to \bigoplus_{s\in S} \iota_{s}*Q^s\big)$.

\end{definition/lemma}

\subsubsection{Lift of $\widetilde{\pi}$ to stack of Bhosle bundles}

\begin{lemma}\label{LiftLemma} Let $R=k[[t]]$, and any map $\pi: \op{Spec}(R)\setminus \{0\} \to \op{Bun}_{\op{SL}(r)}(X_0)$, such that there is an extension $\widetilde{\pi}\op{Spec}(R) \to \op{Bun}_{\op{SL}(r)}(X_0)$.  There is a lift of $\widetilde{\pi}$ to map to $\Upsilon_r(\nu)$, making the following diagram
\[\xymatrix{
&  \Upsilon_r(\nu) \ar[d]^{f}\\
 \op{Spec}(R)  \ar@{->}[ru]^-{\widetilde{\widetilde{\pi}}}\ar[r]^{\widetilde{\pi}}&  \aleph_r(X_0),
}\]
commute.
\end{lemma}

\begin{proof}
Such a lift clearly exists over $\op{Spec}(R)\setminus \{0\}$.  The pull back of the given family of torsion free sheaves on $X_0$ gives a vector bundle on $X_i\times \op{Spec}(R)\setminus S_i\times \{0\}$ where $S_i=\nu^{-1}(S)\cap X_i$, for each $i \in I$. Each curve $X_i$ is smooth, hence $X_i\times \op{Spec}(R)$ is a regular scheme of dimension two, and  so the push forward of such a vector bundle from $X_i\times \op{Spec}(R)\setminus S_i\times \{0\}$ to $X_i\times \op{Spec}(R)$
gives rise to a vector bundle on $X_i\times \op{Spec}(R)$ (use \cite[Corollary 1.4]{Harty}).

As the Grassmann variety of quotients is proper, one may extend the family of quotients  to $t=0$. We therefore obtain a map $\nu: \op{Spec}(R)\to \Upsilon_r(\nu)$. The resulting family of torsion free sheaves parameterized by
$\op{Spec}(R)$ agrees with the family $\widetilde{\pi}$ outside of the central fiber over $t=0$.   By depth considerations, the families have to coincide (canonically) since both sheaves, being flat over $\op{Spec}(R)$ and relatively torsion free, are canonically the sheaf theoretic push-forwards from $X_0\times \op{Spec}(R)\setminus S\times\{0\}$ (see proof of \cite[Lemma 1.17]{simpson}).

The determinant of $\mc{E}(t)|_{X_i}$ on $X_i$ may develop a zero or pole at $t=0$. However, we can assume the following: The new family has a trivialization that extends to $t=0$, moreover there is  a function $f: I\to \mathbb{Z}_{\geq 0}$, so that for every node $s$ on $X_0$, and $\{\sas,\sbs\}=\nu^{-1}(s)$, the diagram

\[
 \xymatrix{
\mathbb{C} \ar@{->}[rr]^{\displaystyle t^{f(c(\sas))}\delta} \ar@{->}[dd]_{\displaystyle t^{f(c(\sbs))}\delta}
      && \op{det}(\mc{E}(t)_{\sas}) \ar@{->}[dd]^{\displaystyle \op{det}(q_{\sas})}    \\ \\
  \op{det}(\mc{E}(t)_{\sbs}) \ar@{->}[rr]_{\displaystyle \op{det}(q_{\sbs})} && \op{det}(Q^s)
 }
 \]
 commutes.

We then obtain a quadruple $(\mc{E}(t), \vec{q}(t), \delta(t), f)$, where $(\mc{E}(t), \vec{q}(t), t^{f}\delta )$ is a family of Bhosle bundles, such that  $(\mc{E}(t), \vec{q}(t), t^{f}\delta )$  is equivalent to the original family of Bhosle bundles for $t\neq 0$, and if
 \begin{equation}\label{sound}
\mc{K}_t=Ker(\nu_*\mc{E}(t)\to \otimes_{s\in S}i_* Q^s),
\end{equation}
then  for $t\ne 0$, the torsion free sheaf $\mc{K}_t$ will have trivial determinant.
\end{proof}

By Proposition \ref{normalform}, for each $s \in S$, we may choose bases for the vector spaces $\mc{E}(t)|_{\sas}$ and $\mc{E}(t)|_{\sbs}$, and $Q^s$, and represent the maps $q(t)_{\sas}$ and $q(t)_{\sbs}$ by matrices of the form

\[
\begin{bmatrix}
    t^{\alpha^s_1} & \dots  & 0 & 0 &  \dots  & 0\\
   \vdots  &   \ddots & \vdots & \vdots & \ddots &\vdots \\
   0 &\cdots & t^{\alpha^s_{k_s}}&  0 & \cdots & 0 \\
    0 & \dots  & 0& 1&\dots  & 0 \\
   \vdots &  \ddots & \vdots & \vdots & \ddots & \vdots \\
    0 &0& 0 & \dots  & 0&1
\end{bmatrix}
\mbox{ and }
\begin{bmatrix}
    1 &  \dots  & 0 & 0  & \dots  & 0\\
    \vdots &  \ddots  & \vdots & \vdots & \ddots & \vdots \\
    0 & \dots& 1 & 0 & \dots   & 0 \\
     0 &\dots &0 &  t^{\beta^s_{k_{s+1}}} & \dots  & 0 \\
       \vdots&  \ddots & \vdots & \vdots & \ddots & \vdots \\
  0&   \dots  & 0 & 0  & \dots  & t^{\beta^s_{r}}
\end{bmatrix},
\]
where, without loss of generality, $\alpha^s_1 \ge \cdots \ge \alpha^s_{k_s}$, and $\beta^s_{k_{s+1}} \le \cdots \le \beta^s_{r}$.

\subsubsection{Using Bhosle bundles to classify the poles of $a_{\sigma}$}\label{PartThree}
\begin{proposition}\label{PoleBreakDown}
The pole contribution to $a_{\sigma}$ is a product of the following three quantities.
\begin{enumerate}
\item[(1)] A sum over $i\in I$, here $\mathcal{G}_i$ is the pull back of $\mathcal{G}$ under $X_i\to X_0$:
$$\sum_{i\in I} t^{f(i)} \chi(X_i, \mc{E}_i\tensor\mathcal{G}_i).$$
\end{enumerate}
The remaining two terms are sums over nodes $s\in X$. Let $\{\sas,\sbs\}=\nu^{-1}(s)$. Using the notation established above,
these are (for each $s$),
\begin{enumerate}
\item[(2)] A order of pole of a sum of terms of the form $v^*(g_t v)$ where $v\in V_{\lambda_{a_x}}$ and $v^*\in V_{\lambda_{b_x}}^*$, and
    \begin{equation}\label{gluon}
    g_t= t^{\frac{f(\sas)-f(\sbs)}{r}}q(t)_{\sbs}^{-1}q(t)_{\sas}.
    \end{equation}
\item[(3)] $t$ raised to the exponent $-\ell(f(\sas)+ \sum_{j=1}^{k_s} \alpha^s_j)$. This exponent comes from the order of zero (raised to $\ell$) of $t^{f(\sas)}\det q(t)_{\sas}$, which equals in turn the order of zero of  $t^{f(\sbs)}\det q(t)_{\sbs}$. So we could have written this term also as
    $-\ell(f(\sbs)+ \sum_{j=k_s+1}^{r} \beta^s_j)$.

\end{enumerate}
\end{proposition}

\begin{proof}
Form a new family of Bhosle bundles for $t\neq 0$: $(\mc{E}(t), \vec{q}'(t), \delta)$, with
$q(t)'_x= t^{\frac{f(c(x)}{r}}q(t)_x$. Therefore the vector bundles $\mc{E}(t)$, the trivialization $\delta$ of determinants,
 and the spaces $Q^s$ remain the same, only the maps $q(t)$ change. This gives rise of a family of vector bundles $\mathcal{K}'(t)$ for $t\neq 0$ on $X_0$, and therefore $\sigma$ induces an element in the determinant of cohomology of $\mathcal{K}'(t)\tensor \mg$. Now $\mk'(t)$ is isomorphic to $\mk(t)$: The isomorphism is induced by a map
 $\mc{E}(t)\to \mc{E}(t)$ which is multiplication by $t^{\frac{f(i)}{r}}$ on component $X_i$, the identity maps $Q^s\to Q^s$.

 We need to transfer the induced section (from $\sigma$) in $\mathcal{D}(\mk'(t)\tensor\mg)$ to $\mathcal{D}(\mk'(t)\tensor\mg)$. Recall that the determinant of cohomology of $\mathcal{K}(t)$ is (using \eqref{sound})

 $$\bigotimes_{i\in I}\mathcal{D}(\mc{E}_i(t)\tensor\mg)\tensor \bigotimes_{s\in S} \det(Q^s)^{\rk(\mg)}.$$

 This transfer results in Part (1) of Proposition \ref{PoleBreakDown}.

Terms in Parts (2) and (3) of Proposition \ref{PoleBreakDown} arise in the new family $(\mc{E}(t), \vec{q}'(t), \delta)$. A contribution of type (2) is a so-called ``Peter-Weyl" term.

 The gluing data from $\sas$ to $\sbs$ is given by Eq \eqref{gluon}. The third term is a transfer factor from the determinant of cohomology of $\nu^*\mk'(t)$, where the produced section live, to the determinant of cohomology of $\mk(t)$ as
we next explain.

Suppose we are given a Bhosle-bundle $(\mc{E},q,\delta)$, as described in Def \ref{BBDef}, such that $q_x$ maps are all surjections, and the determinants $\delta$  patch together to give rise to a vector bundle $\mc{E}_0$ with trivialized determinant on $X_0$. It is clear that
$\mc{E}=\nu^*\mc{E}_0$, and $\mc{E}_0$ sits in an exact sequence
\begin{equation}\label{sess2}
0\to \mc{E}_0\to \nu_*\mc{E}\to \oplus_s Q^s\to 0.
\end{equation}

We find a map
$\mathcal{D}(\mc{E}_0)^{\ell}\to \mathcal{D}(\mc{E})^{\ell}\tensor \bigotimes_{s\in S} \det(Q^s)^{\ell}$, but by Eq \eqref{long},
we have an isomorphism $\mathcal{D}(\mc{E})^{\ell}\to \mathcal{D}(\mc{E}_0)^{\ell}$.
We claim that the composite map $\mathcal{D}(\mc{E})^{\ell}\to \mathcal{D}(\mc{E}_0)^{\ell}\to \mathcal{D}(\mc{E})^{\ell}\tensor \bigotimes_{s\in S} \det(Q^s)^{\ell}$
is multiplication  by  the following quantity raised to $\ell$, with $x=\sas$
\begin{equation}\label{composite}
\Bbb{C}\leto{\delta_x}\mc{E}_x\leto{\det q_x} Q^s.
\end{equation}

 This follows from the natural map from $(\mc{E}_0)_s$ in Eq \eqref{sess1}, to $Q^s$ in Eq \eqref{sess2}, inducing a map on exact sequences. The map $\Bbb{C}=\det (\mc{E}_0)_s\to \det Q^s$ equals the composite Eq \eqref{composite}. The above claim justifies the third term.

\end{proof}
\subsection{Final step: Calculating contribution of each pole type for proof of Proposition \ref{SomethingIntro}}
\subsubsection{Pole contributions of type (1)} The order of pole in (1) simplifies, using Eq \eqref{perturbation}, to
$$\frac{f(i)}{r}r(\deg(\mg_i)+(1-g_i)\ell)=f(i)\ell(\frac{n_i}{2} + \sum_{x\in \nu^{-1}(S)\cap X_i} \epsilon_x),$$
which can be written as $$\sum_{x\in \nu^{-1}(S)\cap X_i} f(c(x))(\epsilon_x +\frac{1}{2}).$$
Therefore the total contribution of (1) can also be written as a sum over nodes in $X_0$: For every node $s$, we have a sum of $\ell$ times.
$$f(c(\sas))(\frac{1}{2}+\epsilon_{\sas})+f(c(\sbs))(\frac{1}{2}+\epsilon_{\sbs})=\frac{1}{2}(f(c(a_x))+f(c(b_x))+ (\epsilon_{\sas}f(c(\sas))+\epsilon_{\sbs}f(c(\sbs))).$$

Therefore contribution (1) takes the form
\begin{enumerate}
\item[(1)'] A sum over nodes $s\in X_0$, of
$$(\frac{1}{2}(f(c(a_x))+f(c(b_x))+ (\epsilon_{\sas}f(c(\sas))+\epsilon_{\sbs}f(c(\sbs))))\cdot \ell.$$
\end{enumerate}

Our aim now is to show that the contributions to (1)', (2) and (3) from a fixed node, sum to $\leq 0$. Let $\mu=\lambda_{a_x}$.
If $f(c(a_x))$ and $f(c(b_x))$ are both increased by one, then the contribution to (1)' from $s$ increases by $\ell$ since
 $\epsilon_{\sas}+\epsilon_{\sbs}=0$. The contribution to (3) decreases by $\ell$, and (2) remains unchanged. Therefore for the purposes of showing the local contribution at $s$ is $\leq 0$, we may assume that one of the following three cases occurs
 \begin{enumerate}
 \item[(a)] $f(c(a_x))=f(c(b_x))=0$.
 \item[(b)] $f(c(a_x))>0$ and $f(c(b_x))=0$.
 \item[(c)] $f(c(a_x))=0$ and $f(c(b_x))>0$.
 \end{enumerate}
\subsubsection{Pole contributions of remaining types: Case (a)}

In this case only (2) and (3) contribute, and in (3), we have $\ell$ times $-\sum_{j=1}^{k_s} \alpha^s_j=-|\alpha^s|$ and (2) contributes terms of the form $v^*(g_t)v$ with $g_t=q(t)_{\sbs}^{-1}q(t)_{\sas}$ (a diagonal matrix), which, by Lemma \ref{openeye} below, produces a pole of order at most $(\mu_1-\mu_r)|\alpha^s|\leq \ell |\alpha^s|$. Therefore there are no poles in this case.
\subsubsection{Pole contributions of remaining types: Case (b)}
Set $f(c(a_x))=m$. Therefore $m+|\alpha^s|=|\beta^s|$, and the term (1)' for the node $s$ is
$\frac{m}{2}+ (\epsilon_{\sas}m)\cdot \ell =m\ell(\frac{1}{2} +\epsilon_{\sas})$. Note that the determinant of $\me(t)_{a_s}$  has a trivialization that may not agree with the basis given, but since the quotient is invertible, this does not effect the pole calculation.

By Lemma \ref{openeye}, term (3) contributes $-\ell(|\alpha^s|+m)$ and term (2) contributes no more than
$(\mu_1-\mu_r)|\alpha^s| + m \mu(\xrminusone)$.  Therefore the order of pole is no more than $\ell$ times
$$\frac{1}{\ell}(\mu_1-\mu_r)|\alpha^s| -|\alpha^s| -m +  m (\frac{\mu(\xrminusone)}{\ell}  + \frac{1}{2} +\epsilon_{\sas})
=(\frac{1}{\ell}(\mu_1-\mu_r)-1)|\alpha^s| +m(\frac{\mu(\xrminusone)}{\ell} -\frac{1}{2} +\epsilon_{\sas}).$$
Noting that $(\mu_1-\mu_r)\leq \ell$, we therefore need to verify:
$$\epsilon_{\sas}\leq \frac{1}{2}-\frac{1}{\ell}\mu(\xrminusone).$$
\subsubsection{Pole contributions of remaining types: Case (c)}
Set $f(c(b_x))=n$. Therefore $|\alpha^s|=|\beta^s|+n$, and the term (1)' for the node $s$ is
$\frac{n}{2}+ (-\epsilon_{\sas}n)\cdot \ell =n\ell(\frac{1}{2} -\epsilon_{\sas})$. Term (3) contributes $-\ell(|\beta^s|+n)$ and term (2) contributes no more than (by Lemma \ref{openeye})
$(\mu_1-\mu_r)|\beta^s| - (-n) \mu^*(\xrminusone)$. Therefore the order of pole is no more than $\ell$ times
$$\frac{1}{\ell}(\mu_1-\mu_r)|\beta^s| -|\beta^s| -n +  n (\frac{\mu^*(\xrminusone)}{\ell}  + \frac{1}{2} -\epsilon_{\sas})=(\frac{1}{\ell}(\mu_1-\mu_r)-1)|\alpha^s| +n(\frac{\mu^*(\xrminusone)}{\ell} -\frac{1}{2} -\epsilon_{\sas}).$$
Therefore, noting that  $(\mu_1-\mu_r)\leq \ell$, we need to verify that
$$\epsilon_{\sas}\geq \frac{1}{\ell}\mu^*(\xrminusone)-\frac{1}{2}.$$

\subsubsection{Pole contributions of remaining types: Conclusion}
Therefore all in all, we need
\begin{equation}\label{allinall}
\frac{1}{\ell}\mu^*(\xrminusone)-\frac{1}{2}\leq \epsilon_{\sas}\leq \frac{1}{2}-\frac{1}{\ell}\mu(\xrminusone).
\end{equation}
It follows that $\frac{1}{\ell}\mu^*(\xrminusone)-\frac{1}{2}\leq \frac{1}{2}-\frac{1}{\ell}\mu(\xrminusone)$ since by Lemma \ref{integra}
$(\mu+\mu^*)(\xrminusone)= \mu_1-\mu_r\leq \ell$. Therefore either of the two choices for $\epsilon_x$ works for the pole calculations.

\begin{lemma}\label{openeye}
Let $A_t$ and $B_t$ be diagonal matrices with entries in $R$, of the form
\[
\begin{bmatrix}
    t^{\alpha_1} & \dots  & 0 & 0 &  \dots  & 0\\
   \vdots  &   \ddots & \vdots & \vdots & \ddots &\vdots \\
   0 &\cdots & t^{\alpha_{k}}&  0 & \cdots & 0 \\
    0 & \dots  & 0& 1&\dots  & 0 \\
   \vdots &  \ddots & \vdots & \vdots & \ddots & \vdots \\
    0 &0& 0 & \dots  & 0&1
\end{bmatrix}
\mbox{ and }
\begin{bmatrix}
    1 &  \dots  & 0 & 0  & \dots  & 0\\
    \vdots &  \ddots  & \vdots & \vdots & \ddots & \vdots \\
    0 & \dots& 1 & 0 & \dots   & 0 \\
     0 &\dots &0 &  t^{\beta_{k}} & \dots  & 0 \\
       \vdots&  \ddots & \vdots & \vdots & \ddots & \vdots \\
  0&   \dots  & 0 & 0  & \dots  & t^{\beta_{r}}
\end{bmatrix},
\]

Let $m+|\alpha|=|\beta|$, and $|\alpha|=\sum_{j=1}^k \alpha_j$, $|\beta|=\sum_{j=k+1}^r \beta_j$. Let $V=V_{\mu}$ be a representation of $\op{SL}(r)$, and $v\in V, v^*\in V^*$. Set
$$g_t= t^{\frac{m}{r}}A_t B_t^{-1}.$$
Then the order of pole of $v^*(g_t v)$, at $t=0$ is at most

\begin{equation}\label{openeye1}
(\mu_1-\mu_r)|\beta| +m(\mu_r-\frac{1}{r}|\mu|)=(\mu_1-\mu_r)|\beta| -m \mu^*(\xrminusone)
\end{equation}
which may also be written as
\begin{equation}\label{openeye2}
(\mu_1-\mu_r)|\alpha| +m(\mu_1-\frac{1}{r}|\mu|)=(\mu_1-\mu_r)|\alpha| + m \mu(\xrminusone).
\end{equation}
\end{lemma}

\begin{proof}
Assume $\alpha_j$ are increasing and $\beta_j$ are decreasing non negative integers. By highest weight theory, the pole terms are maximized with value (of order of pole) equal to
$$\mu_1(\beta_{r}-\frac{m}{r})+\dots+\mu_{r-k}(\beta_{k+1}-\frac{m}{r})+\mu_{r-k+1}(-\alpha_{k} -\frac{m}{r}) +\dots + \mu_{r}(-\alpha_{1} -\frac{m}{r})$$
which is $\leq$
$\mu_1|\beta|-\mu_r|\alpha| -\frac{m}{r}|\mu|$
we can write this in two different ways as in the statement of the Lemma.
\end{proof}
\begin{lemma}\label{integra}
Let $\mu=(\mu_1\geq\mu_2\geq \dots\geq \mu_r)$ be a dominant integral weight of level $\ell$.
\begin{enumerate}
\item Then
$(\mu^*+\mu)(\xrminusone)=\mu_1-\mu_r \in \Bbb{Z}$.
\item $\frac{\mu^*(\xrminusone)}{\ell} -\frac{1}{2}\leq \frac{1}{2}-\frac{\mu(\xrminusone)}{\ell}$.
\item One of the two numbers in $\{\frac{\mu^*(\xrminusone)}{\ell} -\frac{1}{2},\frac{1}{2}-\frac{\mu(\xrminusone)}{\ell}\}$ is in the range $[-\frac{1}{2},\frac{1}{2}]$.
\end{enumerate}
\end{lemma}
\begin{proof}
The first calculation is immediate. For the second, note that the difference between the two choices for $\epsilon_x$ is less than or equal to one. Their average is
\begin{equation}\label{diamond}
 \frac{1}{2\ell}(\mu-\mu^*)(\xrminusone)=\frac{1}{2\ell}((\mu_1+\mu_{r})-\frac{2}{r}|\mu|).
 \end{equation}

These expressions do not change if we increase all $\mu_i$ by one, and therefore we may assume that $\mu_1\leq \ell$, and $\mu_r=0$. Hence the quantity in \eqref{diamond} is at most  $$\frac{1}{2\ell}((\mu_1-\frac{2}{r}\mu_1)=\frac{\mu_1}{2\ell}(1-\frac{2}{r})\leq \frac{1}{2}(1-\frac{2}{r})= \frac{1}{2}- \frac{1}{r}$$
which gives the desired assertion. The lower bound for $\epsilon$ follows from duality.
\end{proof}

\begin{proposition}\label{normalform}
Let $V,W$ and $Q$ be  vector bundles over $\operatorname{Spec}(k[[t]])$ of rank  $r$. Assume that we are given $\phi:V\to Q$ and $\psi:W\to Q$ so that
\begin{enumerate}
\item[(a)] The resulting map $V\oplus W\to Q$ is surjective (i.e., surjective on fibers at $t=0$).
\item[(b)] $\phi$ and $\psi$ are isomorphisms over $\op{Spec}(k((t)))$.
\end{enumerate}
Then for each $t$, one can choose bases for fibers $V_t$, $W_t$, and $Q_t$, say
$$\{e_1,\dots,e_r\}, \  \ \{f_1,\dots,f_r\}, \  \mbox{and }  \ \{q_1,\dots,q_r\} \ \ (\mbox{respectively})$$ so that for a suitable
index $p$, such that  $1\leq p\leq r$, one has
\begin{enumerate}
\item $\phi(e_i)= q_i$ for $i=1,\dots,p$.
\item $\phi(e_{j})= t^{a_j} q_j$ for $j>p$ and $a_j\geq 0$.
\item $\psi(f_i)= t^{b_i} q_i$ for $i=1,\dots,p$ with $b_i\geq 0$.
\item $\psi(f_{j})= q_j$ for $j> p$.
\end{enumerate}
\end{proposition}
\begin{proof}
For $t\neq 0$, the quotient map in (a) is $V\oplus W$ modulo the graph (up to a sign) of a map $\phi:V\to W$.  We can choose bases for $V$ and $W$ so that the matrix for $V\to W$, for $t\neq 0$ is diagonal with entries $t^{m_i}$ with $m_i$ integers. Indeed, consider $t^m\phi$ for $m>>0$:  Any square matrix over a principal ideal domain has a Smith normal form  (see \cite[Theorem 3.8]{Jake}) $MAN$ with $M,N$ invertible and $A$ diagonal, all of the same size as the original matrix. By  separatedness properties of the Grassmannian, we can now reduce to the case $V$ and $W$ one dimensional.

In this case we need only check the case $\mathcal{O}\to \mathcal{O}$ is multiplication by $t^m$, $m$ a positive integer ($m=0$ is trivial, and $m<0$ can be handled by reversing the roles of $V$ and $W$). In this case the quotient map
is $\mathcal{O}\oplus\mathcal{O}\to Q=\mathcal{O}$ which takes
$(\alpha,\beta)$ to $\alpha- t^m\beta$. This map is of the desired form in the basis $1$ and $-1$ for $V$ and $W$ respectively.
\end{proof}

\section{Finite generation}\label{FG}
Here in Section \ref{MainPoint}, we complete the proof of  Theorem \ref{qmain}.
This argument relies on two preliminary results: Part (c) of Proposition \ref{se2},  and Lemma \ref{verona}, proved here as well.
\subsection{Notation}
Recall from Lemma \ref{WhiteDog} that $T(\mathcal{A})$ on $\mc{M}^0_{X_0}$ is the pull back of  $\ml_{\mathcal{G}}$ from $\aleph(X_0)$, the moduli-stack of torsion free sheaves.

\begin{definition}\label{collection}Let $\mc{L}$ be an ample line bundle on $X_0$, and for every $j \in \{1,\ldots, t\}$, let weights $\vec{a}_j=\{a^j_{i}\}_{i\in I}$ be given.  To this data, for each $j \in \{1,\ldots, t\}$, we can associate an ample line bundle $\mathcal{A}_j$, and  for every vector $\vec{m} =(m_1,\dots,m_t)$, we
define $$\ml^{\vec{m}}=\bigotimes_{j=1}^t  T(\mathcal{A}_j)^{\tensor m_j}.$$
\end{definition}
Note that, on $\mc{M}^0_{X_0}$:
$$\ml^{\vec{m}}=\bigotimes_{i=1}^s  \ml_{\mg_i}^{\tensor m_i}=\ml_{\oplus \mathcal{G}_i^{\oplus m_i}}.$$
We will let $\ell(\vec{m})=\sum_i m_i \rk\mathcal{G}_i$.

\subsection{Preliminary results}

\begin{proposition}\label{se2}
\begin{enumerate}
\item[(a)] The natural restriction map
$\op{H}^0(\mc{M}_{X_0}, \ml^{\vec{m}})^{\op{GL}(V)}\to \op{H}^0(\mc{M}^0_{X_0}, \ml^{\vec{m}})^{\op{GL}(V)}$ is an isomorphism if $m_i\geq 0$ for $i=1,\dots,t$.
\item[(b)] The algebra  $\bigoplus_{\vec{m}\geq \vec{o}} \op{H}^0(\mc{M}^0_{X_0}, \ml^{\vec{m}})^{\op{GL}(V)}$ is finitely generated.
\item[(c)] The algebra
\begin{equation}\label{bigone}
\bigoplus_{\stackrel{\vec{m}\geq \vec{o},}{  \ell(\vec{m}) \text{ is a multiple of } d}} \op{H}^0(\mc{M}^0_{X_0}, \ml^{\vec{m}})^{\op{GL}(V)}
\end{equation}
is finitely generated for any positive integer $d$.
\end{enumerate}
\end{proposition}
To prove Proposition \ref{se2}, we need the following result.
\begin{lemma}\label{EndGoal} An element $\tilde{x}\in \mc{M}_{X_0}\setminus \mc{M}_{X_0}^0$  is not semistable for $\mc{L}^{\vec{m}}$.
\end{lemma}
\begin{proof}Recall from Def \ref{biglist}, that $\mc{M}_{X_0}\to \overline{\op{Q}}^{\op{det}}_{X_0}$ is the normalization,  and so in particular,  $\mc{M}_{X_0}$ is  finite over $\overline{\op{Q}}^{\op{det}}_{X_0}$. Let $x=[V\tensor \mc{L}^{-1} \twoheadrightarrow \mc{E}]\in \overline{\op{Q}}^{\op{det}}_{X_0}\subseteq\op{Quot}_{\mc{L},\op{P},1}(\ml)$ be the image of $\tilde{x}$.

Now by assumption, $x\in \op{Quot}_{X_0}(V\tensor \mc{L}^{-1},\op{P},1)\setminus \op{Q}_{X_0}^0$, and hence by Proposition \ref{simpsonian}, $x$ is not semistable for $T(\mc{A}_i)$. Our witnesses of non-semistability in that proposition are subspaces $H\subseteq V$ which do not depend upon $\mathcal{A}_i$. Therefore, the same one parameter subgroup of $\op{GL}(V)$ renders $x$
 non-semistable for $T(\mathcal{A}_i)$, and hence renders $x$ non-semistable also for a tensor product (with non-negative exponents, not all zero) of the $T(\mathcal{A}_i)$. By GIT (see Theorem I.19, and the comments after Corollary 1.20 on page 48 in \cite{MumBook}),   $\tilde{x}$ is also non semistable for $\mc{L}^{\vec{m}}$.
 \end{proof}

\begin{proof}(of Proposition \ref{se2})
We may apply Lemmas \ref{extendo} and  \ref{EndGoal} and obtain (a).

Now recall that since $Z=\mc{M}_{X_0}$ is a projective variety with an action of a reductive group $\op{GL}(V)$, and  $\ml_1,\dots,\ml_s$ are $\op{GL}(V)$ linearized ample
line bundles on $Z$, so
\begin{equation}\label{fillin}
\bigoplus_{\vec{m}\geq \vec{o}} \op{H}^0(Z, \ml^{\vec{m}})
\end{equation}
is finitely generated as a $\Bbb{C}$-algebra (this is credited to Zariski: See \cite[Lemma 2.8]{HuKeel}). Since $\op{GL}(V)$ is reductive, and acts algebraically on the algebra given in Eq \eqref{fillin},
the algebra of invariants is also finitely generated (Hilbert). This proves (b).

Since Veronese subrings of finitely generated rings are finitely generated, (c) follows from (b).
\end{proof}

\begin{lemma}\label{verona}Let $R_{\bullet}=\oplus R_m$ be a graded integral domain of $A$-modules where $A$ is an excellent integral domain. Assume $R_i$ are finitely generated and free as $A$-modules. Suppose $R^{[d]}_{\bullet}=\oplus_{m\geq 0}R_{md}$ is finitely generated as an $A$-algebra. Then,
$R_{\bullet}$  is finitely generated as an $A$-algebra.
\end{lemma}
\begin{proof}(Standard, see e.g., \cite[Proof of Proposition 2.3, p 151]{FaltingsChai})
Let $L$ be the field of fractions of $R_{\bullet}$, and $K\subseteq L$ the field of fractions of $R^{[d]}_{\bullet}$. We make the following observations
\begin{enumerate}
\item $R_{\bullet}$ is integral over $R^{[d]}_{\bullet}$.
\item $L$ is a finite algebraic extension of $K$. Here assume for simplicity that $R_1\neq 0$ (this will be the case for us). Then $L=K(u_1,\dots,u_s)$ where $u_1,\dots,u_s$ generate $R_1$ as an $A$-module.
\item The integral closure of $R^{[d]}_{\bullet}$ in $L$ is finitely generated as an $R^{[d]}_{\bullet}$-module, and hence Noetherian as a $R^{[d]}_{\bullet}$-module. Now $R_{\bullet}$ is a submodule of the integral closure and hence is finitely generated as a $R^{[d]}_{\bullet}$-module. The module generators of $R_{\bullet}$ together with ring generators of $R^{[d]}_{\bullet}$  generate $R_{\bullet}$ as a ring.

\end{enumerate}
\end{proof}
\subsection{Proof of Theorem \ref{qmain}}\label{MainPoint}
\begin{proof}
For $g \ge 2$, we consider tuples $(\{d_i\}_{i\in I},\ell)$  where the $d_i$ and $\ell$ are integers, and satisfy the following conditions:
\begin{enumerate}
\item  $\sum d_i = (g-1)\ell$.
\item $d_i\geq r_0 \ell$ where $r_0$ is a possibly negative rational number.
\end{enumerate}
The set of such tuples is finitely generated as a semigroup by Gordon's Lemma (see \cite{Gordon} for a proof of Gordon's Lemma).  We pick generators $(\{d^{(j)}_i\}_{i\in I},\ell^{(j)})$ for $j=1,\dots,t$, and set $a^{(j)}_i=\frac{d^{(j)}_i}{\ell^{(j)}}$. This gives the elements
$\vec{a}_j$ in Def \ref{collection}.

The algebra from Eq \eqref{bigone} maps to $\mathscr{A}^{C}_{\bullet}$:
A summand with $\ell(\vec{m})=m$ in Eq \eqref{bigone} maps to the
summand $\op{H}^0(\op{Bun}_{\op{SL}(r)}(C), \mc{D}^{\tensor m})$.
Proposition \ref{SomethingIntro} shows that the image of  Eq $\eqref{bigone}$  contains the summands $\op{H}^0(\op{Bun}_{\op{SL}(r)}(C), \mc{D}^{\tensor m})$ with $m$ sufficiently divisible. Theorem \ref{qmain} now follows from Lemma \ref{verona} and Proposition \ref{se2} (c).

\end{proof}

\subsection{A variation on definition of weights $\vec{a}$}\label{prius}
We have proved finite generation of the section ring of the determinant of cohomology line bundle using varieties $\mc{X}(\vec{a})$ and vector bundles $\mc{G}$ on  $X_0$ where weights $\vec{a}=(a_i)_{i\in I}$ were chosen as in Def \ref{CP} such that the  $a_i=\op{deg}(\mc{G}|_{X_{0,i}})=\op{deg}(\mc{G}_i)$ may be negative, as long as $\sum_{i\in I}a_i>0$.
In particular,
\begin{equation}\label{perturbationAgain}
(g-1) a_i = \frac{\deg(\mg_i)}{\rk(\mg)}=\frac{2g_i-2+n_i}{2}+\sum_{x\in \nu^{-1}(S) \cap X_i} \epsilon_x,
\end{equation}
where $\epsilon_x$ is either element of the following two element set $\{\frac{\lambda_x^*(\xrminusone)}{\ell} -\frac{1}{2},\frac{1}{2}-\frac{\lambda_{x}(\xrminusone)}{\ell}\},$
so that if $\nu^{-1}(s)=\{\sas,\sbs\}$, $\epsilon_{\sas}+\epsilon_{\sbs}=0$ for all nodes $s$ on $X_0$.
There is another way to choose $\epsilon_x$ so that the $a_i=\op{deg}(\mg_i)$ are necessarily non-negative (and $>0$ if the genera $g_i$ are all $>0$).  In this section we describe this alternative approach and show that it works.

\subsubsection{The midpoint choice for $\epsilon_x$}
If we take $\epsilon_x$ to be the midpoint of the two extreme choices in Proposition \ref{SomethingIntro}, i.e.,
\begin{equation}\label{middy}
\epsilon_x= \frac{\lambda_x^*(\xrminusone)-\lambda_x(\xrminusone)}{2\ell},
\end{equation}
then the degrees of $\mg_i$ are half-integers and not necessarily integers. That is, we get a "trace" element in $(\Bbb{Z}/2\Bbb{Z})^{|I|}$.

 Suppose $\sigma$ and $\tau$ are in components $\lambda$ and $\mu$ of $\op{H}^0(\op{Bun}_{\op{G}}(X_0),\phi^*\mathcal{D}^{\ell})$,
and $\op{H}^0(\op{Bun}_{\op{G}}(X_0),\phi^*\mathcal{D}^{\ell'})$ respectively such that the traces  produced for each are equal (e.g., if $\sigma=\tau$), where $G=\op{SL}(r)$. Then it follows from the proof of \ref{SomethingIntro}) that the product of the sections corresponding to $\sigma$ and $\tau$ in $\op{H}^0(\op{Bun}_{\op{G}}(X_0),\phi^*\mathcal{D}^{\ell})$ comes from an element in $\op{H}^0(\mc{M}^0(\ml),\mathcal{L}_{\mathcal{G}})^{\op{GL}(V)}$
with $\rk(\mg)=\ell +\ell'$ and the degrees of $\mg_i$ integers (which are sum of those for $\sigma$ and $\tau$).

Therefore if $R_0$ is the subring of the ring $\mathscr{A}^{C}_{\bullet}$ coming from compactifications (i.e., the image over all $\mathcal{G}$ of $\op{H}^0(\mc{M}^0(\ml),\mathcal{L}_{\mathcal{G}})^{\op{GL}(V)}$, then $\mathscr{A}^{C}_{\bullet}$ is integral over $R_0$ (squares of generating elements of $\mathscr{A}^{C}_{\bullet}$ are in $R_0$).
Furthermore picking a representative $\sigma$ each (if available) for each $(\Bbb{Z}/2\Bbb{Z})^{|I|}$, we see that the fraction field of $\mathscr{A}^{C}_{\bullet}$ is algebraic over $R_0$. Since $R_0$ is finitely generated, its integral closure in any algebraic extension is finite over it, as is $\mathscr{A}^{C}_{\bullet}$ (a $R_0$ submodule). Hence the ring $\mathscr{A}^{C}_{\bullet}$ is finitely generated.

\subsubsection{Non negativity of $\deg(\mg_i)$ for the midpoint choice of $\epsilon_x$}
\begin{lemma}\label{vince}
Let $\epsilon =\frac{1}{2\ell}(\mu-\mu^*)(x_{r-1})$ where $\mu_1-\mu_r\leq \ell$. Then $\epsilon \in (\frac{1}{r}-\frac{1}{2}, \frac{1}{2}- \frac{1}{r})$.
\end{lemma}
\begin{lemma}\label{NN}
The "midpoint" value \eqref{middy} for $\frac{\deg(\mg_i)}{\rk(\mg)}$ in \eqref{perturbation} is $>0$ if $g_i>0$ and is $\geq 0$ for $g_i=0$.
\end{lemma}
\begin{proof}One has
$$
\frac{\deg(\mg_i)}{\rk(\mg)}=\frac{2g_i-2+n_i}{2}+\sum_{x\in \nu^{-1}(S)\cap X_i} \epsilon_x =\frac{2g_i-2}{2} + \sum_{x\in \nu^{-1}(S)\cap X_i}(\epsilon_x+\frac{1}{2}).$$

The claim for $g_i> 0$ therefore  follows from Lemma \ref{vince}. The proof for $g_i=0$ is the following: Let $n=n_i$, then for every $\lambda$ which contributes a non-zero term in the factorization formula, then setting $\mu^{x}=\frac{1}{\ell}\lambda^*_{x}$
\begin{equation}\label{horn}
\sum_{x\in \nu^{-1}(S)\cap X_i} ((\mu^{x}_1+\mu^{x}_r) -\frac{2}{r}|\mu^x|)\leq n-2.
\end{equation}

We will show that this is an multiplicative unitary eigenvalue inequality \cite{AW,Bloc} for the group $\op{SL}(r)$. We will find it convenient
to quote formulations from \cite{b4}. Set $\nu^{-1}(S)\cap X_i=\{p_1,\dots,p_n\}\subseteq X_i=\Bbb{P}^1$.
\begin{itemize}
\item Let $I=\{1,r\}\subseteq \{1,2,\dots, r\}$. Then the quantum product $\sigma_I^n\in QH(\op{Gr}(2,r))$ has a term $q^c\sigma_{J}$ with $c\leq n-2$: Note that (use, e.g.,  \cite[Lemma 2.6]{b4}) $\sigma_{I}^2=\sigma_{\{1,2\}}$, and $\sigma_{I}^b= q^{b-2}\sigma_{\{1+b-2,2+b-2\}}$ if $2\leq b\leq r$.
Write $n=ar+b$ with $0\leq b<r$, we get $\sigma_{I}^n= q^{(r-2)a}\sigma_{I}^b$. If $b=0$ or $b=1$, we
 need verify that $(r-2)a= n-b-2a\leq n-2$. If $a>0$, this is immediate. If $a=0$, then we need $0\leq n-2$ which holds since $n\geq 3$. If $b\geq 2$, we need show $(r-2)a+b-2\leq n-2$ which is clear.
 \item Therefore an $n+1$ pointed small Gromov-Witten invariant ($\sigma_K$ is the class dual to $\sigma_J$)
\begin{equation}\label{GW}
\langle \sigma_I,\dots,\sigma_I,\sigma_K\rangle_{c}\neq 0,\ c\leq n-2.
\end{equation}
\item From the data of $\lambda_{p_1},\dots,\lambda_{p_n}$ at level $\ell$ with corresponding conformal block on $\Bbb{P}^1$ non-zero, we get the data of $A^{(1)},\dots,A^{(n)}$ in $\op{SU}(r)$ which product to the identity $I\in \op{SU}(r)$, with eigenvalues, $\mu^{p_i}- \frac{1}{r}|\mu|$, $i=1,\dots,n$ (see e.g., \cite[Proposition 3.5]{b4}).
\item Now clearly the $n+1$ fold product
 $A^{(1)}A^{(2)}\cdots A^{(n)}\cdot I=I\cdot I=I$, and we write the eigenvalue inequality corresponding to \eqref{GW} (see \cite[Theorem 1.1]{b4}). This gives \eqref{horn} with $n-2$  replaced by $c$, but since $c\leq n-2$, we are done. Note that the contribution
    of $\omega_K$ on the identity matrix is zero.
\end{itemize}
Note that the inequality in \eqref{GW} may hold as an equality: For $\op{SL}(r)$, take $n=r$, $\ell=1$ and the $\mu$ weights $\omega_1=(1,0,\dots,0)$.

\end{proof}

\section{Conformal Blocks}\label{CBS}
The remainder of the paper, essentially the applications to Theorem \ref{qmain}, involve vector bundles of conformal blocks, which we define in Section \ref{VS}.   The applications are given in Section \ref{Apps}.  A key result is Theorem \ref{One}, proved in this section.

\subsection{Brief sketch of construction of the sheaf of conformal blocks}\label{VS} Given a triple $(\mathfrak{g},\vec{\lambda},\ell)$ as above,
let  ${\hat{\mathfrak{g}}}=\big( \mathfrak{g} \otimes k ((\xi)) \big)\oplus \  \mathbb{C} \cdot c$, be the Lie algebra with bracket
  $[\op{X}\otimes f(\xi), \op{Y}\otimes g(\xi)]=[\op{X}, \op{Y}]\otimes f(\xi) g(\xi)
  + ( \op{X} , \op{Y} ) \cdot \op{Res}(  g(\xi) df(\xi) ) \cdot c$, where  $\op{X}$, $\op{Y}\in \mathfrak{g}$, and  $c$ is in the center of ${\hat{\mathfrak{g}}}$.

  For each $\lambda_i$, there is a unique ${\hat{\mathfrak{g}}}$-module $\op{H}_{\lambda_i}$.  Set $\op{H}_{\vec{\lambda}}=\bigotimes_{i=1}^n \op{H}_{\lambda_i}$, and  let
$T$ be a smooth variety over a field $k$, and $\pi: \mathcal{C} \rightarrow T$ a proper flat family of curves whose fibers
have at worst ordinary double point singularities.  For $1\le i \le n$,  let  $p_i: T\rightarrow \mathcal{C}$ be sections of $\pi$ whose images are disjoint and contained in the smooth locus of $\pi$.

First suppose that  $T=\op{Spec}(\op{A})$ for some $k$-algebra $\op{A}$, and for each $i$  assume there are isomorphisms $\eta_i: \widehat{\mathcal{O}}_{\mathcal{C},p_i(T)} \rightarrow \op{A}[[\xi]]$.  Set $\op{B}=\Gamma(\mc C\setminus \cup_{i=1}^n p_i(T))$.  Then
for each $i$,  using the $\eta_i$,  there are maps $B \rightarrow \op{A}((\xi))$.   It can be shown that $\mathfrak{g} \otimes_k \op{B}$ is a Lie sub-algebra of
$\hat{\mathfrak{g}} \otimes_k  \op{A}$, and moreover that $\op{H}_{\vec{\lambda}} \otimes_k \op{A}$ is a representation of $\hat{\mathfrak{g}} \otimes_k  \op{A}$. Now  define the sections of the sheaf of conformal blocks $V_{C}(\vec{\lambda},\vec{p})$ over $T$ to be the quotient
$V_{C}(\vec{\lambda},\vec{p})=\op{H}_{\vec{\lambda}} \otimes_k \op{A} / (\mathfrak{g} \otimes_k \op{B}) \cdot \op{H}_{\vec{\lambda}} \otimes_k \op{A}$.
To define $V_{C}(\vec{\lambda},\vec{p})$ for  non-affine $T$,  take an open affine covering and extend by the sheaf property.
In this description,  the open set $\mc C\setminus \cup_{i=1}^n p_i(T)$ has been implicitly assumed to be affine.  But this premise can be removed using a descent argument: See \cite[Prop 2.1]{Fakh}, and the discussion following.

 \subsubsection{The algebra of conformal blocks}

 \begin{definition}For integers $m \ge 1$, we let $\mathbb{V}(\mathfrak{g}, \vec{\lambda}, \ell)[m] = \mathbb{V}(\mathfrak{g}, m\vec{\lambda}, m\ell)$,
  where if $\vec{\lambda}=\{\lambda_1,\ldots,\lambda_n\}$, then $m\vec{\lambda}=\{m\lambda_1,\ldots,m\lambda_n\}$.
 \end{definition}

Let  $\mathbb{V}[m]$ be a vector bundle of conformal blocks on $\ovmc{M}_{g,n}$, and let $x=(C,\vec{p})$ be any (closed) point in $\ovmc{M}_{g,n}$.  The direct sum
 $$\bigoplus_{m\in \mathbb{Z}_{>0}} \mathbb{V}[m]|^*_x$$
 carries an algebra structure \cite{Manon}.

\subsection{General geometric interpretation  in terms of stacks}\label{TheoremOneProof}

  \begin{theorem}\label{One} Let  $\mathbb{V}=\mathbb{V}(\mathfrak{g},\vec{\lambda},\ell)$ be a vector bundle of conformal blocks on $\ovmc{M}_{g,n}$, and  $(X_0,\vec{p}) \in \ovmc{M}_{g,n}$ a point on the boundary.  There is an isomorphism of algebras
$$\bigoplus_{m \in \mathbb{Z}_{\ge 0}} \mathbb{V}[m]|_{(X_0;\vec{p})}^* \cong \bigoplus_{m \in \mathbb{Z}_{\ge 0}} {\op{H}^0}(\op{Parbun}_{\op{G}}(X_0,\vec{p}),\mc{L}_{\op{G}}(X_0,\vec{p})^{\otimes m}),$$
where $\op{Parbun}_{\op{G}}(X_0,\vec{p})$ is the moduli stack of quasi-parabolic $\op{G}$-bundles on $X_0$,  and  $\mc{L}_{\op{G}}(X_0,\vec{p})$ is the line bundle described in Def \ref{LX_0}(2).
\end{theorem}

We note that Theorem \ref{One} does not, a priori, imply finite generation of the algebra of conformal blocks, because section rings of arbitrary line bundles on algebraic stacks (indeed, even of non-ample line bundles over projective varieties) are not necessarily finitely generated.

To prove this, we recall  the definition of the affine Grassmannian $\mathcal{Q}_{{\op{G}}}$, which for affine open sets $U$,  parameterizes pairs $(E,\phi)$ where $E$ principal ${\op{G}}$-bundle on  $X_0$ and $\phi: E|_{U}\rightarrow U\times G$ is a trivialization of $E$.  Letting $A_{U}$ denote the algebra of functions on   $U$,  and $\gamma$, the natural map:
\begin{equation}\label{tau}
\gamma: \mathcal{Q}_{\op{G}} \times (G/B)^n  \longrightarrow {\op{Parbun}}_{{\op{G}}}(X_0, \vec{p}),
\end{equation}
the proof of  Theorem \ref{One} follows from three assertions:
\begin{enumerate}
\item  There is an injective map
$${\op{H}}^0({\op{Parbun}}_{{\op{G}}}(X_0, \vec{p}), \mc{L}^{m}_{{\op{G}}}(X_0,\vec{p})) \hookrightarrow {\op{H}}^0( \mathcal{Q}_{{\op{G}}}\times ({\op{G}}/B)^n , \gamma^*(\mc{L}^m_{\op{G}}(X_0,\vec{p})))^{\mathfrak{g}\otimes_k A_U};$$
\item
${\op{H}}^0( \mathcal{Q}_{{\op{G}}}\times ({\op{G}}/B)^n , \gamma^*(\mc{L}^m_{\op{G}}(X_0,\vec{p})))^{\mathfrak{g}\otimes_k A_U} \cong \mathbb{V}[m]|^*_{(X_0,\vec{p})}$; and
\medskip
\item  ${\op{h}}^0({\op{Parbun}}_{{\op{G}}}(X_0, \vec{p}), \mc{L}_{{\op{G}}}(X_0,\vec{p})) \ge {\op{dim}}({\op{H}}^0(\mathcal{Q}_{{\op{G}}}\times ({\op{G}}/B)^n, \gamma^*\mc{L}_{{\op{G}}}(X_0,\vec{p}))^{\mathfrak{g}\otimes A_{U}})$.
\end{enumerate}

\subsubsection*{The Affine Grassmannian $\mathcal{Q}_{{\op{G}}}$}  Given $X_0$, we remove smooth points $\{q_1,\ldots, q_k\}$ so that $U=X_0\setminus \{q_1,\ldots, q_k\}$ is affine.  It is well known that one may  parameterize pairs $(E,\phi)$ where $E$ principal ${\op{G}}$-bundle on  $X_0$ and $\phi: E|_{U}\rightarrow U\times G$ is a trivialization of $E$, for $U=X_0\setminus q_i$,  by the quotient
$$\mathcal{Q}_{i}={\op{G}}(\mathbb{C}((\xi_i)))/{\op{G}}(\mathbb{C}[[\xi_i]]).$$
Here, if  $R$ is any commutative \ $\mathbb{C}$-algebra, then ${\op{G}}(R)={\op{Hom}}_{Sch/\mathbb{C}}({\op{Spec}}(R),{\op{G}})$.
 One needs to remove at least one point for every component of $X_0$; and to  parameterize  principal ${\op{G}}$-bundles $E$ on $X_0$ with trivialization $\phi: E|_{U}\rightarrow U\times G$
for $U=X_0\setminus \{q_1,\ldots,q_k\}$, one takes  $\mathcal{Q}_{{\op{G}}}=\Pi_{i=1}^k \mathcal{Q}_{i}$.
This correspondence is explained, for example, in \cite[Theorem 3.8, page 61]{Gomez}.

\subsubsection{Step $1$}
\begin{claim}\label{Step1}  Let $A_{U}$ be the algebra of functions on the open set  $U=X_0 \setminus \{p_1,\ldots,p_k\}$.  There is an injective map
$${\op{H}}^0({\op{Parbun}}_{{\op{G}}}(X_0, \vec{p}), \mc{L}^{m}_{{\op{G}}}(X_0,\vec{p})) \hookrightarrow {\op{H}}^0( \mathcal{Q}_{{\op{G}}}\times ({\op{G}}/B)^n , \gamma^*(\mc{L}^m_{\op{G}}(X_0,\vec{p})))^{\mathfrak{g}\otimes_k A_U}.$$
\end{claim}

\begin{lemma}\label{Lemma1}The natural map $\gamma: \mathcal{Q}_{\op{G}} \times (G/B)^n  \longrightarrow {\op{Parbun}}_{{\op{G}}}(X_0, \vec{p})$
 is dominant.
\end{lemma}

\begin{proof}
For simplicity in notation, we first assume there are no marked points and show
$\gamma: \mathcal{Q}_{\op{G}}  \longrightarrow {\op{Bun}}_{{\op{G}}}(X_0)$,
is  dominant.   We note that by \cite{Wang}, ${\op{Bun}}_{{\op{G}}}(X_0)$ is smooth;  a fiber bundle over ${\op{Bun}}_{{\op{G}}}(X)$, with fiber ${\op{G}}$, it is irreducible.

Let $U_1=X_0\setminus \{q_1,\ldots, q_k\}$ be affine, and  $U_2$ be any affine open containing the points  $q_1,\ldots, q_k $, so that $\mathcal{U}=\{U_1,U_2\}$ is a \v{C}ech covering of $X_0$.
 Let $A_1$, $\ldots$, $A_N \in \mathfrak{g}$ be any collection of nilpotent elements that freely generate $\mathfrak{g}$ as a $k$-module. One can always find such a basis as it exists for $\sL_2(k)$, and by hypothesis $\mathfrak{g}$ is simple.    In particular,  $A_1$, $\ldots$, $A_N$ is also a basis for
 $\mathfrak{g} \otimes \mathcal{O}_{X_0}(U_1\cap U_2)$ as an $\mathcal{O}_{X_0}(U_1\cap U_2)$-module.
   We may represent a basis for ${\op{H}}^1(X_0,\mathfrak{g}\otimes \mathcal{O}) \cong \check{{\op{H}}}^1(\mathcal{U}, \mathfrak{g}\otimes \mathcal{O})$ by
    $\{\phi_1,\ldots,\phi_{D}\}$, where for $j \in \{1,\ldots, D\}$,  we write
 $$\phi_j=\sum_{i=1}^N A_i \otimes f_{ji}, \ \ \mbox{ with } \ f_{ji} \in \mathcal{O}_{X_0}(U_1\cap U_2).$$

For $M=ND$  consider the map $\mathbb{A}^M_k \longrightarrow {\op{G}}(U_1\cap U_2)$,
 given by
 \begin{multline}
 (\alpha_{11},\ldots, \alpha_{1N}, \alpha_{21} \ldots \alpha_{2N},\ldots, \alpha_{D1}, \ldots, \alpha_{DN}) \mapsto \\
 {\op{exp}}(A_1 \otimes \alpha_{11} f_1)\cdots {\op{exp}}(A_1 \otimes \alpha_{1N} f_N)\cdot {\op{exp}}(A_2 \otimes \alpha_{21} f_2)\cdots
  {\op{exp}}(A_2 \otimes \alpha_{2N} f_2)\\ \cdots \cdots {\op{exp}}(A_N \otimes \alpha_{D1} f_N)\cdots {\op{exp}}(A_N \otimes \alpha_{DN} f_N) ,
 \end{multline}
which may  be used to patch trivializations $U_i \times {\op{G}} \rightarrow G$, thereby giving a map $\delta: \mathbb{A}^M_k \rightarrow \mathcal{Q}_{{\op{G}}}$.

 The map $\gamma  \circ \delta$  gives an induced map  on tangent spaces.  The Kodaira-Spencer map:

 $$T_{0}(\mathbb{A}^M_k)=(k)^M \longrightarrow {\op{H}}^1(X_0,\mathfrak{g}\otimes \mathcal{O}) \cong \check{{\op{H}}}^1(\mathcal{U}, \mathfrak{g}\otimes \mathcal{O}),$$
 $$(\alpha_{11},\ldots, \alpha_{1N}, \alpha_{21} \ldots \alpha_{2N},\alpha_{D1}, \ldots, \alpha_{DN})
 \mapsto \sum_{1\le i \le D, 1 \le j\le N}^N A_j \otimes \alpha_{ij} f_{ij},$$
takes  $(1,\ldots,1,0,\ldots \ldots,0)$ to $\phi_1$, $(0,\ldots, 0, 1,\ldots,1,0,\ldots \ldots,0)$ to $\phi_2$, $\ldots$, $(0\ldots \ldots 0, 1\ldots 1)$ to $\phi_D$, and hence is surjective.     So the map $\mathbb{A}^M_k \longrightarrow  {\op{Bun}}_{{\op{G}}}(X_0)$ is dominant and hence the claim holds.

An analogous argument will carry through if the bundles in question have parabolic structures at marked points.
\end{proof}

\begin{remark}
The map $\gamma$ is shown to be surjective in \cite{Serre}, for $G={\op{SL}}(r)$, and for general semisimple groups in \cite{BF}.
\end{remark}

\begin{proof}(of Claim \ref{Step1}) \ The group scheme ${\op{G}}(A_U)$ acts on $\mathcal{Q}_{{\op{G}}}\times ({\op{G}}/B)^n$ and the map $\tau$ is equivariant for this action.  Therefore since  by Lemma \ref{Lemma1},  $\tau$ is dominant, we have that
$${\op{H}}^0({\op{Parbun}}_{{\op{G}}}(X_0, \vec{p}), \mc{L}^{m}_{{\op{G}}}(X_0,\vec{p})) \hookrightarrow {\op{H}}^0( \mathcal{Q}_{{\op{G}}}\times ({\op{G}}/B)^n , \gamma^*(\mc{L}^m_{\op{G}}(X_0,\vec{p})))^{{\op{G}}(A_U)}.$$
The claim follows from the fact that invariants for the k-group  ${\op{G}}(A_U)$ are a subset of invariants for its associated Lie algebra $\mathfrak{g}\otimes_k A_U$.
\end{proof}

\subsubsection{Step $2$}
\begin{claim}\label{Step2}  \begin{equation}\label{S2E}\mathbb{V}[m]|^*_{(X_0,\vec{p})} \cong {\op{H}}^0( \mathcal{Q}_{{\op{G}}}\times ({\op{G}}/B)^n , \gamma^*(\mc{L}^m_{\op{G}}(X_0,\vec{p})))^{\mathfrak{g}\otimes_k A_U}
\end{equation}
\end{claim}

\begin{proof} Let $X_i$ be the irreducible components of the smooth curve $X$.
By Def \ref{LX_0},  $$\gamma^*(\mc{L}^m_{\op{G}}(X_0,\vec{p}))=(p\circ \gamma)^*(\mc{L}^{m}_{\op{G}}(X,\vec{p})).$$
By assumption $U=X_0 \setminus \{q_1,\ldots, q_k\}$ is affine, and we set $V=U\setminus \{p_1,\ldots, p_n\}$.  Putting trivial representations at the points $\{q_1,\ldots, q_k\}$, the left hand side of Eq \ref{S2E} with $m=1$ looks like:
\begin{multline}\label{S2EE}
\mathbb{V}(\mathfrak{g}, \vec{\lambda}, \ell)|^*_{(X, \vec{p})}\cong \mathbb{V}(\mathfrak{g}, \vec{\lambda}\cup \{0^k\}, \ell)|^*_{(X, \vec{p}\cup \vec{q})} \\
\cong \bigg[\mathcal{H}_{\lambda_1} \otimes \cdots \otimes \mathcal{H}_{\lambda_n} \otimes \mathcal{H}_{0} \otimes \cdots \otimes \mathcal{H}_{0}
\bigg/ (\mathfrak{g} \otimes \mathcal{A}_{V})\cdot (\mathcal{H}_{\lambda_1} \otimes \cdots \otimes \mathcal{H}_{\lambda_n} \otimes \mathcal{H}_{0} \otimes \cdots \otimes \mathcal{H}_{0})  \bigg]^*.
\end{multline}

We now apply \cite[Theorem 3.18, page 58]{Ueno}, to get the right hand side of Eq \ref{S2EE}  is:
$$ \bigg[V_{\lambda_1} \otimes \cdots \otimes V_{\lambda_n} \otimes \mathcal{H}_{0} \otimes \cdots \otimes \mathcal{H}_{0}
\bigg/ (\mathfrak{g} \otimes \mathcal{A}_{U})\cdot (\mathcal{H}_{\lambda_1} \otimes \cdots \otimes \mathcal{H}_{\lambda_n} \otimes \mathcal{H}_{0} \otimes \cdots \otimes \mathcal{H}_{0})  \bigg]^*.$$

\end{proof}

\subsubsection{Step $3$}
\begin{claim}\label{Step3}
${\op{h}}^0({\op{Parbun}}_{{\op{G}}}(X_0, \vec{p}), \mc{L}_{{\op{G}}}(X_0,\vec{p})) \ge {\op{dim}}({\op{H}}^0(\mathcal{Q}_{{\op{G}}}\times ({\op{G}}/B)^n, \gamma^*\mc{L}_{{\op{G}}}(X_0,\vec{p}))^{\mathfrak{g}\otimes A_{U}})$.
\end{claim}

\begin{proof}
For simplicity, we will assume that $X_0$ has a single node $\xzero$ and $\nu^{-1}(\xzero)=\{\xone,\xtwo\}$. We now show
$${\op{h}}^0({\op{Parbun}}_{{\op{G}}}(X_0, \vec{p}), \mc{L}_{{\op{G}}}(X_0,\vec{p})) \ge {\op{dim}}({\op{H}}^0(\mathcal{Q}_{{\op{G}}}\times ({\op{G}}/B)^n, \gamma^*\mc{L}_{{\op{G}}}(X_0,\vec{p}))^{\mathfrak{g}\otimes A_{U}}).$$

By Lemma \ref{Lemma3},
\begin{multline}\label{decompose}
{\op{H}}^0({\op{Parbun}}_{{\op{G}}}(X_0, \vec{p}), \mc{L}^m_{{\op{G}}}(X_0,\vec{p}) )
= {\op{H}}^0({\op{Parbun}}_{{\op{G}}}(X, \vec{p}), \mc{L}^m_{{\op{G}}}(X,\vec{p})\otimes \Big(  \bigoplus_{\lambda } \mathscr{E}_{\lambda^*}^{\xtwo} \otimes \mathscr{E}_{\lambda}^{\xone}\Big)) \\
=\bigoplus_{\lambda} {\op{H}}^0({\op{Parbun}}_{{\op{G}}}(X, \vec{p}), \mc{L}^m_{{\op{G}}}(X)   \otimes  \Big(  \mathscr{E}_{\lambda^*}^{\xtwo} \otimes \mathscr{E}_{\lambda}^{\xone}\Big)) \\
\bigoplus_{\lambda} \mathbb{V}(\mathfrak{g}, \{\lambda_1,\ldots, \lambda_n, \lambda, \lambda^*\}, \ell) [m]|^*_{(X; p_1,\ldots, p_n, \xone,\xtwo)}.
\end{multline}

By Factorization, the part of this sum indexed by weights $\lambda \in \mathcal{P}_{\ell}(\mathfrak{g})$
 is isomorphic to $$\mathbb{V}(\mathfrak{g}, \{\lambda_1,\ldots,\lambda_n\}, \ell)[m]|^*_{(X_0, \{p_1,\ldots,p_n\})}.$$
 Therefore the dimension is at least as big as that of $\mathbb{V}(\mathfrak{g}, \{\lambda_1,\ldots,\lambda_n\}, \ell)[m]|^*_{(X_0, \{p_1,\ldots,p_n\})}$.
\end{proof}

\begin{remark}\label{YesBlocks} By combining all three steps we see that in fact the proof of Claim \ref{Step3} gives that the only nonzero contributions to the sum
$$\bigoplus_{\lambda} \mathbb{V}(\mathfrak{g}, \{\lambda_1,\ldots, \lambda_n, \lambda, \lambda^*\}, \ell) [m]|^*_{(X; p_1,\ldots, p_n, \xone,\xtwo)}$$
come from $\lambda \in \mathcal{P}_{\ell}(\mathfrak{g})$. This gives another proof of Lemma \ref{Lemma3}(3).
\end{remark}

\section{Applications}\label{Apps}
Here we prove Theorem \ref{globaleJJs}, Theorem \ref{A2}, and Proposition \ref{C} as well as give an example illustrating  Theorem \ref{A2}.
\subsection{Application One: Proof of Theorem \ref{globaleJJs}}

Let $\scr{A}_{\bullet} = \bigoplus_{m\in \mathbb{Z}_{\ge 0}} \mathbb{V}(\sL_r,m \ell)^*$, set $\mc{X}=\op{Proj}(\scr{A}_{\bullet})$, and consider the map $p: \mc{X} \to \ovmc{M}_g$.
By Proposition \ref{flat}, the map $p$ is flat.

For Part (1): For  any closed point $[C] \in \ovmc{M}_{g}$, one has that  $\scr{A}_{\bullet}|_{[C]}$ is an integral domain, since ${\op{Bun}}_{{\op{G}}}(C)$ is smooth and connected (fiber it over the moduli of $G=\op{SL}(r)$-bundles on the normalization of $C$). One could also use Theorem \ref{one}: For  any closed point $[C] \in \ovmc{M}_{g}$, one has that  $\scr{A}_{\bullet}|_{[C]}$ is  $\oplus_m \mathbb{V}(\sL_r, m)|_{[C]}^*$. This a subalgebra, formed by suitable Lie-algebra invariants, of a product of the algebra of sections of a line bundle on the ind-integral affine Grassmannian, with an $n$-fold product of complete flag varieties  \cite[Section 10]{LaszloSorger}. Therefore, the fibers $\op{Proj}(\scr{A}_{\bullet}|_x)$ are integral schemes, and hence are integral and irreducible. Normality of the fibers of $p$ follows from the normality of ${\op{Bun}}_{{\op{G}}}(C)$: The sheaf $\oplus \mathcal{L}^n$ is a sheaf of normal algebras over ${\op{Bun}}_{{\op{G}}}(C)$.

Part $(2)$ follows from the work of \cites{BeauvilleLaszlo,  Faltings, KNR}.

Lastly, for finite generation of $\scr{A}_{\bullet}$, we appeal to the fact that the moduli stack  $\ovmc{M}_g$ is stratified by the topological type of the curves being parameterized: Curves having $k$-nodes of a particular type determine the (generic) element of  each component of the codimension $k$ boundary strata.
In Section \ref{Uniform}, we show that there is a uniform such constant $m$ for curves of a given stratum, such that $\scr{A}_{\bullet}|_{[C]}$ is generated by $\scr{A}_{i}|_{[C]}$, $i=1,\dots,m$.
Together with Lemma \ref{montreal},  this  will finish the proof of Theorem \ref{globaleJJs}.

\begin{lemma}\label{montreal}
Suppose we can find a uniform global constant $m$, such that for all stable curves $C$ of genus $g$, $\scr{A}_{\bullet}|_{[C]}$ is generated by $\scr{A}_{i}|_{[C]}$, $i=1,\dots,m$. Then, $\scr{A}_{\bullet}$ is finitely generated.
\end{lemma}

 \begin{proof}This follows from Nakayama's Lemma.
 \end{proof}

 \begin{proposition}\label{flat}Let $X$ be a scheme and $S_{\bullet}=\oplus_m S_m$ a graded sheaf of algebras over $X$. Assume that the sheaves $S_m$ are locally free of finite rank over $X$, and also that $S_{\bullet}=\oplus_m S_m$ is finitely generated as an $S_0$ algebra. Then $\op{Proj}(S_{\bullet})\to X$ is flat.
\end{proposition}
\begin{proof}  This is standard: local rings of $\op{Proj}(S_{\bullet})$ are grade zero summands in localizations of $S_{\bullet}$.
\end{proof}

\subsubsection{Uniform constants $m$  on strata}\label{Uniform}
Let $X_0\to T$ be a family of equisingular curves (i.e., the graph encoding singularities and genera is constant), let $X\to T$ be the family of normalizations. After sufficient surjective base change of $T$, we can assume that the labelings of components of the curve, etc can be made in families. We will also assume that $T$ is smooth. Note further that
we are also at liberty to replace $T$ by a non-empty open subset, since the complement will have smaller dimension, and bounds $m$ can be found for the complement by induction.

Given an ample line bundle $\mc{L}$ on $X_0\to T$, let $V$ be a fixed vector space isomorphic to $\op{k}^{P(m)}$ for $m >>0$, as before.  Let $\op{Quot}_{X_0/T}(V\tensor \mc{L}^{-1},\op{P})$ be relative Quot scheme:  For every $T'\to T$, an $T'$-valued point of $\op{Quot}_{X_0/T}(V\tensor \mc{L}^{-1},\op{P})$
 is a pair $(\mc{E},\alpha)$ with $\mc{E}$ a coherent sheaf on $X_0'=T'\times_{T}X_0$, flat over $T'$, with fibers $\mc{E}_t$ having Hilbert polynomial $P$ for every $t \in T'$, and
 $[\alpha: (V\tensor \mc{L}^{-1})' \twoheadrightarrow \mc{E}]$ a surjection $X'_0 \to T'$.    Relative versions $\op{Quot}_{X_0/T}(V\tensor \mc{L}^{-1},\op{P},1)$, $\op{Q}_{X_0/T}^0(\ml)$,  $\op{Q}^{\op{LF}}_{X_0/T}(\ml)$, $\overline{\op{Q}}^{\op{det}}_{X_0/T}(\ml)$,  $\mc{M}_{X_0/T}(\mc{L})$, and $\mc{M}_{X_0/T}^0(\mc{L})$ can be defined.
Note that by passing to a non-empty open subset of $T$, we can assume that the fibers of these objects coincide with the objects defined earlier (when we were working fiber-wise). Here the only subtle point is  $\mc{M}_{X_0/T/T}(\mc{L})$ which involves normalization, and we need to show that normalization commutes with taking fibers, at least generically, in characteristic zero. This is Lemma \ref{normality}. We also note Grothendieck's generic representability theorem (see e.g., \cite[Theorem 4.18.2]{Kleiman}).

One can define an embedding of $\op{Quot}_{X_0/T}(V\tensor \mc{L}^{-1},\op{P},1)$ into a relative Grassmannian variety, and denote the corresponding line bundle by $T_{X_0/T}(\mathcal{A})$.   To define the analogous map, and show it  is an
embedding, one needs that boundedness works uniformly in families, which Simpson proves \cite[Corollary 1.6]{simpson}.

One can also define a relative version of $\vec{a}$-slope semistability (respectively $\vec{a}$-slope stability).

In order to compare $\vec{a}$-slope semistable sheaves with GIT-(semi)stable loci in the Quot scheme one needs a relative and uniform bound on the number of global sections of  a sheaf in terms of its maximal slope (aka the Le Potier-Simpson estimate) which is given in \cite[Corollary 1.7]{simpson} .

All the arguments hold uniformly in families. We list the assertions that need to be generalized in families, and comment on the subtleties, if any.
\begin{enumerate}
\item Proposition \ref{simpsonian} generalizes easily. It is known that if $Q\to T$ is a family with a relatively ample bundle $T$.
Then, semistability on fibers corresponds in the expected manner with semistability in the family \cite[Lemma 1.13]{simpson}.
 \item Lemma \ref{extendo}: This is a problem of extending sections over normal schemes, here we have to extend over codimension one points. We can extend over generic point of the base using Lemma \ref{extendo}. For rest of the codimension one points, which map to codimension one points on the base the section is already defined there and hence does not have a pole (we can shrink the base to make $\mc{M}_{X_0/T}(\mc{L})-\mc{M}^0_{X_0/T}(\mc{L})$ flat over the base $T$).
\item In Proposition \ref{OneIntro}: Replace $H^0$ by push-forward of corresponding sheaves to the base  $T$.
\end{enumerate}

\subsubsection{}

Let $U\subset X$ be schemes over a smooth base $T$. Assume that $U$ is smooth over $T$, and $X$ proper over $T$. Assume the fibers of $U\to T$ at closed points are irreducible. Let $Y$ be the closure of $U$ in $X$ and $\widetilde{Y}$ the normalization of $Y$. Assume $T$ is smooth over a field of characteristic zero.
\begin{lemma}\label{normality}
There is a non-empty open subset $V$ of $T$ such that the fiber of $p:\widetilde{Y}\to T$ over a closed point in $V$, is an irreducible normal variety.
\end{lemma}
\begin{proof}
We will assume that $p:\widetilde{Y}\to T$ is flat over $T$ by replacing $T$ by a non-empty open subset.
It is also a proper map. $\widetilde{Y}-U$ has strictly smaller dimension than $\widetilde{Y}$. There is a non-empty open subset $V_1$ of $T$ such that for $t\in V_1$ each irreducible component of the fiber $p^{-1}(t)$ is of dimension $\dim X -\dim T$, and each irreducible component of $p^{-1}(t)\cap(\widetilde{Y}-U)$ has dimension $<\dim X-\dim T-1$. It follows that  $p^{-1}(t)\cap U$ is dense in $p^{-1}(t)$ at such points.

Now for the normality: The set of scheme theoretic points $t\in T$ such that the geometric fiber (fiber, base changed to algebraic closure of the residue field at $t$) of $p$ over $t$ is normal - is an open subset $V_2$ (possibly empty) of $T$ [EGAIV, (12.2.4)].
Now if $\eta$ is the generic point of $T$, then $\widetilde{Y}_{\eta}$ is a limit of open subsets of $\widetilde{Y}$ and is hence normal. Normal schemes over fields of characteristic zero remain so upon base change to algebraic closure, since this base change is a limit of etale extensions (e.g., \cite[page 167, exercise 6.19]{GW}).  Therefore $\eta\in V_2$ and hence $V_2$ is non-empty. We can take $V=V_1\cap V_2$ and see that the desired properties are satisfied.

\end{proof}

\subsection{Application Two: The proof of Theorem \ref{A2}}

\begin{proof}
Suppose $C$ is a nodal curve with a non-separating node. Then the polarization produced in
Proposition \ref{SomethingIntro} is the canonical one, and  $\rk\mg=1$. Therefore we find that for a suitable $\mg$, $$\op{H}^0(\mc{M}^0(\ml),\mathcal{L}_{\mathcal{G}}^{\ell})^{\op{GL}(V)}\to\op{H}^0(\op{Bun}_{\op{SL}(r)}(X_0),\phi^*\mathcal{D}^{\ell})$$
is an isomorphism. Now set $\mathcal{X}_C(r)$ to be the GIT quotient of $\mc{M}(\ml)$ by $\op{GL}(V)$, and the  linearization corresponding to $\mg$. By Kempf's descent Lemma one can show  that the line bundle $\mathcal{L}_{\mg}$ descends to the GIT quotient, and hence obtain the desired assertion (i.e, part (1)).

If  $C$ is arbitrary, but $r=2$, the midpoint choice in  Proposition \ref{SomethingIntro} (see Remark \ref{something}) still returns the canonical polarization since $\lambda=\lambda^*$ for $\op{SL}_2$. We can take $\rk(\mg)=2$ (and hence get only even levels) and proceed as before. Part (3) is a consequence of Theorem \ref{one}, see Lemma \ref{monty} below.
\end{proof}
\begin{lemma}\label{monty}
Let $X$ be a scheme and $S_{\bullet}=\oplus_m S_m$ a graded sheaf of algebras over $X$. Assume that the sheaves $S_m$ are locally free of finite rank over $X$, and also that $S_{\bullet}=\oplus_m S_m$ is finitely generated as an $S_0$ algebra. Then there exists a relatively ample line bundle $\ml$, and a positive integer $\ell$ such that setting $p:Y=\op{Proj}(S_{\bullet})\to X$, we have an isomorphism of algebras
$$\bigoplus_{m\geq 0}S_{m\ell}=\bigoplus_{m\geq 0}p_*\ml^m$$
\end{lemma}
\begin{proof}
By \cite[Exercise 5.9]{Hartshorne}, for sufficiently large $d$, and if $S_{\bullet}$ is generated in degree $1$, $p_{*}\mathcal{O}(d)=S_d$. The desired assertion follows.
\end{proof}

 \subsubsection{Example: The Veronese surface $( \mathbb{P}^3, \mc{O}(2))$}\label{FirstExample}
We recall that in case $C$ is a smooth  curve and $\mathbb{V}(\mathfrak{g},\ell)$ is a vector bundle of conformal blocks on $\ovmc{M}_g$, then
  \begin{equation}\label{isom2}
\mathbb{V}(\mathfrak{g},\ell)|_{C}^* \cong \op{H}^0(\mc{X}_{C},\mc{L}_C),
\end{equation}
for some projective variety $\mc{X}_{C}$ and some ample line bundle $\mc{L}_C$ on $\mc{X}_{C}$ (consistently with multiplication operations).
For $\mathfrak{g}=\sL_{r+1}$, then $\mathcal{X}_{C}$ is isomorphic to $\op{SU}_C(r+1,d)$, the moduli space of stable bundles on $C$ of rank $r+1$ and degree $d=0$.
 By \cite{NR}, if $C$ is smooth of genus $g=2$, then  $\mathbb{V}(\sL_2, m)|_{[C]}^*\cong \op{H}^0(\op{SU}_C(2), \mc{L}_{C}) \cong \op{H}^0( \mathbb{P}^3, \mc{O}(m))$.
 In \cite[Example 3.9]{BGK}, it was shown that  Eq \ref{isom2} does not extend, consistently with multiplication to  {\em{all}} points $[C] \in \Delta_{1}$, although it does hold for {\em{all}} points $[C] \in \Delta_{irr}\setminus \Delta_1$.  This was shown by giving recursive identities which must be satisfied if such an extension were to exist.  One of the main points of this work is that by starting with sufficiently divisible level,  extensions do exist at all points of $\ovmc{M}_g$.  For $\sL_2$, the level must be divisible by $2$.  Using techniques from \cite{BGK} we give, in Proposition \ref{Level2Prop}, a recursive identity which will be necessarily be satisfied by the first Chern classes of multiples $\mathbb{V}[m]= \mathbb{V}(\sL_2, 2m)$ if $(\mc{X}_{C},\mc{L}_C)$ can be taken to be $( \mathbb{P}^3, \mc{O}(2))$ for all $C$.

\begin{proposition}\label{Level2Prop}For $\Bbb{V}[m]=\mathbb{V}(\sL_2, 2m),$ suppose $\Bbb{V}[m]|_{[C]}^*\cong H^0(\mathbb{P}^3, \mc{O}(2m))$, for all $[C]\in \ovmc{M}_g$ consistently with multiplication operations.
Then $c_1(\Bbb{V}[m])=\alpha(m) \ c_1(\Bbb{V}) + \beta(m)\delta_1$, where

$$\alpha(m)={\binom{9+m}{10}}-20 \ { \binom{7+m}{10}}
+64 \ {\binom{6+m}{10}}-90 \ {\binom{5+m}{10}}+64 \ {\binom{4+m}{10}}-20 \ {\binom{3+m}{10}}+ {\binom{1+m}{10}},$$ and $$\beta(m)=-8{\binom{9+m}{9}}+\frac{192}{5} \ {\binom{6+m}{9}}-72 \ {\binom{5+m}{9}}
+64 \ {\binom{4+m}{9}}-24 \ {\binom{3+m}{9}}+ \frac{8}{5}{\binom{1+m}{9}}.$$
\end{proposition}
The proof uses resolutions of ideal sheaves of Veronese embeddings due to Lascoux, see \cite[Section 6.3.9 ]{Weyman}. We have verified these formulas for low values of $m$, and in principle verification for all $m$ should be possible since $c_1(\Bbb{V}[m])$ has been explicitly computed in \cite[Example 6.8]{BGK}. However such a verification will not ensure that
$(\mc{X}_{C},\mc{L}_C)$ equals $( \mathbb{P}^3, \mc{O}(2))$ for all $C$. We therefore ask if this equality holds for all reducible $C$. Note that the answer is known to be positive for irreducible curves by \cite{BGK}.

\subsection{Application Three: Chern classes of bundles in type $A$}\label{Chern}

\begin{definition}\label{QP}
A quasi-polynomial of degree $d$ can be written as $f(k) = \alpha_d(k) k^d + \alpha_{d-1}(k) k^{d-1} + \cdots + \alpha_0(k)$, where $\alpha_{i}(k)$ is a periodic function with integral period, and  $\alpha_d(k)$ is not identically zero. Equivalently, a function $f  : \mathbb{N} \to \mathbb{N}$ is a quasi-polynomial if there exist polynomials $p_0$, $\dots$, $p_{s-1}$ such that $f(n) = p_i(n)$ when $n \equiv i \bmod s$. The polynomials $p_i$ are called the constituents of $f$.
\end{definition}
\begin{proposition}\label{C}For $m$ sufficiently large, the Chern character $\op{Ch}(\mathbb{V}[m])$ is a quasi-polynomial in $m$, with coefficients in the Chow ring $\op{A}^{*}(\ovmc{M}_g)$.
\end{proposition}

\begin{corollary}\label{HC}For every $k$, and $m$ sufficiently large, the $k$-th Chern class $c_k(\mathbb{V}[m])$ is a linear combination of  boundary cycles of codimension $k$, with coefficients that are quasi-polynomial in $m$.\end{corollary}

Recursive formulas for the first Chern class and Chern character are given in \cites{Fakh,MOP,MOPPZ}. From our experience, it would seem to be rather challenging  to  use explicit formulas to conclude the quasi-polynomiality given in Proposition \ref{C} and Corollary \ref{HC}. Note that already for $\mathbb{V}(\sL_2, 1)$ on $\ovmc{M}_g$, the Chern classes are quasi-polynomial with period two (this can be computed from \cite{Fakh} as was done in \cite[Example 6.8]{BGK}). It would be interesting to determine/bound the period in general.

\begin{proof}(of Proposition \ref{C}) We know $\scr{A}_{\bullet} = \bigoplus_{m\in \mathbb{Z}_{\ge 0}} \mathbb{V}(\sL_r, m \ell)^*$ is finitely generated over $\mc{A}_0$, say by
$\{\mc{A}_{d_i}\}$, for  $i\in \{1,\ldots, r\}$, such that $\Pi_{i=1}^rd_i =d$.
Let $\mc{B}_{\bullet}$ be the $(dr)$-th Veronese subring of $\scr{A}_{\bullet}$: The graded ring such that $\mc{B}_{m}=\mc{A}_{drm}$, for $m \in \mathbb{Z}_{\ge 0}$.  Then $\mc{B}_{\bullet}$ is generated in degree one \cite[Lemma 2.1.6(v), page 21]{EGA2}.  Moreover, $\op{Proj}(\scr{A}_{\bullet})\cong \op{Proj}(\mc{B}_{\bullet})$, eg. \cite[Prop 3.3]{Reid}, and we consider the representable morphism
$p:\mc{X}= \op{Proj}(\mc{B}_{\bullet}) \to \ovmc{M}_g$.
As is explained in \cite[Proposition 2.2.3]{Edidin}, using \cite[Theorem 2.1]{DM}, $\ovmc{M}_g$ is a smooth Deligne-Mumford quotient stack, and  by \cite[Theorem 3.1]{EG},
 for any vector bundle $E$ on $\mc{X}$, one has $\tau_{\mc{X}}(E)=\op{Ch}(E)\op{Td}(T_{\mc{X}})$.
We can apply this to $E=\mc{O}(m)$ to get the Grothendieck-Riemann-Roch (GRR) identity in  $\op{Ch}(\ovmc{M}_g)^*_{\mathbb{Q}} \cong \op{A}(\ovmc{M}_g)^*_{\mathbb{Q}}$:
\begin{equation}\label{MainGRR}
p_*(\op{Ch}(\mc{O}(m)))\cdot \op{Td}(T_{\ovmc{M}_{g}}) = p_*(\op{Ch}(\mc{O}(m)\cdot \op{Td}(T_{\mc{X}})).
\end{equation}
Since $\op{Td}(T_{\ovmc{M}_{g}})$ is invertible in $\op{Ch}^*(\ovmc{M}_g)_{\mathbb{Q}} \cong \op{A}^*(\ovmc{M}_g)_{\mathbb{Q}}$, one obtains
\begin{equation}\label{RelGRR}
p_*(\op{Ch}(\mc{O}(m)))= p_*(\op{Ch}(\mc{O}(m))\cdot \op{Td}(T_{p})).
\end{equation}  Therefore, since Chern characters have the property that they are additive over exact sequences,
$\op{Ch}(\mc{O}(1)^{\tensor m})=\op{Ch}(\mc{O}(1))^{\tensor m}$,
it is enough to compute $p_*(\op{Ch}(\mc{O}(1))^{\tensor m}\cdot \op{Td}(T_{p}))$.  For this, since push-forward preserves dimension, we want to find the components of the intersection whose dimension will be  $k \le 3g-3$.  We will obtain an expression, given by a linear combination of cycles on $\ovmc{M}_g$, with coefficients that are polynomials that depend on $m$.

The right hand side of Eq \ref{RelGRR} simplifies, taking into account explicit formulas for the Chern character $\op{Ch}(p_*(\mc{O}(1)))$ and the Todd class $\op{Td}(\ovmc{M}_g)$,
which we can get from \cite{Fulton}, for example. The coefficients in the intersection $\op{Ch}(\mc{O}(1))^{\tensor m}  \cdot \op{Td}(T_{p})$
depend only on $m$. Therefore, we have proved (Lemma \ref{monty}) that for sufficiently large $m$, $\op{Ch}(\Bbb{V}[drm])$ is a polynomial in $m$. To cover other modulo classes (for $dr$), we define sheaves, $\mf_a$ for  $0\leq a<dr$ over $\mathcal{X}$ corresponding to the graded $\mc{B}_{\bullet}$-modules $\oplus_m \mc{A}_{a+mdr}$. We now apply GRR to the
sheaves $\mf_a(m)$ and complete the argument.

\end{proof}

 \section{History and open questions}\label{OtherQuestions}
\subsection{History of the problem motivating  Theorem \ref{globaleJJs}}\label{historie}  There is a natural question about the family of moduli spaces $ \op{SU}_{C}(r)$ over $\mc{M}_g$, which we answer in Theorem \ref{globaleJJs}.

\begin{bigquestion}\label{QO} Can one extend the family of moduli spaces $ \op{SU}_{C}(r)$ over $\mc{M}_g$, to a family $\mathscr{X} \to \ovmc{M}_g$  with $\mathcal{X}$ relatively projective and flat over $\ovmc{M}_g$?
\end{bigquestion}

Here we  compare our solution to Question \ref{QO}, given in Theorem \ref{globaleJJs}, to work of Newstead, Seshadri, Pandharipande and Sun, which give answers to problems related to Question \ref{QO}. These other constructions seem to present qualitatively different solutions from the one we give.
In particular, we know of no flat family, other than the one in Theorem \ref{globaleJJs}, extending the family of moduli spaces of vector bundles with trivial determinant over $\mc{M}_g$ into a flat family over  $\ovmc{M}_g$.  The family in  Theorem \ref{globaleJJs} is independent of any choices, and is related  to conformal field theory for singular curves. However, unlike the families discussed below, at the moment it lacks a full modular interpretation.  A proposed approach toward such an interpretation is outlined in Section \ref{partiale}.

\medskip

The moduli space of semistable vector bundles of degree $e$ and rank $r$ on a smooth curve $C$ has come to be denoted by $\op{U}_{C}(e,r)$ since via the Narasimhan-Seshadri correspondence \cite{NSesh}, points in $\op{U}_{C}(0,r)$ are identified with homomorphisms from the fundamental group $\pi_1(C)$ to the unitary group $\op{U}(r)$.  The problem of extending the family for $\op{U}_{C}(e,1)$ for $C \in \mc{M}_g$, to include curves $C$ with singularities, has a particularly rich history \cites{Igusa, Ishida, OdaSeshadri, DSouza, AltmanKleiman, Ca}.  In \cite[Theorem 5.8]{Newstead}, Newstead extends the family of   $\op{U}_{C}(e,r)$ over $C \in \mc{M}_g$, to curves with nonseparating nodes; and in \cite[pgs 155-]{Seshadri}, Seshadri extends it to arbitrary reduced curves.   These modular constructions depend on a choice of polarization for the curve.    By \cite{Pandh} and \cite[Theorem 1.21]{simpson}, there exists a  proper, modular family over  $\ovmc{M}_g$, such that fibers over points $[C] \in \mc{M}_g$, corresponding to smooth curves $C$, are isomorphic to $\op{U}_{C}(e,r)$, and  fibers over   $[C] \in \ovmc{M}_g\setminus\mc{M}_g$ are the moduli spaces of torsion free sheaves which are semistable for the canonical polarization on $C$, as constructed by Seshadri.  There is an approach initiated by Gieseker \cite{gie} in rank $2$, and extended to higher ranks by other authors (see \cite{NS2,Xia,Su3} and the references therein). Gieseker's approach uses semistable models of the nodal curve $C$, and the corresponding moduli spaces admit regular birational morphisms to Seshadri's compactifications.

In order to apply these constructions to answer Question \ref{QO}, one could form the closure of semistable bundles on smooth curves with trivial determinant in the space constructed by Pandharipande and Simpson.  One would then ask if the family so obtained is flat over $\ovmc{M}_g$. 
 This closure problem is related to work done surrounding  the Nagaraj-Seshadri locus, which we now describe.

Consider a one-parameter degeneration of a family of smooth curves $C_t$ into a stable curve $C_0$ with one double point. Now form the closure of the relative moduli space of semistable bundles of fixed rank and trivial determinant on the curves $C_t$ for $t\neq 0$  in a relative moduli stack of torsion free sheaves (fixing a relative polarization).  The Nagaraj-Seshadri locus consists of those bundles which arise as the fiber of the closure over $t=0$ for such families.
A conjectural description of the underlying reduced set of this locus, given  in  \cite[Conjecture page 136]{NS}, was shown to hold by Sun in \cites{SunA, Su3} (also see \cite{NRam, Bhosle1, BhosleReducible}). Sun \cite{Su3} also considered the case $C_0$ is reducible, and showed that in  this case the Nagaraj-Seshadri locus is possibly reducible.  It is uncertain whether the scheme structures of the fibers of the closures depend only on $C_0$ (see \cite{Su3,schmitt}).

 In particular, it is not clear that there is a flat family over $\ovmc{M}_g$ that can be obtained  by taking closures, as described, in the family constructed by Pandharipande and Simpson. By \cite{SunA}, except for the case of $\op{SL}(2)$, such a family will have reducible fibers  over points corresponding to  reducible curves, and therefore will differ from the flat family in Theorem \ref{globaleJJs}.

Parallel to this story, work in conformal field theory by Tsuchiya-Ueno-Yamada \cite{TUY}, produced vector bundles on $\ovmc{M}_g$ which extended the  vector bundle  with fibers $H^0(\op{SU}_{C}(r),\theta^{\tensor m})$ on $\mc{M}_g$ \cite{BeauvilleLaszlo,Faltings,KNR}. These fibers are referred to as generalized theta functions.  To the best of our knowledge, these vector bundles, obtained via conformal field theory, were not related to the compactifications in the previous paragraphs.

\subsection{Toward a modular interpretation of the fibers $\operatorname{Proj}(\mc{A}^{X_0}_{\bullet})$}\label{partiale}
 Fix a singular stable curve $X_0$.  We describe an approach which we hope will lead to a modular interpretation for  $\operatorname{Proj}(\mc{A}^{X_0}_{\bullet})$.
\begin{definition}
A vector bundle $\me$ of rank $r$ with trivial determinant on $X_0$ has  property $\mathscr{S}$
if there exists a global section (for some $m$) in $H^0(\operatorname{Bun}_{\operatorname{SL}(r)}(X_0),\mathcal{D}^{\otimes m})$ which does not vanish at $\me$.
\end{definition}

The open locus $\operatorname{Bun}^{\mathscr{S}}_{\operatorname{SL}(r)}(X_0)\subset \operatorname{Bun}_{\operatorname{SL}(r)}(X_0)$ of points with property $\mathscr{S}$ admits a natural map:
\begin{equation}\label{MAP}
\Phi: \operatorname{Bun}^{\mathscr{S}}_{\operatorname{SL}(r)}(X_0)\to \operatorname{Proj}(\oplus_m H^0(\operatorname{Bun}_{\operatorname{SL}(r)}(X_0),\mathcal{D}^{\otimes m}))=\operatorname{Proj}(\mc{A}^{X_0}_{\bullet}).
\end{equation}

\begin{question}\label{AlperQ}\begin{enumerate}
\item Is   $\Phi$  open?
\item Is the image of $\Phi$  a good quotient in the sense of Alper \cite{alper}?
\item Give a geometric description of the open locus $\operatorname{Bun}^{\mathscr{S}}_{\operatorname{SL}(r)}(X_0)$.
\item  Identify what is ``added" to form the compactification $\operatorname{Proj}(\mc{A}^{X_0}_{\bullet})$ of the image of $\Phi$.
\end{enumerate}
\end{question}

In terms of Part (3) of Question \ref{AlperQ}, the following result goes part of the way towards giving a modular description of the ``interior" of  $\operatorname{Proj}(\mc{A}^{X_0}_{\bullet})$:
\begin{lemma}\label{remmy}
$\me$ has property $\mathscr{S}$ iff it is semistable for some choice of weights $\vec{a}=(a_i)_{i \in I}$.
\end{lemma}

\begin{proof}
If $\me$ has property $\mathscr{S}$,  there is a section in one of the $\lambda$ summands of $H^0(\operatorname{Bun}_{\operatorname{SL}(r)}(X_0),\mathcal{D}^{\otimes m})$ which does not vanish at $\me$. By Prop \ref{SomethingIntro}, this section extends to a section in a compactification for some $\vec{a}$.  For the other implication use Proposition \ref{OneIntro}.
\end{proof}

\begin{definition}\cite{Osserman}
 A vector bundle $\me$ of rank $r$ on $X_0$ with trivial determinant is limit-semistable
if for any non-zero subsheaf $\mf\subset \me$ which has uniform multi-rank $r'$ on  $X_0$, we
have
$$\frac{\chi(X_0,\mf)}{r'}\leq \frac{\chi(X_0,\me)}{r}.$$
\end{definition}
A vector bundle is limit semistable if it is  linear semistable (Def \ref{Generalization}) for some choice of weights.
\begin{question}\label{OssermanQ}
Given a vector bundle $\me$ with trivial determinant on $X_0$, does $\me$  have property $\mathscr{S}$ if and only if $\me^{\oplus m}$ is limit semistable for all positive integers $m$?
\end{question}
We know that if $\me$  has property $\mathscr{S}$, then $\me^{\oplus m}$ is limit semistable for all positive integers $m$.  We don't know the other implication even in the simple case of a curve with two components.

\subsection{Acknowledgements}
We thank N. Fakhruddin, A. Kazanova and M. Schuster for useful discussions and clarifications.
 A.G. was supported by  NSF DMS-1201268 and by  DMS-1344994.

\bigskip

\noindent
P.B.: Department of Mathematics, University of North Carolina, Chapel Hill, NC 27599\\
email: belkale@email.unc.edu

\medskip

\noindent
A.G.: Department of Mathematics, University of Georgia, Athens, GA 30602\\
email: agibney@math.uga.edu


\begin{bibdiv}
\begin{biblist}


\bib{AW}{article} {
    AUTHOR = {Agnihotri, S.}
    AUTHOR ={ Woodward, C.},
     TITLE = {Eigenvalues of products of unitary matrices and quantum
              {S}chubert calculus},
   JOURNAL = {Math. Res. Lett.},
  FJOURNAL = {Mathematical Research Letters},
    VOLUME = {5},
      YEAR = {1998},
    NUMBER = {6},
     PAGES = {817--836},
      ISSN = {1073-2780},
}

\bib{alper}{article}{
    AUTHOR = {Alper, Jarod},
     TITLE = {Good moduli spaces for {A}rtin stacks},
   JOURNAL = {Ann. Inst. Fourier (Grenoble)},
  FJOURNAL = {Universit\'e de Grenoble. Annales de l'Institut Fourier},
    VOLUME = {63},
      YEAR = {2013},
    NUMBER = {6},
     PAGES = {2349--2402},
      ISSN = {0373-0956},
}
	
\bib{AltmanKleiman}{article}{
   author={Altman, Allen B.},
   author={Kleiman, Steven L.},
   title={Compactifying the Picard scheme},
   journal={Adv. in Math.},
   volume={35},
   date={1980},
   number={1},
   pages={50--112},
}
	
\bib{BeauvilleLaszlo}{article}{
  author={Beauville, Arnaud},
  author={Laszlo, Yves},
  title={Conformal blocks and generalized theta functions},
  journal={Comm. Math. Phys.},
  volume={164},
  date={1994},
  number={2},
  pages={385--419},
}

	


\bib{Bloc}{article} {
    AUTHOR = {Belkale, Prakash},
     TITLE = {Local systems on {$\Bbb P^1-S$} for {$S$} a finite set},
   JOURNAL = {Compositio Math.},
  FJOURNAL = {Compositio Mathematica},
    VOLUME = {129},
      YEAR = {2001},
    NUMBER = {1},
     PAGES = {67--86},
}




\bib{b4}{article}{
    AUTHOR = {Belkale, Prakash},
    TITLE = {Quantum generalization of the {H}orn conjecture},
   JOURNAL = {J. Amer. Math. Soc.},
  FJOURNAL = {Journal of the American Mathematical Society},
    VOLUME = {21},
    YEAR = {2008},
    NUMBER = {2},
     PAGES = {365--408},
}		
\bib{BF}{article}{
   author={Belkale, Prakash},
   author={Fakhruddin, Nakhmuddin},
   title={Triviality properties of principal bundles on singular curves},
   journal={arXiv:1509.06425 [math.AG]},
   date={2015},
}
\bib{BGK}{article}{
  author={Belkale, Prakash},
  author={Gibney, Angela},
  author={Kazanova, Anna},
  title={Scaling of conformal blocks and generalized theta functions over $\ovmc{M}_{g,n}$},
  journal={arXiv:1412.7204 [math.AG]},
  volume={},
  date={2015},
}

\bib{BK}{article} {
    AUTHOR = {Belkale, Prakash}
    AUTHOR ={ Kumar, Shrawan},
     TITLE = {The multiplicative eigenvalue problem and deformed quantum
              cohomology},
   JOURNAL = {Adv. Math.},
  FJOURNAL = {Advances in Mathematics},
    VOLUME = {288},
      YEAR = {2016},
     PAGES = {1309--1359},
      ISSN = {0001-8708},
}

	







\bib{Bhosle1}{article} {
    AUTHOR = {Bhosle, Usha N.},
     TITLE = {Vector bundles with a fixed determinant on an irreducible
              nodal curve},
   JOURNAL = {Proc. Indian Acad. Sci. Math. Sci.},
  FJOURNAL = {Indian Academy of Sciences. Proceedings. Mathematical
              Sciences},
    VOLUME = {115},
      YEAR = {2005},
    NUMBER = {4},
     PAGES = {445--451},
}

\bib{BhosleReducible}{article} {
    AUTHOR = {Bhosle, Usha N.},
     TITLE = {Vector bundles on curves with many components},
   JOURNAL = {Proc. London Math. Soc. (3)},
  FJOURNAL = {Proceedings of the London Mathematical Society. Third Series},
    VOLUME = {79},
      YEAR = {1999},
    NUMBER = {1},
     PAGES = {81--106},
}




\bib{Ca}{article} {
    AUTHOR = {Caporaso, Lucia},
     TITLE = {A compactification of the universal {P}icard variety over the
              moduli space of stable curves},
   JOURNAL = {J. Amer. Math. Soc.},
  FJOURNAL = {Journal of the American Mathematical Society},
    VOLUME = {7},
      YEAR = {1994},
    NUMBER = {3},
     PAGES = {589--660},
}

\bib{Cartan}{book} {
     Author = {Henri Cartan}
     TITLE = {S\'eminaire {H}enri {C}artan; 10e ann\'ee: 1957/1958.
              {F}onctions {A}utomorphes},
    SERIES = {2 vols},
 PUBLISHER = {Secr\'etariat math\'ematique, 11 rue Pierre Curie, Paris},
      YEAR = {1958},
     PAGES = {ii+214+ii+152 pp. (mimeographed)},
}


\bib{DM}{article}{
   author={Deligne, P.},
   author={Mumford, D.},
   title={The irreducibility of the space of curves of given genus},
   journal={Inst. Hautes \'Etudes Sci. Publ. Math.},
   number={36},
   date={1969},
   pages={75--109},
   issn={0073-8301},
}
		

		
\bib{DNar}{article} {
    AUTHOR = {Drezet, J.-M.}
    AUTHOR =  {Narasimhan, M. S.},
     TITLE = {Groupe de {P}icard des vari\'et\'es de modules de fibr\'es
              semi-stables sur les courbes alg\'ebriques},
   JOURNAL = {Invent. Math.},
  FJOURNAL = {Inventiones Mathematicae},
    VOLUME = {97},
      YEAR = {1989},
    NUMBER = {1},
     PAGES = {53--94},
}

	
	\bib{DSouza}{article}{
   author={D'Souza, Cyril},
   title={Compactification of generalised Jacobians},
   journal={Proc. Indian Acad. Sci. Sect. A Math. Sci.},
   volume={88},
   date={1979},
   number={5},
   pages={419--457},
   issn={0370-0089},
   review={\MR{569548}},
}

\bib{EG}{article}{
   author={Edidin, Dan},
   author={Graham, William},
   title={Riemann-Roch for equivariant Chow groups},
   journal={Duke Math. J.},
   volume={102},
   date={2000},
   number={3},
   pages={567--594},
}
\bib{Edidin}{article}{
   author={Edidin, Dan},
   title={Equivariant geometry and the cohomology of the moduli space of
   curves},
   conference={
      title={Handbook of moduli. Vol. I},
   },
   book={
      series={Adv. Lect. Math. (ALM)},
      volume={24},
      publisher={Int. Press, Somerville, MA},
   },
   date={2013},
   pages={259--292},
}
\bib{Fakh}{article}{
    AUTHOR = {Fakhruddin, Najmuddin},
     TITLE = {Chern classes of conformal blocks},
 BOOKTITLE = {Compact moduli spaces and vector bundles},
    SERIES = {Contemp. Math.},
    VOLUME = {564},
     PAGES = {145--176},
 PUBLISHER = {Amer. Math. Soc., Providence, RI},
      YEAR = {2012},


}
\bib{Faltings4}{article} {
    AUTHOR = {Faltings, Gerd},
     TITLE = {Stable {$G$}-bundles and projective connections},
   JOURNAL = {J. Algebraic Geom.},
  FJOURNAL = {Journal of Algebraic Geometry},
    VOLUME = {2},
      YEAR = {1993},
    NUMBER = {3},
     PAGES = {507--568},
      ISSN = {1056-3911},
}

\bib{Faltings}{article}{
  author={Faltings, Gerd},
  title={A proof for the Verlinde formula},
  journal={J. Algebraic Geom.},
  volume={3},
  date={1994},
  number={2},
  pages={347--374},
}
\bib{Faltings2}{article} {
    AUTHOR = {Faltings, Gerd},
     TITLE = {Moduli-stacks for bundles on semistable curves},
   JOURNAL = {Math. Ann.},
  FJOURNAL = {Mathematische Annalen},
    VOLUME = {304},
      YEAR = {1996},
    NUMBER = {3},
     PAGES = {489--515},
      ISSN = {0025-5831},
     CODEN = {MAANA},
}
		






\bib{FaltingsChai}{book}{
  author={Faltings, Gerd},
  author={Chai, Ching-Li},
  title={Degeneration of abelian varieties},
  series={Ergebnisse der Mathematik und ihrer Grenzgebiete (3) [Results in Mathematics and Related Areas (3)]},
  volume={22},
  note={With an appendix by David Mumford},
  publisher={Springer-Verlag},
  place={Berlin},
  date={1990},
  pages={xii+316},
}




\bib{Fulton}{book}{
   author={Fulton, William},
   title={Intersection theory},
   series={Ergebnisse der Mathematik und ihrer Grenzgebiete. 3. Folge. A
   Series of Modern Surveys in Mathematics [Results in Mathematics and
   Related Areas. 3rd Series. A Series of Modern Surveys in Mathematics]},
   volume={2},
   edition={2},
   publisher={Springer-Verlag, Berlin},
   date={1998},
   pages={xiv+470},
}

\bib{gie}{article} {
    AUTHOR = {Gieseker, D.},
     TITLE = {A degeneration of the moduli space of stable bundles},
   JOURNAL = {J. Differential Geom.},
  FJOURNAL = {Journal of Differential Geometry},
    VOLUME = {19},
      YEAR = {1984},
    NUMBER = {1},
     PAGES = {173--206},
}




\bib{Gomez}{article}{
   author={G{\'o}mez, Tom{\'a}s L.},
   title={Quantization of Hitchin's integrable system and the geometric
   Langlands conjecture},
   conference={
      title={Affine flag manifolds and principal bundles},
   },
   book={
      series={Trends Math.},
      publisher={Birkh\"auser/Springer Basel AG, Basel},
   },
   date={2010},
   pages={51--90},

}	

\bib{GW}{book} {
    AUTHOR = {G{\"o}rtz, Ulrich},
    AUTHOR = {Wedhorn, Torsten},
     TITLE = {Algebraic geometry {I}},
    SERIES = {Advanced Lectures in Mathematics},
      NOTE = {Schemes with examples and exercises},
 PUBLISHER = {Vieweg + Teubner, Wiesbaden},
      YEAR = {2010},
     PAGES = {viii+615},
}
	
	
\bib{EGA2}{article}{
   author={Grothendieck, A.},
   title={\'El\'ements de g\'eom\'etrie alg\'ebrique. II. \'Etude globale
   \'el\'ementaire de quelques classes de morphismes},
   journal={Inst. Hautes \'Etudes Sci. Publ. Math.},
   number={8},
   date={1961},
   pages={222},
}

\bib{Harty}{article} {
    AUTHOR = {Hartshorne, Robin},
     TITLE = {Stable reflexive sheaves},
   JOURNAL = {Math. Ann.},
  FJOURNAL = {Mathematische Annalen},
    VOLUME = {254},
      YEAR = {1980},
    NUMBER = {2},
     PAGES = {121--176},
}




	
\bib{Hartshorne}{book} {
    AUTHOR = {Hartshorne, Robin},
     TITLE = {Algebraic geometry},
      NOTE = {Graduate Texts in Mathematics, No. 52},
 PUBLISHER = {Springer-Verlag, New York-Heidelberg},
      YEAR = {1977},
     PAGES = {xvi+496},
      ISBN = {0-387-90244-9},
}
\bib{Xia}{article} {
    AUTHOR = {Xia, Huashi},
     TITLE = {Degenerations of moduli of stable bundles over algebraic
              curves},
   JOURNAL = {Compositio Math.},
  FJOURNAL = {Compositio Mathematica},
    VOLUME = {98},
      YEAR = {1995},
    NUMBER = {3},
     PAGES = {305--330},
}
\bib{HuKeel}{article}{
   author={Hu, Yi},
   author={Keel, Sean},
   title={Mori dream spaces and GIT},
   note={Dedicated to William Fulton on the occasion of his 60th birthday},
   journal={Michigan Math. J.},
   volume={48},
   date={2000},
   pages={331--348},
}
\bib{Igusa}{article}{
   author={Igusa, Jun-ichi},
   title={Fibre systems of Jacobian varieties},
   journal={Amer. J. Math.},
   volume={78},
   date={1956},
   pages={171--199},
   issn={0002-9327},
   review={\MR{0084848}},
}

\bib{Ishida}{article}{
   author={Ishida, Masa-Nori},
   title={Compactifications of a family of generalized Jacobian varieties},
   conference={
      title={Proceedings of the International Symposium on Algebraic
      Geometry},
      address={Kyoto Univ., Kyoto},
      date={1977},
   },
   book={
      publisher={Kinokuniya Book Store, Tokyo},
   },
   date={1978},
   pages={503--524},
   review={\MR{578869}},
}

\bib{Jake}{book}{
   author={Jacobson, Nathan},
   title={Basic algebra. I},
   publisher={W. H. Freeman and Co., San Francisco, Calif.},
   date={1974},
   pages={xvi+472},
   review={\MR{0356989}},
}
		
\bib{Kleiman}{incollection} {
    AUTHOR = {Kleiman, Steven L.},
     TITLE = {The {P}icard scheme},
 BOOKTITLE = {Fundamental algebraic geometry},
    SERIES = {Math. Surveys Monogr.},
    VOLUME = {123},
     PAGES = {235--321},
 PUBLISHER = {Amer. Math. Soc., Providence, RI},
      YEAR = {2005},
}


\bib{KNR}{article}{
  author={Kumar, Shrawan},
  author={Narasimhan, M. S.},
  author={Ramanathan, A.},
  title={Infinite Grassmannians and moduli spaces of $G$-bundles},
  journal={Math. Ann.},
  volume={300},
  date={1994},
  number={1},
  pages={41--75},
}





	

\bib{Gordon}{article}{
   author={Kung, Joseph P. S.},
   author={Rota, Gian-Carlo},
   title={The invariant theory of binary forms},
   journal={Bull. Amer. Math. Soc. (N.S.)},
   volume={10},
   date={1984},
   number={1},
   pages={27--85},
   issn={0273-0979},
   review={\MR{722856}},
   doi={10.1090/S0273-0979-1984-15188-7},
}

	
		


\bib{LaszloSorger}{article}{
  author={Laszlo, Yves},
  author={Sorger, Christoph},
  title={The line bundles on the moduli of parabolic $G$-bundles over curves and their sections},
  journal={Ann. Sci. \'Ecole Norm. Sup. (4)},
  volume={30},
  date={1997},
  number={4},
  pages={499--525},
}




\bib{Manon}{article}{
  author={Manon, Christopher},
  title={The Algebra of Conformal blocks},
  date={2009},
  note={arXiv:0910.0577},
}



\bib{Man1}{article}{
   author={Manon, Christopher},
   title={The algebra of $\op{SL}_3(\Bbb{C})$ conformal blocks},
   journal={Transform. Groups},
   volume={18},
   date={2013},
   number={4},
   pages={1165--1187},
}
		

\bib{Man2}{article}{
   author={Manon, Christopher},
   title={Coordinate rings for the moduli stack of $SL_2(\Bbb{C})$
   quasi-parabolic principal bundles on a curve and toric fiber products},
   journal={J. Algebra},
   volume={365},
   date={2012},
   pages={163--183},
}

\bib{MOP}{article}{
   author={Marian, Alina},
   author={Oprea, Dragos},
   author={Pandharipande, Rahul},
   title={The first Chern class of the Verlinde bundles},
   conference={
      title={String-Math 2012},
   },
   book={
      series={Proc. Sympos. Pure Math.},
      volume={90},
      publisher={Amer. Math. Soc., Providence, RI},
   },
   date={2015},
   pages={87--111},
}

\bib{MOPPZ}{article}{
   author={Marian, Alina},
   author={Oprea, Dragos},
   author={Pandharipande, Rahul},
   author={Pixton, Aaron},
   author={Zvonkine, D},
   title={The Chern character of the Verlinde bundle over the moduli space of stable curves},
   date={2013},
   pages={1--21},
   note={\url{https://arxiv.org/pdf/1311.3028v2.pdf}},
}
\bib{MumBook}{book} {
    AUTHOR = {Mumford, D.}
    AUTHOR =  {Fogarty, J.}
    AUTHOR =  {Kirwan, F.},
     TITLE = {Geometric invariant theory},
    SERIES = {Ergebnisse der Mathematik und ihrer Grenzgebiete (2) [Results
              in Mathematics and Related Areas (2)]},
    VOLUME = {34},
   EDITION = {Third},
 PUBLISHER = {Springer-Verlag, Berlin},
      YEAR = {1994},
     PAGES = {xiv+292},
}

\bib{NS}{article}{
   author={Nagaraj, D. S.},
   author={Seshadri, C. S.},
   title={Degenerations of the moduli spaces of vector bundles on curves. I},
   journal={Proc. Indian Acad. Sci. Math. Sci.},
   volume={107},
   date={1997},
   number={2},
   pages={101--137},
   issn={0253-4142},
}

\bib{NS2}{article} {
    AUTHOR = {Nagaraj, D. S.}
    AUTHOR = {Seshadri, C. S.}
     TITLE = {Degenerations of the moduli spaces of vector bundles on
              curves. {II}. {G}eneralized {G}ieseker moduli spaces},
   JOURNAL = {Proc. Indian Acad. Sci. Math. Sci.},
  FJOURNAL = {Indian Academy of Sciences. Proceedings. Mathematical
              Sciences},
    VOLUME = {109},
      YEAR = {1999},
    NUMBER = {2},
     PAGES = {165--201},
}

\bib{NRam}{article}{
    AUTHOR = {Narasimhan, M. S.},
    AUTHOR=  {Ramadas, T. R.},
     TITLE = {Factorisation of generalised theta functions. {I}},
   JOURNAL = {Invent. Math.},
  FJOURNAL = {Inventiones Mathematicae},
    VOLUME = {114},
      YEAR = {1993},
    NUMBER = {3},
     PAGES = {565--623},
}	
\bib{NR}{article}{
    AUTHOR = {Narasimhan, M. S}
    AUTHOR = {Ramanan, S.},
     TITLE = {Moduli of vector bundles on a compact {R}iemann surface},
   JOURNAL = {Ann. of Math. (2)},
  FJOURNAL = {Annals of Mathematics. Second Series},
    VOLUME = {89},
      YEAR = {1969},
     PAGES = {14--51},
}

\bib{NSesh}{article}{
    AUTHOR = {Narasimhan, M. S.}
    AUTHOR =  {Seshadri, C. S.},
     TITLE = {Stable and unitary vector bundles on a compact {R}iemann
              surface},
   JOURNAL = {Ann. of Math. (2)},
  FJOURNAL = {Annals of Mathematics. Second Series},
    VOLUME = {82},
      YEAR = {1965},
     PAGES = {540--567},
}
	

\bib{Newstead}{book}{
   author={Newstead, P. E.},
   title={Introduction to moduli problems and orbit spaces},
   series={Tata Institute of Fundamental Research Lectures on Mathematics
   and Physics},
   volume={51},
   publisher={Tata Institute of Fundamental Research, Bombay; by the Narosa
   Publishing House, New Delhi},
   date={1978},
   pages={vi+183},
}
	
\bib{OdaSeshadri}{article}{
   author={Oda, Tadao},
   author={Seshadri, C. S.},
   title={Compactifications of the generalized Jacobian variety},
   journal={Trans. Amer. Math. Soc.},
   volume={253},
   date={1979},
   pages={1--90},
   issn={0002-9947},
   review={\MR{536936}},
   doi={10.2307/1998186},
}
	





\bib{Osserman}{article}{
author ={Osserman, Brian}
title={Stability of vector bundles on curves and degenerations}
note={ arXiv:1401.0556}
}

\bib{Pandh}{article}{
   author={Pandharipande, Rahul},
   title={A compactification over $\overline {M}_g$ of the universal
   moduli space of slope-semistable vector bundles},
   journal={J. Amer. Math. Soc.},
   volume={9},
   date={1996},
   number={2},
   pages={425--471},
}



	\bib{Reid}{book}{
   author={Reid, Miles},
   title={Flatness, a brief overview},
   series={ \url{http://homepages.warwick.ac.uk/~masda/surf/more/grad.pdf}},
   }
\bib{SatakeCompactification}{article}{
  author={Satake, Ichir{\^o}},
  title={On representations and compactifications of symmetric Riemannian spaces},
  journal={Ann. of Math. (2)},
  volume={71},
  date={1960},
  pages={77--110},
}

\bib{schmitt}{article} {
    AUTHOR = {Schmitt, Alexander H. W.},
     TITLE = {On the modular interpretation of the {N}agaraj-{S}eshadri
              locus},
   JOURNAL = {J. Reine Angew. Math.},
  FJOURNAL = {Journal f\"ur die Reine und Angewandte Mathematik. [Crelle's
              Journal]},
    VOLUME = {670},
      YEAR = {2012},
     PAGES = {145--172},
      ISSN = {0075-4102},
}

\bib{Serre}{article}{
   author={Serre, J.-P.},
   title={Modules projectifs et espaces fibr\'es \`a fibre vectorielle},
   language={French},
   conference={
      title={S\'eminaire P. Dubreil, M.-L. Dubreil-Jacotin et C. Pisot,
      1957/58, Fasc. 2, Expos\'e 23},
   },
   book={
      publisher={Secr\'etariat math\'ematique, Paris},
   },
   date={1958},
   pages={18},
}

\bib{Seshadri}{book}{
   author={Seshadri, C. S.},
   title={Fibr\'es vectoriels sur les courbes alg\'ebriques},
   language={French},
   series={Ast\'erisque},
   volume={96},
   note={Notes written by J.-M. Drezet from a course at the \'Ecole Normale
   Sup\'erieure, June 1980},
   publisher={Soci\'et\'e Math\'ematique de France, Paris},
   date={1982},
   pages={209},
}

\bib{simpson}{article} {
    AUTHOR = {Simpson, Carlos T.},
     TITLE = {Moduli of representations of the fundamental group of a smooth
              projective variety. {I}},
   JOURNAL = {Inst. Hautes \'Etudes Sci. Publ. Math.},
  FJOURNAL = {Institut des Hautes \'Etudes Scientifiques. Publications
              Math\'ematiques},
    NUMBER = {79},
      YEAR = {1994},
     PAGES = {47--129},
      ISSN = {0073-8301},
     CODEN = {PMIHA6},
}
\bib{SunA}{article}{
   author={Sun, Xiaotao},
   title={Degeneration of ${\op{SL}}(n)$-bundles on a reducible curve},

      journal={Algebraic geometry in East Asia, World Sci. Publ., River Edge, NJ},
      address={Kyoto},
      date={2001},




   date={2002},
   pages={229--243},
}

\bib{Su3}{article}{
    AUTHOR = {Sun, Xiaotao},
     TITLE = {Moduli spaces of {${\rm SL}(r)$}-bundles on singular
              irreducible curves},
   JOURNAL = {Asian J. Math.},
  FJOURNAL = {The Asian Journal of Mathematics},
    VOLUME = {7},
      YEAR = {2003},
    NUMBER = {4},
     PAGES = {609--625},
}




\bib{TY}{book} {
    AUTHOR = {Tauvel, Patrice}
    AUTHOR =  {Yu, Rupert W. T.},
     TITLE = {Lie algebras and algebraic groups},
    SERIES = {Springer Monographs in Mathematics},
 PUBLISHER = {Springer-Verlag, Berlin},
      YEAR = {2005},
}

\bib{Teleman}{article}{
   author={Teleman, Constantin},
   title={Borel-Weil-Bott theory on the moduli stack of $G$-bundles over a
   curve},
   journal={Invent. Math.},
   volume={134},
   date={1998},
   number={1},
   pages={1--57},
}

\bib{TUY}{article}{
  author={Tsuchiya, Akihiro},
  author={Ueno, Kenji},
  author={Yamada, Yasuhiko},
  title={Conformal field theory on universal family of stable curves with gauge symmetries},
  conference={ title={Integrable systems in quantum field theory and statistical mechanics}, },
  book={ series={Adv. Stud. Pure Math.}, volume={19}, publisher={Academic Press}, place={Boston, MA}, },
  date={1989},
  pages={459--566},
}

\bib{Ueno}{book}{
  author={Ueno, Kenji},
  title={Conformal field theory with gauge symmetry},
  series={Fields Institute Monographs},
  volume={24},
  publisher={American Mathematical Society},
  place={Providence, RI},
  date={2008},
  pages={viii+168},
}

\bib{Wang}{article}{
  author={Wang, Jonathan},
  title={The moduli stack of G-bundles},
  note={\url{http://arxiv.org/abs/1104.4828}},
}

\bib{Weyman}{book} {
    AUTHOR = {Weyman, Jerzy},
     TITLE = {Cohomology of vector bundles and syzygies},
    SERIES = {Cambridge Tracts in Mathematics},
    VOLUME = {149},
 PUBLISHER = {Cambridge University Press, Cambridge},
      YEAR = {2003},
     PAGES = {xiv+371},
}




\end{biblist}
\end{bibdiv}
\end{document}